\documentclass[a4paper,english,11pt,reqno]{amsart}

\usepackage[utf8]{inputenc}
\usepackage[T1]{fontenc}
\usepackage{babel}
\usepackage[babel]{csquotes}

\usepackage{appendix}

\usepackage{amsmath}
\usepackage{amsfonts}
\usepackage{amsthm}
\usepackage{bbm}
\usepackage{amssymb}
\usepackage[top=3.4cm, bottom=3.6cm, left=3.6 cm, right=3.6 cm]{geometry}
\usepackage{mathrsfs}
\usepackage[pdftex, pdfstartview=FitW,colorlinks=true,linkcolor=blue,%
urlcolor=blue,citecolor=blue]{hyperref}
\usepackage{stmaryrd}
\usepackage{cleveref}
\usepackage{mathtools}
\usepackage{soul}

\usepackage{url}
\usepackage{xcolor}
\usepackage{enumitem}

\newtheorem{th.}{Theorem}[section]
\newtheorem{thm}[th.]{Theorem}

\newtheorem{lemma}[th.]{Lemma}

\newtheorem{remark}[th.]{Remark}

\newtheorem{proposition}[th.]{Proposition}
\newtheorem{corollary}[th.]{Corollary}

\theoremstyle{definition}
\newtheorem{definition}[th.]{Definition}

\numberwithin{equation}{section}

\theoremstyle{definition}

\numberwithin{equation}{section}

\newcommand{\N}{\mathbb{N}}
\newcommand{\Z}{\mathbb{Z}}

\newcommand{\R}{\mathbb{R}}

\newcommand{\dd}{\,\mathrm{d}}
\newcommand{\norm}[1]{\lVert#1\rVert}

\newcommand{\bomega}{\boldsymbol{\omega}}
\newcommand{\bs}{\boldsymbol{s}}
\newcommand{\br}{\boldsymbol{r}}

\newcommand{\cS}{\mathcal{S}}

\newcommand{\integer}[1]{\lfloor #1 \rfloor}
\newcommand{\uphi}{\underline{\phi}}
\newcommand{\abse}[1]{\left\lvert#1\right\rvert}
\newcommand{\ttb}{\mathtt{b}}
\newcommand{\ttr}{\mathtt{r}}
\newcommand{\bp}{\boldsymbol{p}}

\newcommand{\bt}{\boldsymbol{t}}

\newcommand{\bv}{\boldsymbol{v}}

\newcommand{\bttb}{\boldsymbol{\ttb}}
\newcommand{\bttr}{\boldsymbol{\ttr}}
\newcommand{\kg}{\mathfrak{g}}
\newcommand{\sV}{\mathscr{V}}
\newcommand{\dbracket}[1]{\llbracket #1 \rrbracket}
\newcommand{\eps}{\epsilon}
\newcommand{\cE}{\mathcal{E}}
\newcommand{\cF}{\mathcal{F}}
\newcommand{\cC}{\mathcal{C}}

\DeclareMathOperator{\diag}{diag}
\DeclareMathOperator{\SL}{SL}
\DeclareMathOperator{\inj}{inj}
\DeclareMathOperator{\supp}{supp}
\DeclareMathOperator{\SO}{SO}
\DeclareMathOperator{\Aff}{Aff}
\DeclareMathOperator{\Id}{Id}

\DeclareMathOperator{\GL}{GL}
\DeclareMathOperator{\dist}{dist}
\DeclareMathOperator{\Gr}{Gr}
\DeclareMathOperator{\MNC}{MNC}
\DeclareMathOperator{\Lie}{Lie}
\DeclareMathOperator{\Ad}{Ad}
\DeclareMathOperator{\leb}{Leb}

\title[Effective equidistribution]{On effective equidistribution for product-type self-affine measures}

\author{Han Zhang}
\address{School of Mathematical Sciences, Shenzhen University, 518060,  Shenzhen, China
}
\email{hzmath@szu.edu.cn}
\thanks{H.Z. is supported by  the National Natural Science Foundation of China (No. 12501250)}
\date{}

\begin{document}

\begin{abstract}
We establish an effective equidistribution theorem on homogeneous spaces for expanding translates of product-type self-affine measures seen on unipotent orbits along any sequence of diagonal elements drifting away from walls. In particular, we do not require the  diagonal elements to lie in a fixed proper and closed subcone.

The main new ingredients in the proof of our  equidistribution theorem include an  adaptation of the methods in \cite{BHZ25} to the setting of product-type self-affine measures, as well as  a weighted effective recurrence theorem for intertwined  associated random walks. The latter involves decomposing the  random walks according to suitable stopping times and applications of intertwined Benoist-Quint height functions.
\end{abstract}

\maketitle

\tableofcontents
\setcounter{tocdepth}{1}
\large

\section{Introduction}

Let $d\geq 2$ be an integer and $\Aff(\R^d)$ be the collection of all affine maps from $\R^d$ to $\R^d$ whose linear part lies in $\GL_d(\R)$. We consider a self-affine iterated function system (IFS) $\lambda$, which is a probability measure on $\Aff(\R^d)$ such that $\lambda$ is contractive with finite support. By \cite{BP92}, there exists a unique probability measure $\sigma$ on $\R^d$ satisfying $\lambda*\sigma=\sigma$. The measure $\sigma$ is referred to as the \textit{self-affine measure} associated to $\lambda$. If each $\phi\in \supp \lambda$ is a similitude, then $\lambda$ is called a self-similar IFS and $\sigma$ is the associated self-similar measure. In this paper, we consider a special class of self-affine measures defined as follows:

\begin{definition}\label{defn: product type}
   A self-affine measure $\sigma$ on $\R^d$ is said to be \textit{product-type} if $\sigma=\sigma_1\otimes \cdots \otimes \sigma_d$, where each $\sigma_i$ ($1\leq i\leq d$) is a self-similar measure on $\R$ whose associated IFS $\lambda_i$ is contractive, has finite support and affine maps in $\lambda_i$ have no common fixed point.
\end{definition}

\begin{remark}\label{rmk: not self-similar}
    In general a product-type self-affine measure $\sigma$ is not self-similar. For example, the product of the middle thirds Cantor measure and the middle fifths Cantor measure is not self-similar. $\sigma$ is self-similar if one can find an associated IFS such that all contraction ratios agree.
\end{remark}

\medskip

In this paper, we investigate expanding translates by arbitrary diagonal flows  of a product-type self-affine measure supported on a unipotent orbit on the homogeneous space  $\SL_{d+1}(\R)/\SL_{d+1}(\Z)$. Denote $G=\SL_{d+1}(\R)$, $\Gamma=\SL_{d+1}(\Z)$  and $X=G/\Gamma$. Given $\bt=(t_1,\cdots,t_d)\in \R^d_{>0}$, we define the multi-parameter diagonal flow
\begin{align}\label{eq: defn of multi parameter}
    a(\bt)=\diag(t_1^{\tfrac{1}{d+1}},\cdots,t_d^{\tfrac{1}{d+1}},\prod_{i=1}^d t_i^{-\tfrac{1}{d+1}}), 
\end{align}
and $A=\{a(\bt):\bt\in \R^d_{>0}\}$. When $\bt=(t,\cdots,t)$, we write $a(t)=a(\bt)$ for simplicity. Denote by $u:\R^d\to G$ the upper triangular matrix given by 
\[ u(\bs)=\begin{pmatrix}
     1 & & & s_1\\
     & \ddots & & \vdots\\
     & & 1 & s_d\\
     & & & 1
\end{pmatrix}, \]
and $U=\{u(\bs): \bs\in \R^d\}$. For $l\in \N$, we denote by $B^{\infty}_{\infty,l}(X)$ the collection of smooth functions from $X$ to $\R$ with bounded derivative up to order $l$, and let $\cS_{\infty,l}(\cdot)$ be the associated $C^l$-norm. Let $m_X$ be the Haar probability measure on $X$, that is, $m_X$ is the unique $G$-invariant probability measure on $X$.

\medskip

\begin{thm}\label{thm: weighted effective equid}
Let $\sigma$ be a product-type self-affine measure on $\R^d$ and $l=\lceil \frac{1}{2}\dim \SO(d+1)\rceil$. Then there exist $C,c>0$ depending only on $\sigma$ such that for every $x\in X$, $\bt\in \R^d_{>1}$ and $f\in B^\infty_{\infty, l}(X)$, we have
\begin{align}\label{eq: effective conv}
    \abse{\int_{\R^d} f\bigl(a(\bt) u(\bs) x\bigr) \dd \sigma(\bs) - m_{X}(f) } \leq C \cS_{\infty, l}(f) \inj(x)^{-1} \integer{\bt}^{-c},
\end{align}
where $\integer{\bt}=\min_{i} t_i$, and $\inj(x)$ is the injectivity radius of $x$.
\end{thm}
\medskip

Let $A^+=\{a(\bt):\bt\in \R^d_{> 1}\}$.  $A^+$ is referred to as the expanding cone of $U$ (see \cite{Shi21}). Following \cite{KW08}, we say that a sequence $\{a(\bt_n)\}\subset A^+$ \textit{drifts away from walls} (of the expanding cone) if  $\integer{\bt_n}\to \infty$ as $n\to\infty$. When $\bt=(t,\cdots,t)$, the convergence in \eqref{eq: effective conv} translated by $a(t)$ as $t\to \infty$ is referred to as the \textit{unweighted} equidistribution. Theorem \ref{thm: weighted effective equid} asserts that for any sequence $\{a(\bt_n)\}$ in $A^+$, the expanding translates by $a(\bt_n)$ of $\sigma$ seen on a unipotent orbit get equidistributed in $X$ with rate depending only on the speed of $\{a(\bt_n)\}$ drifting away from walls. The error term $\inj(x)^{-1}$ indicates that it will take long time to equidistribute if the base point $x$ is deep in the cusp.

The important feature in Theorem \ref{thm: weighted effective equid} is that for \eqref{eq: effective conv} to hold, a sequence $\{\bt_n\}\subset \R^{d}_{>1}$ is allowed to escape to infinity at arbitrarily different rates in different coordinates. This is made possible by the following  non-divergence estimate (or rather its equivalent form of Theorem \ref{thm: weighted effective recurrence}):

\begin{thm}\label{thm: nondiv}
Let $\sigma$ be as in Theorem \ref{thm: weighted effective equid}. There exist constants $C,c>0$ depending only on $\sigma$ and $G$ such that the following holds. For every $\bt\in \R^d_{> 1}$, any $x\in X$ and $\rho\lll_{\sigma} 1$, we have
\[ \sigma\{ \bs: \inj(a(\bt)u(\bs)x)<\rho \}< C\rho^c(1+\integer{\bt}^{-c}\inj(x)^{-C}).  \]

\end{thm}

\medskip

Theorem \ref{thm: nondiv} (and Theorem \ref{thm: weighted effective recurrence}) is of independent interest. If $\sigma$ is the usual Lebesgue measure on $\R^d$ restricted to a bounded open box, then Theorem \ref{thm: nondiv} \textit{fully recovers} the non-divergence estimate in \cite[Corollary 3.4]{KM12} allowing $\bt$ to vary in $\R^d_{> 1}$. Moreover, we indeed offer an alternative proof\footnote{The proof \cite[Corollary 3.4]{KM12} crucially depends on the fact that Lebesgue measure is friendly, which is not the case in our setting.} of \cite[Corollary 3.4]{KM12} in the case of Lebesgue measure. This relies on applications of intertwined Benoist-Quint height functions, which also works for the broader class of product-type self-affine measures that are not necessarily friendly (see Remark \ref{rmk: not friendly}). A similar estimate appears in \cite[Proposition 3.1]{AG26}, where the authors assume additionally that $\sigma$ is a product-type self-similar measure and  $\bt$ lies in a closed proper subcone of $\R^d_{>1}$.

\bigskip

The study of equidistribution for expanding translates of unipotent orbits is a central topic in homogeneous dynamics. The landmark has been the celebrated work of Ratner \cite{Ratner91} settling Raghunathan's topological conjecture on unipotent orbits. Recently, there are many significant works effectivizing Ratner's equidistribution theorem in various settings, see e.g. \cite{Strombergsson15,Kim24, LM23,LMW26,LMWY25}. Specified to our setting in the qualitative aspect, when $\bt=(t,\cdots,t)$ and $\sigma$ is the Lebesgue measure on $\R^d$ restricted to a bounded open set, the qualitative equidistribution in Theorem \ref{thm: weighted effective equid} follows from the work of Shah \cite{Shah96} extending the results of Ratner.  Khalil, Luethi and Weiss \cite{KLW26} obtained the qualitative convergence as in \eqref{eq: effective conv} of translates by any sequence $\{a_n\}\subset A^+$ drifting away from walls  for certain fractal carpets. This class of fractal carpets consists of self-similar measures which are rational with equal contraction ratios and have no rotation part. Hence, the collection of product-type self-affine measures considered in Theorem \ref{thm: weighted effective equid} has a nonempty intersection with those in the work of Khalil-Luethi-Weiss \cite{KLW26}.

For the quantitative aspect, Kleinbock and Margulis \cite{KM96} proved the effective equidistribution theorem of expanding translates of horospherical subgroups in general homogeneous spaces. In the setting of Theorem \ref{thm: weighted effective equid}, $U$ is not necessarily the horospherical subgroup of the translation sequence. When $\sigma$ is the Lebesgue measure on $\R^d$ restricted to a bounded open set, Theorem \ref{thm: weighted effective equid} is due to a later work \cite{KM12} of Kleinbock and Margulis. For the translation of fractal measures, Khalil and Luethi \cite{KL23} obtained the unweighted effective equidistribution for a special class of self-similar measures on $\R^d$ which are irreducible, rational, large dimensional with open set condition. Subsequent works of B\'enard, He and Zhang removed  additional assumptions of Khalil-Luethi and established unweighted effective convergence \eqref{eq: effective conv} for any irreducible and contractive in average self-similar measures with finite exponential moments for $d=1$ in \cite{BHZ26}, then for any $d\geq 1$ in \cite{BHZ25}. See also the related works \cite{Yu21,DJ24} obtaining similar effective equidistribution results for a class of fractal measures with large dimension in some sense.

Very recently, Aggarwal and Gorodnik \cite{AG26} established the effective convergence \eqref{eq: effective conv} for a certain class of product-type self-similar measures on $\R^d$  for any sequence $\{a(\bt_n)\}$ drifting away from walls but restricted in a fixed \textit{ proper and closed} subcone of $A^+$. More precisely, in \cite{AG26}, the authors consider the self-similar measure on $\R^d$ which is a product of self-similar measures on $\R$ all sharing a common contraction ratio, and assume that the growth rates among coordinates of the sequence $\{\bt_n\}\subset \R^d_{>1}$ are comparable. Comparing with their work, the class of product-type self-affine measures we consider in Theorem \ref{thm: weighted effective equid} contains those in \cite{AG26}, and we allow the translating diagonal elements to vary in the whole $A^+$. 

\bigskip

\noindent{\bf On the proof of Theorem \ref{thm: weighted effective equid}}. We follow the general strategy in \cite{KM12} which allows us to deduce the effective convergence in \eqref{eq: effective conv}  from its unweighted counterpart. The main inputs of this deduction include an unweighted effective equidistribution translated by $a(t)$ as $t\to \infty$, and a suitable non-divergence estimate for certain random walks associated to product-type self-affine measures. 

Since we consider a product-type self-affine measure $\sigma$, which is not self-similar in general (see Remark \ref{rmk: not self-similar}), the unweighted effective equidistribution theorem we need does not follow directly from \cite[Theorem 1.2]{BHZ25}. The main obstruction stems from the fact that the contraction ratios of the affine maps in a self-affine IFS may differ depending on the direction. Fortunately, the product structure allows us to make use of   intertwined random walks. These random walks, when chosen appropriately  (see \eqref{eq: overline mu n1}), correspond to an IFS whose typical affine maps have roughly uniform contraction ratios across all directions. Hence, the resulting dynamics on $X$ driven by these random walks share many properties with the one induced by self-similar measures in \cite{BHZ25}. This enables us to modify the methods in \cite{BHZ25} to prove the desired unweighted equidistribution theorem (Theorem \ref{thm: unweighted effective equid}).

Given $a(\bt)\in A^+$ and $x\in X$. Thanks to the product structure and stationarity of $\sigma$ under convolution of certain random walks, we can transform the orbit $\{a(\bt)u(\bs)x: \bs\in \supp \sigma\}$ in \eqref{eq: effective conv} into its unweighted counterpart. Since $a(\bt)$ is allowed to vary in the whole $A^+$ instead of being restricted to a proper and closed subcone of $A^+$, in this step we consider random walks pertaining to suitable stopping times determined by $\bt$ (see \eqref{eq: lambda_i,k eps and mu_i,k eps}). The outcome of this transformation is a class of unweighted translations $\{a(t)u(\bs)y:\bs\in \supp \sigma\}$ for some $t>1$, where $y$ varies in the support of the stopping time random walks starting at $x$. The effective convergence in \eqref{eq: effective conv} will follow if outside an exponentially small exceptional set, $y$ lies in a large compact subset of $X$. Here is the place where an effective recurrence estimate (Theorem \ref{thm: weighted effective recurrence}) of stopping time random walks comes in.  

Theorem \ref{thm: weighted effective recurrence} is of independent interest, and the proof of it does not follow from standard methods developed in the literature. On one hand, the measure $\sigma$ is generally not friendly in the sense of Kleinbock-Lindenstrauss-Weiss \cite{KLW04} (for details, see Remark \ref{rmk: not friendly}). This prevents us from applying arguments in \cite{KLW04} relying on $(C,\alpha)$-good properties of friendly measures in this setting. On the other hand, since we allow  $a(\bt)$ to vary in the whole $A^+$, it could happen that one coordinate of $\bt$ is not comparable to the rest of the coordinates. Hence, an  application of the height function technique as in \cite{AG26}, which crucially relies on the assumption that the growth speed among all coordinates of $\bt$ are comparable, does not yield the desired recurrence. Moreover, the fact that the IFS associated to $\sigma$ might have different contraction ratios in different directions  introduce some complexities. Despite these difficulties, we succeed in proving Theorem \ref{thm: weighted effective recurrence}. This is based on the product structure of $\sigma$, which allows us to consider intertwined random walks induced by the IFS associated to $\sigma$ (see Section \ref{subsec: associated rw}). We perform a stopping time measure decomposition (Lemma \ref{lemma: rw stopping time decomposition}) for these random walks, and then make use of certain intertwined Benoist-Quint height functions to show that suitable convolutions of these random walks exhibit an effective recurrence property (Proposition \ref{prop: effective recurrence of suitable convolution}). These effective recurrence estimates can be combined in a suitable way to show the desired non-divergence estimate in Theorem \ref{thm: weighted effective recurrence}.

\bigskip

\noindent{\bf Applications to Diophantine approximations}. Given a fractal measure on $\R^d$, it is intriguing to explore Diophantine properties of typical points with respect to this fractal measure. This is a research topic motivated by Mahler's question \cite{Mahler84} asking for Diophantine properties of points in the middle thirds Cantor set. We discuss here some applications of Theorem \ref{thm: weighted effective equid}, its double equidistribution strengthening (Proposition \ref{prop: double equid}) to Diophantine approximations. The deduction of these results in Diophantine approximations from dynamics are now well-known (see e.g. \cite{KSY22,BHZ25,DJ24,Yu21,AG26}). Therefore, we refer the interested reader to related papers on these applications. In the following, we let $\sigma$ be a product-type self-affine measure on $\R^d$. 

\begin{itemize}
    \item[1.] A weighted Khintchine's theorem under certain assumptions on approximation functions: Let $\psi_i:\N\to \R_{\geq 0}$ be a non-increasing function for $1\leq i\leq d$. Denote $\Psi=(\psi_1,\cdots,\psi_d)$. A real vector $\bs\in \R^d$ is called $\Psi$-approximable if there exist infinitely many $(\bp,q)\in \Z^d\times \N$ such that 
    \begin{align}\label{eq: Diophantine}
        \abse{qs_i-p_i}<\psi_i(q), \,\, \forall \, 1\leq i\leq d. 
    \end{align}
Let $W(\Psi)$ be the collection of all $\Psi$-approximable real vectors. With Theorem \ref{thm: weighted effective equid} and Proposition \ref{prop: double equid}, one can follow the argument as in the proof of \cite[Theorem 1.5]{AG26} to obtain the following result: Suppose that $\Psi$ is $\alpha$-balanced and $\beta$-proportional for some $\alpha\in (0,1)$ and $\beta>1$ (see \cite[Definition 1.2]{AG26}), then 
\begin{align*}
    \sigma(W(\Psi))=\begin{cases}
        0 & \text{ if } \sum_{q}\psi_1(q)\cdots \psi_d(q)<\infty,\\
        1 & \text{ if } \sum_{q}\psi_1(q)\cdots \psi_d(q)=\infty.
    \end{cases}
\end{align*}
Moreover, an asymptotic Schmidt-type counting of solutions to \eqref{eq: Diophantine} with bounded denominator $q$ holds in the divergence case (see \cite[Theorem 1.5]{AG26}).

\medskip

\item[2.] A zero-one law in uniform Diophantine approximations: Let $\bomega=(\omega_i)\in (0,1)^d$ be such that $\sum_i \omega_i=1$. Let $\psi:\N\to \R_{>0}$ be a non-increasing function satisfying  that the map $t\mapsto t\psi(t)$ is increasing and $t\psi(t)<1$. A real vector $\bs \in \R^d$ is called $\psi_{\bomega}$-Dirichlet if the Diophantine inequalities
\[ \norm{q\bs-\bp}_{\bomega}<\psi(t),\,\, 0<q<t \]
has a solution $(\bp,q)\in \Z^d\times \N$ for all sufficiently large $t$. Here $\norm{\bs}_{\bomega}=\max_{i}\{|s_i|^{1/\omega_i}\}$ for $\bs\in \R^d$. Denote by $\mathrm{DI}_{\bomega}(\psi)$ the collection of all $\psi_{\bomega}$-Dirichlet real vectors. Theorem \ref{thm: weighted effective equid} and Proposition \ref{prop: double equid} enable one to follow the strategy developed in \cite{KSY22} (see also \cite{StrombergssonYu2026}) to deduce the following result: Let $F_{\psi}(t)=1-t\psi(t)$ for $t>0$. Then
\begin{align*}
    \sigma(\mathrm{DI}_{\bomega}(\psi))=\begin{cases}
        0 & \text{ if } \sum_{1}^{\infty} q^{-1}F_{\psi}(q)^{\chi_d}\log^{\lambda_d}\left(\frac{1}{F_{\psi}(q)} \right)=\infty,\\
        1 & \text{ if } \sum_{1}^{\infty} q^{-1}F_{\psi}(q)^{\chi_d}\log^{\lambda_d}\left(\frac{1}{F_{\psi}(q)} \right)<\infty.
    \end{cases}
\end{align*}
Here $\chi_d=\tfrac{d^2+3d-2}{2}$, $\lambda_d=\tfrac{d(d+1)}{2}$.

\medskip

\end{itemize}

\bigskip

\noindent{\bf Structure of the paper.} 
In Section \ref{sec: preliminary}, we collect  some useful preliminaries on self-similar measures, the stopping time decomposition and intertwined  random walks on homogeneous spaces induced by a product-type self-affine IFS. In Section \ref{sec: height function}, we recall the construction of Benoist-Quint height functions and prove the desired contraction hypothesis of two different averaging operators determined by certain stopping time random walks and  the random walks in \eqref{eq: overline mu n1}, respectively. In Section \ref{sec: effective recurrence}, we prove the effective recurrence theorem (Theorem \ref{thm: weighted effective recurrence}) for suitable convolutions among various stopping time random walks. This is done by effective recurrence properties of intertwined random walks guaranteed by the Benoist-Quint height function technique. In Section \ref{sec: proof of weighted equid}, modulo the unweighted counterpart (Theorem \ref{thm: unweighted effective equid}) and following the strategy in \cite{KM12}, we prove Theorem \ref{thm: weighted effective equid} with the help of Theorem \ref{thm: weighted effective recurrence}. In Section \ref{sec: unweighted equid}, we show how to adapt the methods developed in \cite{BHZ25} in the current setting and finish the proof of Theorem \ref{thm: unweighted effective equid}. In Appendix \ref{sec: double equid} , we give a self-contained proof of a double effective equidistribution result.

\medskip

\noindent{\bf Notations.} For any $N\in \N$, we denote $\dbracket{1,N}=[1,N]\cap \Z$. The Landau notation $O(\cdot)$ and the Vinogradov symbol $\ll$ are used. Given $a,b>0$, we write $a\simeq b$ if $a\ll b\ll a$. The notation $a\lll b$ means that there exists $\eps>0$ small such that $a\leq \eps b$. Given a parameter $\kappa>0$, the symbol $\lll_{\kappa}$ means that $\eps$ depends on $\kappa$. Without specifying subscripts, we allow $\eps$ to depend on the underlying self-affine measure.

\medskip

\noindent{\bf Acknowledgements.}
We thank Weikun He for helpful discussions on this work.

\medskip

\noindent{\bf AI disclosure.} This work was initiated in May 2026, and was produced by the human author without the use of AI. After its completion, AI was used to proofread and polish the language.

\bigskip

\section{Preliminaries}\label{sec: preliminary}

In this section, we collect some useful facts about self-similar measures and their induced random walks on the homogeneous space $X$. 

\subsection{Iterated function system} Let $k\in \N$ and $\Aff(\R^k)$ be the set of all affine maps on $\R^k$ whose linear part lies in $\GL_d(\R)$. That is, any $\phi\in \Aff(\R^k)$ can be written as 
\[ \phi(\bs)=M_{\phi}\bs+\ttb_{\phi},  \, \forall \, \bs\in \R^k, \]
where $M_{\phi}\in \GL(\R^k)$ and $\ttb_{\phi}\in \R^k$. Let $\lambda$ be a finitely supported probability measure on $\Aff(V)$. In particular, if for any $\phi\in \supp \lambda$, the matrix norm $\norm{M_{\phi}}<1$, then $\lambda$ is said to be contractive. In this case, $\lambda$ is called a \textit{self-affine} IFS on $V$. For a contractive self-affine IFS $\lambda$,  by \cite[Theorem 2.5, Corollary 2.7]{BP92}, there exists a unique probability measure $\sigma$ on $\R^k$ such that $\lambda*\sigma=\sigma$. This $\sigma$ is referred to as a \textit{self-affine} measure  associated to $\lambda$. If for any $\phi\in \supp \lambda$, $M_{\phi}=\ttr_{\phi}O_{\phi}$ for $\ttr_{\phi}\in (0,1)$ and $O_{\phi}\in \SO_k(\R)$, then $\lambda$ is called a self-similar IFS, and $\sigma$ is a self-similar measure. Throughout this paper, we always assume that every self-similar IFS on $\R$ is contractive and has no common fixed point.
\medskip

Recall that in Definition \ref{defn: product type}, a self-affine IFS $\lambda$ on $\Aff(\R^d)$ is \textit{product-type} if $\lambda=\otimes_1^d\lambda_i$, where for each $1\leq i\leq d$, $\lambda_i$ is a contractive and finitely supported self-similar IFS on $\R$ without common fixed point for all affine maps in $\supp \lambda_i$. Note that $\supp \lambda=\otimes_i\supp \lambda_i$. The self-affine measure $\sigma$ associated to such $\lambda$ is a product-type self-affine measure. In particular, $\sigma=\otimes_1^d \sigma_i$, where $\sigma_i$ is the self-similar measure associated to $\lambda_i$ for $1\leq i\leq d$.

\medskip

\subsection{The Lyapunov exponent}

Let $\lambda$ be a contractive self-similar IFS on $\Aff(\R)$. For any $\phi\in \Aff(\R)$, we write $\phi=\ttr_{\phi}\Id+\ttb_{\phi}$ for some $\ttr_{\phi}>0$ and $\ttb_{\phi}\in \R$. The Lyapunov exponent is given by 
\[ \ell_{\lambda}=-\int \log\ttr_{\phi} \dd \lambda(\phi)>0. \]

We collect the following useful fact concerning the generic value of $\ttr_{\phi}$ when $\phi$ is distributed according to $\lambda^{*n}$ for $n$ sufficiently large, where $\lambda^{*n}$ denotes the convolution of $\lambda$ with itself $n$ times.

\begin{proposition}\label{prop: LDP}[Large deviation principle]
Given any $\epsilon>0$, there exists $\delta=\delta(\lambda,\epsilon)>0$ such that for all $n\ggg_{\epsilon,\lambda} 1$, we have
\[ \lambda^{*n}\{\phi: |\log \ttr_{\phi}+n\ell_{\lambda}|>n\epsilon\}<e^{-\delta n}. \]
\end{proposition}

\medskip

\subsection{The stopping time}

Let $\lambda$ be a finitely supported and contractive self-similar IFS on $\Aff(\R)$ with $\supp \lambda=\Lambda$.
\begin{definition}
    Given $\epsilon>0$,  we define the stopping time function $\tau_{\epsilon}:\Lambda^{\N}\to \N\cup \{0\}$ by
\[ \tau_{\epsilon}(\uphi)=\begin{cases}
    \inf\{n\in \N: \ttr_{\uphi_n}<\epsilon\} & \text{ if } \eps\in (0,1],\\
    0 & \text{ if } \eps>1.
\end{cases} \]
Here
\[ \uphi_n=\phi_1\circ\cdots\circ \phi_n \, \text{ for } \uphi=(\phi_1,\phi_2,\cdots)\in \Lambda^{\N}. \]
By convention, whenever $\eps>1$, we set $\uphi_0=\Id$ and $\ttr_{\uphi_0}=1$ for any $\uphi\in \Lambda^{\N}$.
\end{definition}
\medskip

Let $T:\Lambda^{\N}\to \Lambda^{\N}$ be the left shift operator defined by 
\[ T((\phi_{1},\phi_{2},\cdots))=(\phi_{2},\phi_{3},\cdots). \]
By convention, $T^0=\Id$.

\begin{lemma}\label{lemma: basic prop of stopping time}
Let $r_{min}=\min_{\phi\in \Lambda}\ttr_{\phi}$. Given $0<\epsilon_2\leq \epsilon_1$,  $\uphi\in \Lambda^{\N}$, let $\tau_{\epsilon_1}(\uphi)=n$, $\tau_{\eps_2}(\uphi)=m+n$ for some $m,n\in \Z_{\geq 0}$. Then the following hold:
\begin{itemize}
\item[(1)] If $\eps_1\leq 1$, then $r_{min}\eps_1\leq \ttr_{\uphi_n}<\epsilon_1\leq \ttr_{\uphi_{n-1}}<r_{\min}^{-1}\eps_1$.

\item[(2)] $\tau_{\epsilon_2}(\uphi)=m+n$ if and only if   $m=\tau_{\epsilon_2/\ttr_{\uphi_n}}(T^n(\uphi))$.

\end{itemize}
    
\end{lemma}

\begin{proof}
For (1), we only need to show $r_{min}\epsilon_1\leq \ttr_{\uphi_n}$ and $\ttr_{\uphi_{n-1}}<r_{min}^{-1}\eps_1$. As $\epsilon_1\leq 1$, $n\geq 1$. Assume by contradiction that $r_{min}\eps_1>\ttr_{\uphi_n}$, then we have $\ttr_{\uphi_{n-1}}\leq \ttr_{\uphi_n} r_{min}^{-1}<\epsilon_1$, contradicting the assumption that $n=\tau_{\epsilon_1}(\uphi)$. The inequality $\ttr_{\uphi_{n-1}}<r_{min}^{-1}\eps_1$ can be proved similarly by contradiction and we omit it here.

For (2), first we observe that if $m=n=0$, then the equivalence clearly holds since this implies that $\eps_1\geq \epsilon_2>1$. Therefore, we may assume $m+n\geq 1$. We first suppose that $m=\tau_{\epsilon_2/\ttr_{\uphi_n}}(T^n(\uphi))$ and set $k=\tau_{\eps_2}(\uphi)$. Note that $k\geq n$ since $\eps_2\leq \eps_1$. As
\[ \ttr_{\uphi_{m+n}}=\ttr_{\uphi_n}\ttr_{(T^n(\uphi))_{m}}<\ttr_{\uphi_n}\frac{\epsilon_2}{\ttr_{\uphi_n}}=\epsilon_2,\]
We have $k\leq m+n$. On the other hand, if $k<m+n$, then $\ttr_{T^n(\uphi)_{k-n}}<\epsilon_2/\ttr_{\uphi_n}$, contradicting the assumption that $m=\tau_{\epsilon_2/\ttr_{\uphi_n}}(T^n(\uphi))$. Hence, $\tau_{\epsilon_2}(\uphi)=m+n$. Now we suppose  $\tau_{\epsilon_2}(\uphi)=m+n$. Then 
\[ \ttr_{(T^n(\uphi))_m}=\frac{\ttr_{\uphi_{m+n}}}{\ttr_{\uphi_n}}<\frac{\epsilon_2}{\ttr_{\uphi_n}}, \]
which implies that $\tau_{\epsilon_2/\ttr_{\uphi_n}}(T^n(\uphi))\leq m$. If $\tau_{\epsilon_2/\ttr_{\uphi_n}}(T^n(\uphi))=k<m$, then 
\[ \ttr_{\uphi_{k+n}}=\ttr_{\uphi_n}\ttr_{(T^n(\uphi))_{k}}<\epsilon_2, \]
contradicting the assumption that $\tau_{\eps_2}(\uphi)=m+n$.
\end{proof}

\bigskip

\subsection{The stopping time measure and its decomposition}

Given $n\in \N$ and $(\phi_1,\cdots,\phi_n)\in \Lambda^{n}$ ($\Lambda=\supp \lambda$), we define a cylinder set
\[ C(\phi_1,\cdots,\phi_n)=\{(\phi_1,\cdots,\phi_n)\}\times \prod_{n+1}^{\infty} \Lambda\subset \Lambda^{\N}. \]
Consider the countable subfamily $\mathcal{C}\subset \mathcal{P}(\Lambda^{\N})$ consisting of all cylinder sets and the empty set $\emptyset$, i.e., 
\[ \mathcal{C}=\{C(\phi_1,\cdots,\phi_n): (\phi_1,\cdots,\phi_n)\in\Lambda^n,\, n\in \N\}\cup \{\emptyset\}. \]
For any $C=C(\phi_1,\cdots,\phi_n)\in \mathcal{C}$, we denote
\[ \ttr(C)=\ttr_{\phi_1}\cdots\ttr_{\phi_n}\in \R,\,\, \text{ and } \lambda^{\N}\{C\}=\lambda_{\phi_1}\cdots\lambda_{\phi_n}\in \R. \]
Here $\lambda_{\phi}=\lambda\{\phi\}$ for $\phi\in \supp \lambda$. By convention, we set $\ttr(\emptyset)=1$.

\medskip

We equip $\Lambda^\N$ with the Borel $\sigma$-algebra $\mathcal{B}$ generated by $\{C:C\in \mathcal{C}\}$, and the set $\mathcal{C}$  with the Borel $\sigma$-algebra $\mathcal{P}(\mathcal{C})$, which is the collection of all subsets of $\mathcal{C}$. For any $\epsilon>0$, we define a measurable map ${\tau}^{\cC}_{\epsilon}$ from $\Lambda^{\N}$ to $\cC$ by
\begin{align}\label{eq: map from symbolic to cylinder}
{\tau}^{\cC}_{\epsilon}:&\Lambda^{\N}{\longrightarrow} \mathcal{C} \nonumber\\
\uphi&\mapsto C(\phi_1,\cdots,\phi_{\tau_{\epsilon}(\uphi)}).
\end{align}
Here by convention if $\eps>1$, then we set $\tau^{\cC}_{\epsilon}(\Lambda^{\N})=\emptyset$.
We further define a map ${\varphi}:\mathcal{C} \to \Aff(\R)$ by
\[ {\varphi}(C(\phi_1,\cdots,\phi_n))=\phi_1\circ\cdots\circ \phi_n \,\text{ if } C\neq \emptyset,\,\, \text{ and } \varphi(\emptyset)=\Id. \]
Note that for any $C\in\cC$, $\ttr_{\varphi(C)}=\ttr(C)$. 

\medskip

\begin{definition}[The stopping time measure]
Given $\epsilon>0$, let $\lambda^{\otimes\tau_{\epsilon}}$ be a probability measure on $\mathcal{C}$ given by 
\[ \lambda^{\otimes\tau_{\epsilon}}=(\tau^{\cC}_{\epsilon})_* \lambda^{\N}. \]
The push-forward measure $\lambda^{\epsilon}=\varphi_*\lambda^{\otimes \tau_{\epsilon}}$ via $\varphi$ on $\Aff(\R)$ is called the \textit{$\epsilon$-stopping time measure} associated to $\lambda$.
\end{definition}

\medskip

We observe some basic properties of $\lambda^{\otimes \tau_{\eps}}$, thus of $\lambda^{\eps}$. As $\Lambda$ is finite, $\lambda^{\otimes \tau_{\epsilon}}$ is an atomic measure and $\supp \lambda^{\otimes \tau_{\epsilon}}$ is finite. In particular, we have the following explicit description of $\supp \lambda^{\otimes \tau_{\epsilon}}$: If $\epsilon\leq 1$, then a cylinder set $C=C(\phi_1,\cdots,\phi_n)\in \supp \lambda^{\otimes \tau_{\eps}}$ if and only if there exists (thus for any) $\uphi\in C$, $\tau_{\eps}(\uphi)=n$. In this case, $\lambda^{\otimes \tau_{\eps}}\{C\}=\lambda^{\N}\{C\}$. If $\eps>1$, then $\supp \lambda^{\otimes \tau_{\eps}}=\{\emptyset\}$. By definition, $\lambda^{{\epsilon}}$ is the distribution of $\phi_1\circ\cdots \circ \phi_{\tau_{\epsilon}(\uphi)}$ when $\uphi$ is distributed according to the law $\lambda^{\N}$.

The following lemma asserts that self-similar measures are stationary under convolution of their stopping time measures (cf. \cite[Lemma A.2]{Ben22}).

\begin{lemma}\label{lemma: stationarity w.r.t stoppoing time}
Let $\sigma$ be the self-similar measure on $\R$ associated to $\lambda$. Then for any $\epsilon>0$, we have
\[ \sigma=\int \phi_*\sigma \dd \lambda^{{\epsilon}}(\phi). \]
\end{lemma}

\begin{proof}
Let $\tau^{\cC}_{\eps}$ be the map defined as in \eqref{eq: map from symbolic to cylinder}. Let $\mathcal{B}_{\epsilon}=(\tau^{\cC}_{\epsilon})^{-1}\mathcal{P}(\cC)$ be a $\sigma$-subalgebra of $\mathcal{B}$ on $\Lambda^{\N}$.  Let $\mathbb{E}(\cdot|\mathcal{B}_{\epsilon})$ be the associated conditional expectation. Then given any measurable function $f$ on $\R$, we denote by $\bar{f}(\uphi)=f(\uphi(0))$ for any $\uphi\in \Lambda^{\N}$. Then we have for any $\uphi\in \supp \lambda^{\N}$, 
\[ \mathbb{E}(\bar{f}|\mathcal{B}_{\epsilon})(\uphi)=\int \bar{f}(\phi_1,\cdots,\phi_{\tau_{\epsilon}(\uphi)},\uphi')\dd \lambda^{\N}(\uphi'). \]
Note that $\sigma=\int \delta_{\uphi(0)}\dd \lambda^{\N}(\uphi)$. Hence, the property of conditional measures (see e.g. \cite[Theorem 5.1]{EW11}) and Fubini's theorem imply
\begin{align*}
   \int f(s)\dd \sigma(s)&= \int \bar{f}(\uphi)\dd \lambda^{\N}(\uphi)\\
   &=\int \mathbb{E}(\bar{f}|\mathcal{B}_{\epsilon})(\uphi)\dd \lambda^{\N}(\uphi)\\
    &=\int \int \bar{f}(\phi_1,\phi_2,\cdots,\phi_{\tau_{\epsilon}(\uphi)},\uphi')\dd \lambda^{\N}(\uphi)\dd \lambda^{\N}(\uphi')\\
    &=\int f(s)\dd \lambda^{\epsilon}* \sigma(s).
\end{align*}
This finishes the proof of the lemma.
\end{proof}

\medskip

We would like to perform a suitable decomposition of stopping time measures for later use. For this, note that there is a natural concatenation operation denoted by $\bullet$ on $\mathcal{C}\times \mathcal{C}$ given by
\begin{align*}
    \mathcal{C}\times \mathcal{C}&\overset{\bullet
    }{\longrightarrow} \mathcal{C}\\
     C(\phi_1,\cdots,\phi_n)\bullet C(\phi_1',\cdots,\phi_m')&\mapsto C(\phi_1,\cdots,\phi_n,\phi_1',\cdots,\phi_m'). 
\end{align*}
In particular,
\[ \emptyset \bullet C=C\bullet \emptyset=C, \,\, \forall C\in \mathcal{C}. \]

\medskip

The convolution of stopping time measures associated to $\lambda$ does not necessarily share some good properties with that of convolutions of $\lambda$ by itself. For example, given $\epsilon_1,\epsilon_2>0$, it is not true in general that $\lambda^{\epsilon_1}*\lambda^{\eps_2}=\lambda^{\eps_1 \eps_2
}$. The following lemma allows us to decompose a stopping time measure into a conditional convolution of two suitable stopping time measures.

\begin{lemma}\label{lemma: stopping time decomposition}
Given any $0<\epsilon_2\leq \epsilon_1$, we have
\begin{align}\label{eq: stopping time decomposition}
\lambda^{\epsilon_2}=\int\int \delta_{\phi_1\circ \phi_2}\dd \lambda^{{\epsilon_2/\ttr_{\phi_1}}}(\phi_2)\dd \lambda^{\epsilon_1}(\phi_1).
\end{align}

\end{lemma}

\begin{proof}
Since $\lambda^{\epsilon}=\varphi_* \lambda^{\otimes\tau_{\epsilon}}$ and $\varphi(C_1\bullet C_2)=\varphi(C_1)\circ\varphi(C_2)$ for any $\epsilon>0$ and $C_1,C_2\in \cC$, to establish \eqref{eq: stopping time decomposition} it suffices to show that 
\begin{align}\label{eq: stopping decomp-1}
\lambda^{\otimes\tau_{\epsilon_2}}=\int\int \delta_{C_1\bullet C_2}\dd \lambda^{\otimes\tau_{\epsilon_2/\ttr(C_1)}}(C_2)\dd \lambda^{\otimes\tau_{\epsilon_1}}(C_1).
\end{align}

Let $\nu$ be the measure on the right hand side of \eqref{eq: stopping decomp-1}. Since both $\lambda^{\otimes\tau_{ \epsilon_2}}$ and $\nu$ are atomic measures, it suffices to prove that $\supp \lambda^{\otimes\tau_{ \epsilon_2}}=\supp \nu$ and for each $C\in \supp \lambda^{\otimes\tau_{ \epsilon_2}}$, we have $\lambda^{\otimes\tau_{\epsilon_2}}\{C\}=\nu\{C\}$. If $\epsilon_2>1$, then $\supp \lambda^{\otimes \tau_{\epsilon_2}}=\supp \nu=\{\emptyset\}$ and the lemma follows.

We are left with the case $\eps_2\leq 1$. We first prove that $\supp \lambda^{\otimes\tau_{ \epsilon_2}}=\supp \nu$.  Let $C=C(\phi_1,\cdots,\phi_n)\in \supp \lambda^{\otimes \tau_{\eps_2}}$ and $\uphi\in C$. Let $m=\tau_{\eps_1}(\uphi)\geq 0$. Then $C_1=C(\phi_1,\cdots,\phi_m)\in \supp \lambda^{\otimes \tau_{\eps_1}}$. We set $C_2=C(\phi_{m+1},\cdots,\phi_n)$. Applying Lemma \ref{lemma: basic prop of stopping time} (2), we obtain $n-m=\tau_{\eps_2/\ttr(\uphi_m)}(T^m(\uphi))$. As $T^m(\uphi)\in C_2$, we have $C_2\in \supp \lambda^{\otimes \tau_{\eps_2/\ttr(C_1)}}$, whence $\supp \lambda^{\otimes \tau_{\eps_2}}\subset \supp \nu$.

On the other hand, let $C=C(\phi_1,\cdots,\phi_n)\in \supp \nu$. Then $C=C_1\bullet C_2$ for $C_1=C(\phi_1,\cdots,\phi_m)\in \supp \lambda^{\otimes \tau_{\eps_1}}$ and $C_2=C(\phi_{m+1},\cdots,\phi_n)\in \supp \lambda^{\otimes \tau_{\eps_2/\ttr(C_1)}}$. Let $\uphi\in C$. Then $m=\tau_{\eps_1}(\uphi)$ and $n-m=\tau_{\eps_2/\ttr_{\uphi_m}}(T^{m}(\uphi))$. Applying Lemma \ref{lemma: basic prop of stopping time} (2), we obtain that $\tau_{\eps_2}(\uphi)=n$, whence $C\in \supp \lambda^{\otimes \tau_{\eps_2}}$. This concludes the proof that $\supp \lambda^{\otimes\tau_{ \epsilon_2}}=\supp \nu$.

The fact that for any $C\in \supp \lambda^{\otimes\tau_{ \epsilon_2}}$, $\lambda^{\otimes\tau_{\epsilon_2}}\{C\}=\nu\{C\}$ is clear from the definition of $\lambda^{\otimes \tau_{\eps_2}}$ and $\nu$. Indeed, if $C=C(\phi_1,\cdots,\phi_n)$, then 
\[ \lambda^{\otimes\tau_{\epsilon_2}}\{C\}=\nu\{C\}=\lambda_{\phi_1}\cdots\lambda_{\phi_n}. \]

\end{proof}

\medskip

\subsection{Random walks on $X$ associated to IFS}\label{subsec: associated rw}
Let $\lambda=\otimes_1^d \lambda_i$ be a self-affine IFS of product type.  We consider various random walks induced by $\lambda_i$ ($1\leq i\leq d$) on $X$. For $1\leq k\leq d$, we denote by $G_k<G$ the subgroup generated by
\begin{align*}
    \{\Id+tE_{i,j}:i,j\in \dbracket{1,k},& i\neq j,t\in \R\}\cup\{\Id+tE_{i,d+1}:i\in \dbracket{1,k},t\in \R\}\\
    &\cup \{\Id+tE_{d+1,i}: i\in \dbracket{1,k},t\in \R\},
\end{align*}
where $E_{i,j}\in \mathrm{M}_{d+1}(\R)$ is a matrix with $1$ in the $(i,j)$-th entry and $0$ elsewhere. More explicitly, 
\begin{align}\label{eq: defn of Gk}
    G_k=\left\{ \begin{pmatrix}
    M_k & \boldsymbol{0}_{k,d-k} & \bv_1\\
    \boldsymbol{0}_{d-k,k}& \Id_{d-k}& \boldsymbol{0}_{d-k,1}\\
    \bv_2^t& \boldsymbol{0}_{1,d-k}& a
\end{pmatrix}\in \SL_{d+1}(\R) : M_k\in \mathrm{M}_k(\R),\bv_1,\bv_2\in \R^k, \, a\in \R\right\}.
\end{align}
Observe that $G_k$ is isomorphic to $\SL_{k+1}(\R)$ and we have 
\[ G_1\subset G_2\subset \cdots \subset G_d=G. \]
For each $k$, consider the subgroup $P_k=A_kU_k$ of $G_k$ where 
\begin{align}\label{eq: ak and uk}
   & A_k=\{a_k(\bt)=\diag(t_1,\cdots,t_k,1,\cdots,1,\prod_{i=1}^k t_i^{-1}): \bt\in \R^k_{>0}\};\nonumber\\
&U_k=\{u_k(\bs)=u(s_1,\cdots,s_k,0,\cdots,0): \bs\in \R^k\}.
\end{align}
In particular, we denote by $\{a_k(t):t>0\}\subset A_k$ the one-parameter diagonal flow defined by 
\begin{align}\label{eq: defn of a_k}
    a_k(t)=\diag(t^{\tfrac{1}{k+1}},\cdots,t^{\tfrac{1}{k+1}},1,\cdots,1,t^{-\tfrac{k}{k+1}}).
\end{align}
By definition, $a_d(t)=a(t)$. Given $1\leq k\leq d$, there exists an anti-morphism $\iota_k: \Aff(\R^k)\to G_k$  given by $\iota_k(\phi)=g_{\phi}$, where for $\phi(\bs)=M_{\phi}\bs+\ttb_{\phi}$,
\[g_{\phi}=\begin{pmatrix}
     \det(M_{\phi})^{\tfrac{1}{k+1}}M_{\phi}^{-1} & \boldsymbol{0}_{k,d-k} & \boldsymbol{0}_{k,1}\\
    \boldsymbol{0}_{d-k,k}& \Id_{d-k}& \boldsymbol{0}_{d-k,1}\\
    \boldsymbol{0}_{1,k}& \boldsymbol{0}_{1,d-k}& \det(M_{\phi})^{\tfrac{1}{k+1}}
\end{pmatrix}\begin{pmatrix}
    \Id_k & \boldsymbol{0}_{k,d-k} & \ttb_{\phi}\\
    \boldsymbol{0}_{d-k,k}& \Id_{d-k}& \boldsymbol{0}_{d-k,1}\\
    \boldsymbol{0}_{1,k}& \boldsymbol{0}_{1,d-k}& 1
\end{pmatrix}.\]

\medskip

For the purpose of this article, given any $1\leq i\leq k\leq d$, we consider the embedding $\Xi_{i,k}:\Aff(\R)\to \Aff(\R^k)$ defined by
\begin{align}\label{eq: embedding of affine map}
    \Xi_{i,k}(\phi)=(s_1,\cdots,s_{i-1},\phi(s_i),s_{i+1},\cdots,s_k),\, \forall \bs \in \R^k.
\end{align}
Denote by 
\begin{align}\label{eq: lambda_i,k and mu_i,k}
    \lambda_{i,k}=(\Xi_{i,k})_{*}\lambda_i, \, \text{ and }\, \mu_{i,k}=(\iota_k)_* {\lambda_{i,k}}
\end{align}
the probability measure on $\Aff(\R^k)$ and  the induced measure on $G_k$, respectively. Moreover, given $\eps>0$, we let 
\begin{align}\label{eq: lambda_i,k eps and mu_i,k eps}
    \lambda^{\eps}_{i,k}=(\Xi_{i,k})_{*}(\lambda^{\eps}_i),\, \text{ and }\,  \mu^{\eps}_{i,k}=(\iota_k)_* {\lambda^{\eps}_{i,k}}.
\end{align}
 Note that for $1\leq i\leq k\leq d$ and any $g\in \supp \mu_{i,k}$, we have
\begin{align}\label{eq: formula for g in mu_i,k}
    g=\diag(\ttr_{g}^{\tfrac{1}{k+1}},\cdots,\underbrace{{\ttr_{g}^{-\tfrac{k}{k+1}}}}_{\text{ith row}},\cdots,\ttr_{g}^{\tfrac{1}{k+1}},1,\cdots,1,\ttr_{g}^{\tfrac{1}{k+1}} )u_k(0,\cdots,0,\underbrace{\ttb_{g}}_{\text{ith slot}},0,\cdots,0).
\end{align}

We observe that \eqref{eq: embedding of affine map} endows  convolutions among $\mu_{i,k}$'s with a product structure.

\begin{lemma}\label{lemma: product structure}
For any $1\leq i\neq j\leq k\leq d$ and $\eps_1,\eps_2>0$, we have $\mu_{i,k}*{\mu_{j,k}}={\mu_{j,k}}*{\mu_{i,k}}$, and $\mu^{\eps_1}_{i,k}*{\mu^{\eps_2}_{j,k}}={\mu^{\eps_2}_{j,k}}*{\mu^{\eps_1}_{i,k}}$.
\end{lemma}

\begin{proof}
By \eqref{eq: embedding of affine map}, for any $\phi_i\in \supp \lambda_{i,k}$ and $\phi_j\in \supp \lambda_{j,k}$, we have $\phi_i\circ \phi_j=\phi_j\circ \phi_i$. As $\iota_k$ is an anti-morphism, we have $g_{\phi_j}g_{\phi_i}=\iota_k(\phi_i\circ \phi_j)=\iota_k(\phi_j\circ\phi_i)=g_{\phi_i}g_{\phi_j}$, whence the lemma.
\end{proof}

\begin{lemma}\label{lemma: rw stopping time decomposition}
Given $1\leq i\leq k\leq d$ and  $0<\epsilon_2\leq \epsilon_1$, we have
\[ \mu_{i,k}^{\eps_2}=\int \int \delta_{g_2g_1}\dd\, \mu_{i,k}^{{\epsilon_2}/\ttr_{g_1}}(g_2)\dd \,\mu_{i,k}^{\eps_1}(g_1). \]
Moreover, if $\eps_1\leq 1$, then for any $g_1\in \supp \mu_{i,k}^{\epsilon_1}$, we have
\[  \frac{\epsilon_2}{\epsilon_1} \leq \frac{\epsilon_2}{\ttr_{g_1}}\leq r_{min}^{-1}\frac{\epsilon_2}{\epsilon_1}. \]
\end{lemma}

\begin{proof}
For simplicity, we write $\lambda_{i,k}=\lambda$ and $\mu_{i,k}=\mu$. Let $f$ be a measurable function on $G$. By definition of $\mu^{\eps}$ and Lemma \ref{lemma: stopping time decomposition}, we have 
\begin{align*}
    \mu^{{\epsilon_2}}(f)&=\lambda^{{\epsilon_2}}(f\circ \iota_k)\\
    &=\int\int f(\iota_k(\phi_2\circ\phi_1))\dd \,\lambda^{{\epsilon_2/\ttr_{\phi_1}}}(\phi_2)\dd\, \lambda^{\epsilon_1}(\phi_1)\\
    &=\int \int f(g_2g_1)\dd \mu^{{\epsilon_2}/\ttr_{g_1}}(g_2)\mu^{{\epsilon_1}}(g_1).
\end{align*}
Let $g_1\in \supp \mu^{\epsilon_1}$. By Lemma \ref{lemma: basic prop of stopping time} (1), we have
\[ r_{min}\epsilon_1 \leq \ttr_{g_1}<\epsilon_1. \]
Hence, the moreover part follows.
\end{proof}
\medskip

\begin{corollary}\label{cor: prod rw stopping time decomposition}
Given $(\eps_i)_1^k$, $(\eps_i')_1^k\in \R^k_{>0}$ such that for each $1\leq i \leq k$, $\eps_i\leq \eps_i'$. Then
\begin{align*}
\mu_{1,k}^{\eps_1}*\cdots*\mu_{k,k}^{\eps_k}=\int \int \delta_{g_2 g_1} \dd \mu_{1,k}^{\eps_1/\ttr_{g_{1,1}}}*\cdots*\mu_{k,k}^{\eps_k/\ttr_{g_{1,k}}}(g_2)\dd \mu_{1,k}^{\eps_1'}*\cdots*\mu_{k,k}^{\eps_k'}(g_1),
\end{align*}
where $g_1=\prod_1^k g_{1,i}$ for $g_{1,i}\in \supp \mu_{i,k}^{\eps'_i}$.

\end{corollary}

\begin{proof}
The corollary follows by applying the decomposition Lemma \ref{lemma: rw stopping time decomposition} to each component measure $\mu_{i,k}^{\eps_i}$, and then Lemma \ref{lemma: product structure}. 
\end{proof}

\medskip

\subsection{Regularity}

We collect some known facts about regularity of self-similar measures. Let $k,l\in \N$. Given $\eps\in (0,1]$, $n\in \N$ and a contractive non-atomic self-similar measure $\sigma$ on $\R$ associated to an IFS $\lambda$ on $\Aff(\R)$, we denote  
\[\sigma^{\eps}=\lambda^{\eps}*\delta_0\,\, \text{ and }\,\, \sigma^{(n)}=\lambda^{*n}*\delta_0. \]
\begin{remark}\label{rmk: not friendly}
    It is worth noting that in general self-similar measures are not friendly in the sense of Kleinbock-Lindenstrauss-Weiss \cite{KLW04}. For details, we refer the reader to \cite[Section 4]{Baker26} where Baker constructed a product-type self-similar measure on $\R^d$ whose associated IFS is  irreducible, contractive and satisfies the open set condition, but not friendly (indeed not Federer).
\end{remark}

Let $\mathcal{P}_{k,l}$ be the collection of all polynomial maps from $\R^k$ to $\R$ of degree at most $l$. We endow a supremum norm $\norm{\cdot}$ on $\mathcal{P}_{k,l}$ by setting $\norm{P}$ to be the supremum among all absolute values of coefficients of $P$ for $P\in \mathcal{P}_{k,l}$. The following is a special case of \cite[Theorem 4.4]{BHZ25}.

\begin{proposition}\label{prop: Holder regularity}
Given $l\in \N$. There exist $C,\gamma>0$ depending on $\sigma,l$ only such that given any $P\in \mathcal{P}_{1,l}$ and $\delta>0$,
\[ \sigma\{s: \abse{P(s)}\leq \delta\norm{P}\}\leq C\delta^{\gamma}. \]
    
\end{proposition}

\medskip

\begin{corollary}\label{cor: regularity of stopping and finite time}
Let $\sigma$ be a self-similar measure on $\R$ and $l\in \N$.
There exist $C,\gamma>0$ depending on $\sigma,l$ such that the following hold:
\begin{itemize}
    \item[(1)] Given any $\eps\in (0,1]$ and  $P\in \mathcal{P}_{1,l}$. Then for any $\delta\geq \epsilon$, we have
\[ \sigma^{\epsilon}\{s: |P(s)|\leq \delta\norm{P}\}\leq C \delta^{\gamma} . \]

\item[(2)] Given any $n\in \N$, $P\in \mathcal{P}_{1,l}$. Then for any $\delta\geq e^{-n}$, we have
\[ \sigma^{(n)}\{s: |P(s)|\leq \delta\norm{P}\}\leq C\delta^{\gamma}. \]

\end{itemize}
\end{corollary}

\begin{proof}
(1) We may assume $P\neq 0$. Let 
\[E^{\eps}=\{\phi\in \supp \lambda^{\epsilon}: |P(\phi(0))|\leq \delta \norm{P}\}.\]
As $\sigma^{\eps}=\lambda^{\eps}*\delta_0$, it suffices to obtain the desired bound for $\lambda^{\eps}(E)$. Note that for any $\phi\in E^{\eps}$ and any $\uphi'\in \supp \lambda^{\N}$, we have
\begin{align}\label{eq: regularity-1}
     \abse{ \phi\circ\uphi'(0)-\phi(0)}=\abse{\ttr_{\phi}\uphi'(0)}\ll_{\sigma} \eps,
\end{align}
where $\uphi'(0)=\lim_{n\to \infty}\uphi'_n(0)$ exists for any $\uphi'\in \supp \lambda^{\N}$.
Hence, the mean value theorem and \eqref{eq: regularity-1} imply that for any $\phi\in E^{\eps}$ and any $\uphi'\in \supp \lambda^{\N}$ we have
\begin{align}\label{eq: regularity-2}
    |P(\phi\circ\uphi'(0))|\leq |P(\phi\circ\uphi'(0))-P(\phi(0))|+|P(\phi(0))|\leq C_{\sigma,l} \norm{P}(\epsilon+\delta).
\end{align}
Let
\[ \underline{E}=\{\uphi\in \supp \lambda^{\N}: |P(\uphi(0))|\leq C_{\sigma,l}\norm{P}(\epsilon+\delta)\}. \]
Then \eqref{eq: regularity-2} implies that 
\begin{align}\label{eq: identity}
    1_{E^{\eps}}(\phi)=1_{E^{\eps}}(\phi)\cdot 1_{\underline{E}}(\phi\circ \uphi'), \,\forall \,\phi\in \supp \lambda^{\tau_{\epsilon}},\uphi'\in \supp \lambda^{\N}.
\end{align}
As $\sigma=\lambda^{\N}*\delta_0$, \eqref{eq: identity}, Lemma \ref{lemma: stationarity w.r.t stoppoing time} and Proposition \ref{prop: Holder regularity} together imply that
\begin{align*}
    \lambda^{\epsilon}(E)
    \leq \int \int 1_{\underline{E}}(\phi \circ \uphi') \dd \lambda^{\epsilon}(\phi)\dd \lambda^{\N}(\uphi')
    =\int 1_{\underline{E}}(\uphi)\dd \lambda^{\N}(\uphi)\leq (2C_{\sigma,l}\delta)^{\gamma}.
\end{align*}
This establishes (1). For (2), we let 
\[ E^{(n)}=\{\phi \in \supp \lambda^{*n}: \abse{P(\phi(0))}\leq \delta\norm{P}\}. \]
Then for any $\phi \in E^{(n)}$ and $\uphi'\in \supp \lambda^{\N}$, we have
\[ \abse{\phi\circ \uphi'(0)-\phi(0)}\ll_{
\sigma} e^{-cn}, \]
for some $c>0$ depending only on $\sigma$. Then the rest of the proof follows the same lines to that in (1).

\end{proof}

\begin{corollary}\label{cor: regularity-stopping time for product}
Let $k,l\geq 1$ and $\sigma_1,\cdots,\sigma_k$ be self-similar measures on $\R$. Then there exist $C,\gamma>0$ depending on $(\sigma_i)_1^k$, $l$ such that the following hold:

\item[(1)] For any $\epsilon_1,\cdots,\epsilon_k\in (0,1)$, $\delta\geq \max_i \epsilon_i$ and $P\in \mathcal{P}_{k,l}$,
\[ \otimes_{1}^k \sigma_i^{\epsilon_i}\{\bs\in \R^k: |P(\bs)|\leq \delta\norm{P}  \}\leq C \delta^{\gamma}.\]

\item[(2)] For any  $n_1,\cdots,n_k\in \N$, $\delta\geq max_i e^{-n_i}$ and $P\in \mathcal{P}_{k,l}$, 
\[ \otimes_{1}^k \sigma_i^{(n_i)}\{\bs\in \R^k: \abse{P(\bs)}\leq \delta\norm{P}\}\leq C\delta^{\gamma}. \]

\item[(3)] For any $\delta>0$ and $P\in \mathcal{P}_{k,l}$,
\[ \otimes_1^k \sigma_i\{\bs\in \R^k: \abse{P(\bs)}\leq \delta\norm{P}\}\leq C\delta^{\gamma}. \]
    
\end{corollary}

\begin{proof}
We only prove (1) here, the proof of (2) and (3) are almost identical to (1). The proof of (1) is similar to that in \cite[Corollary 4.5]{BHZ25}. When $k=1$, this is Corollary \ref{cor: regularity of stopping and finite time}. Now assume $k\geq 2$ and the corollary holds for $k-1$, with the constants $C,\gamma$ depending on $\sigma_1,\cdots,\sigma_{k-1}$, $k-1$ and $l$. 

Without loss of generality, we assume that $\norm{P}\geq 1$. We can view $P$ as a polynomial with indeterminate variables $(s_1,\cdots,s_{k-1})$ and with coefficients in the ring $\R[s_k]$. The assumption that $\norm{P}\geq 1$ implies one coefficient $Q$ of $P$, viewed as an element in $\R[s_k]$, satisfies $\norm{Q}\geq 1$. Given any $s_k\in \R$, the map $P(\cdot, s_k):\R^{k-1}\to \R$  is a polynomial map with norm at least $\abse{Q(s_k)}$. As $\delta^{1/2}\geq \max_{1\leq i\leq k-1}\epsilon_i$, the induction hypothesis implies that for $s_k$ such that $|Q(s_k)|\geq \delta^{1/2}$, we have
\[ \otimes_{i=1}^{k-1}\sigma_i^{\epsilon_i}\{(s_1,\cdots,s_{k-1}): |P(s_1,\cdots,s_k)|\leq \delta \}\ll_{k-1,l} \delta^{\gamma/2}. \]
To conclude, we apply the Fubini theorem and Corollary \ref{cor: regularity of stopping and finite time} to obtain
\begin{align*}
\otimes_{i=1}^{k}\sigma_i^{\epsilon_i}\{(s_1,\cdots,s_k)&\in \R^k: |P(s_1,\cdots,s_k)|\leq \delta\}\\
&=\int_{\R}\otimes_{i=1}^{k-1}\sigma_i^{\eps_i}\{(s_1,\cdots,s_{k-1}): |P(s_1,\cdots,s_k)|\leq \delta\}\dd \sigma_k^{\eps_k}(s_k)\\
&\ll \delta^{\gamma/2}+\sigma^{\eps_k}_k\{s_k: |Q(s_k)|<\delta^{1/2}\}\\
&\ll \delta^{\gamma/2}+\delta^{\gamma(\sigma_k,1,l)/2}.
\end{align*}
This finishes the proof of the induction hypothesis for $k$, whence (1).
\end{proof}

\begin{corollary}\label{cor: regularity for product 2}
Let $d,m,l\geq 1$, $\sigma_1,\cdots,\sigma_d$ be self-similar measures on $\R$. Then there exist $C,\gamma>0$ depending on $(\sigma_i)_1^d$, $l,m$ such that for any $n_1,\cdots,n_d\in \N$, $P\in \mathcal{P}_{dk,l}$, and $\delta\geq \max_{i}e^{-n_i}$,
\[ \nu^{\otimes m}\{(\bs_i)_{i=1}^m\in (\R^d)^m: |P(\bs_1,\cdots,\bs_m)|\leq \delta\norm{P} \}\leq C\delta^{\gamma}, \]
where $\nu=\otimes_{i=1}^d \sigma_i^{(n_i)}$.
\end{corollary}

\begin{proof}
Applying Corollary \ref{cor: regularity-stopping time for product} (2) inductively, the corollary follows from the similar argument as that in Corollary \ref{cor: regularity-stopping time for product}. 
\end{proof}

\bigskip

\section{The Benoist-Quint height functions}\label{sec: height function}
The aim of this section is to show that certain averaging operators associated to random walks introduced in Section \ref{subsec: associated rw} satisfy the contraction hypothesis for Benoist-Quint height functions. The notion of height function and its contraction hypothesis originated from the works of Eskin-Margulis-Mozes \cite{EMM98} and Eskin-Margulis \cite{EM04} in the context of homogeneous dynamics. Since then, the height function has become an indispensable tool in the study of many problems in homogeneous dynamics, see e.g. \cite{BQ4,BQ1,shi20,PSS23}. The original height function in \cite{EMM98,EM04} only allows one to obtain the recurrence property of random walks whose Zariski closure is $G$ or epimorphic in $G$. This is not sufficient for our purpose. The work of Benoist and Quint \cite{BQ4}  generalized the height function appeared in \cite{EMM98}, which makes it possible for one to study the recurrence property of those random walks generating potentially smaller semisimple subgroups of $G$. This was further developed in \cite{PSS23} to handle a more general class of random walks.

\medskip

We first recall the construction in \cite{BQ4,PSS23} of a Benoist-Quint height function associated to $G_k$ (see \eqref{eq: defn of Gk}) for $1\leq k\leq d$ on $X=G/\Gamma$, for $G=\SL_{d+1}(\R)$ and $\Gamma=\SL_{d+1}(\Z)$. Fix a root system of $G_k$ such that $a_k(t)$ (see \eqref{eq: defn of a_k}) lies in the interior of the positive Weyl chamber for any $t>1$. Let $P_k^+$ be the collection of all highest weights of irreducible sub-representations of $G_k$ acting on $E=\bigwedge^* \R^{d+1}:=\bigoplus_{i=1}^d \bigwedge^i\R^{d+1}$ associated to this root system. Note that if $0\in P_k^+$, then the vectors associated to weight $0$ is invariant under the action of $G_k$. Let 
\[\delta_i=(d+1-i)i \quad \text{for } \,\, 1\leq i\leq d,\]
and
\[\delta_{\lambda}=\lambda(\log a_k(2))>0, \quad \text{ for } \,\lambda\in P_k^+\setminus\{0\}.\] 
We decompose $E=E^0\oplus \bigoplus_{\lambda\in P_k^+\setminus\{0\}} E^{\lambda}$, where $E^{\lambda}$ is the direct sum of all irreducible sub-representations of $G_k$ on $E$ with highest weight $\lambda$, and $E^0$ is the set of all $G_k$-fixed vectors in $E$. For any $\lambda\in P_k^+$, let $q_{\lambda}:E\to E^{\lambda}$ be the orthogonal projection. For any $\epsilon_0>0$, we define the function $\varphi_{k,\epsilon_0}:E_1:=\bigcup_{i=1}^d \bigwedge^i \R^{d+1}\to [0,+\infty]$ by 
\begin{align*}
    \varphi_{k,\epsilon_0}(v)=\begin{cases}
        \min_{\lambda\in P_k^+\setminus\{0\}} \epsilon_0^{\delta_i/\delta_{\lambda}} \norm{q_{\lambda}(v)}^{-1/\delta_{\lambda}}, & \, \text{ if } \norm{q_0(v)}<\epsilon_0^{\delta_i}\\
        0, & \text{ otherwise}
    \end{cases}, \,\, \forall\, v\in \bigwedge^i \R^{d+1}.
\end{align*}
Note that $\varphi_{k,\epsilon_0}(v)=+\infty$ if and only if $q_0(v)=v$ (i.e. $v$ is $G_k$-fixed) and $\norm{q_0(v)}<\epsilon_0^{\delta_i}$. 

Given $x=g\Gamma\in X$, a vector $v\in \bigwedge^i\R^{d+1}$ is said to be an \textit{$x$-integral monomial vector} if it can be written as $v=g(w_1\wedge\cdots\wedge w_i)$ for some $w_1,\cdots,w_i\in \Z^{d+1}$. Let $V(x)\subset E_1$ be the collection of all nonzero $x$-integral monomial vectors. We define the Benoist-Quint height function $f_{k,\epsilon_0}$ (associated to $G_k$)  to be
\begin{align*}
    f_{k,\epsilon_0}(x)=\max_{v\in V(x)}\varphi_{k,\epsilon_0}(v).
\end{align*}
It is clear from the definition that there exists an absolute constant $D>0$ such that for any $k\in \dbracket{1,d}$, $\eps_0>0$, $g\in G$ and $x\in X$,
\begin{align}\label{eq: Lipschitz}
    f_{k,\eps_0}(gx)\leq \norm{g}^D f_{k,\eps_0}(x).
\end{align}


We record the following lemma relating the value of $f_{k,\eps_0}$ and the injectivity radius in $X$.

\begin{lemma}\label{lemma: BQ function vs inj}
There exists an absolute constant $C>0$ such that for any $k\in \dbracket{1,d}$, $\rho,\epsilon_0\in (0,1)$ and $x\in X$, if $f_{k,\epsilon_0}(x)<\rho^{-1}$, then $\inj(x) \geq \epsilon_0^C \rho^C$.
\end{lemma}

\begin{proof}
Since there are only finitely many subgroups $G_k$ ($1\leq k\leq d$) to consider, it is enough to prove the existence of the constant $C$ in the lemma for one $G_k$. By assumption, for any $v\in V(x)\cap \bigwedge^i\R^{d+1}$ we have $\varphi_{k,\epsilon_0}(v)<\rho^{-1}$. Suppose $\norm{q_0(v)}\geq \epsilon_0^{\delta_i}$, then $\norm{v}\geq \norm{q_0(v)}\geq \epsilon_0^{\delta_i}$. If $\norm{q_0(v)}<\epsilon_0^{\delta_i}$, then there exists some $\lambda\in P^+_k\setminus\{0\}$ such that $\epsilon_0^{\delta_i/\delta_{\lambda}}\norm{q_{\lambda}(v)}^{-1/\delta_{\lambda}}<\rho^{-1}$. This implies that
\[ \norm{v}\geq \norm{q_{\lambda}(v)}>\epsilon_0^{\delta_i}\rho^{\delta_{\lambda}}\geq \epsilon_0^{\max \delta_i}\rho^{\max\delta_{\lambda}}. \]
In either case, we have for any $v\in V(x)$, $\norm{v}\geq \epsilon_0^{\max \delta_i}\rho^{\max\delta_{\lambda}}$. By Mahler's compactness criterion, there exists $C_1>0$ depending only on $X$ such that $\inj(x)\geq \min_{v\in V(x)}\norm{v}^{C_1}$. Hence, the lemma follows.

\end{proof}

The following two lemmas explain that certain random walks typically expand vectors in irreducible linear representations of $G_k$ ($k\in \dbracket{1,d}$).
\begin{lemma}\label{lemma: contracting  operator from stopping time}
Let $k\in \dbracket{1,d}$ and $V$ be a nontrivial irreducible linear representation of $G_k$. Then for any  $0<\theta\lll_{G_k,V} 1$ and $0<\epsilon\lll_{(\sigma_i)_1^k} 1$, there exist $c,\delta\in(0,1)$ depending on $V,\theta,(\sigma_i)_1^k$ such that for any $(\eps_i)_1^k\in (0,1)^k$ satisfying
\begin{align*}
    \epsilon^\theta \leq \min_i \frac{\eps_i}{\epsilon}\leq \max_i \frac{\eps_i}{\epsilon}\leq \epsilon^{-\theta},
\end{align*}
and any nonzero $v\in V$, 
\[ \int \frac{1}{\norm{gv}^{\delta}}\dd \mu_{1,k}^{\eps_1}*\cdots*\mu_{k,k}^{\eps_k}(g)\leq  \epsilon^{c}\norm{v}^{-\delta}. \]
\end{lemma}

\begin{proof}
To lighten the notation, we write $\mu=\mu_{1,k}^{\eps_1}*\cdots*\mu_{k,k}^{\eps_k}$. By definition of $\mu$, the assumption of $(\eps_i)_1^k$ in the lemma and Lemma \ref{lemma: basic prop of stopping time} (1), for any $g\in \supp \mu$, we can write $g=a_g a_k(\epsilon^{-1})u(\ttb_g)$ for some $a_g\in (A_k)_{\epsilon^{-O(\theta)}}$, where $(A_k)_{\epsilon^{-O(\theta)}}$ is the $\epsilon^{-O(\theta)}$-neighborhood of the identity in $A_k$, and $\bttb_g=(\ttb_{g_1},\cdots,\ttb_{g_k})$ where $g_i$ is distributed according to the law $\mu_{i,k}^{\eps_i}$.

Let $V^+\subset V$ be the direct sum of eigenspaces of $a_k(\epsilon^{-1})$ with eigenvalue greater than $1$. There exists $c_0=c_0(V)>0$ such that for any $c_1\leq c_0$ and $v\in V^+$, we have $\norm{a_k(\epsilon^{-1})v}\geq \epsilon^{-c_1}\norm{v}$. Let $\pi_+:V\to V^+$ be the orthogonal projection. For any $v\in V$ such that $\norm{v}=1$, we define a map from $\R^k$ to $\R$ through
\[ P_v(\bs)=\norm{\pi_+(u_k(\bs)v)}^2. \]
Then $P_v$ is a nontrivial polynomial of bounded degree  for all $v\in V$ (see e.g. \cite[Example 1.1 and Lemma 3.1]{Shi21} ). Hence, $ \norm{P_v}\simeq_V 1$ for all $v\in V$ with $\norm{v}=1$. By Corollary \ref{cor: regularity-stopping time for product}, there exists $\gamma>0$ depending only on $(\sigma_i)_1^k$ such that for any $\epsilon\lll_{(\sigma_i)_1^k} 1$ and $c_1\lll 1 $, we have 
\[ \mu\underbrace{\{g: \norm{\pi_+(u_k(\ttb_g)v)}<\epsilon^{c_1/2}\}}_{E}\leq \epsilon^{c_1\gamma/2}. \]
For any $g\notin E$, we have $$\norm{a_k(\epsilon^{-1})\pi_+(u_k(\ttb_g)v)}\geq \epsilon^{-c_1/2}.$$
Choose $D=D(V)>0$ such that $\norm{gv}\geq \norm{g}^{-D}\norm{v}$ for any $g\in G_k$ and $v\in V$. Then given $g\notin E$, the above inequality implies that
\[\norm{gv}=\norm{a_g a_k(\epsilon^{-1})u_k(\ttb_g)v}\geq \norm{a_g}^{-D}\norm{a_k(\epsilon^{-1})u_k(\ttb_g)v}\geq \epsilon^{O(\theta) D}\epsilon^{-c_1/2}\geq \epsilon^{-c_1/4},\]
where the last inequality follows provided that $\theta\lll_{D,c_1} 1$. On the other hand, by the Cauchy-Schwarz inequality we have
\[ \int_{E} \norm{gv}^{-\delta}\dd \mu(g)\leq \mu(E)^{1/2}\left( \int \norm{g}^{2D\delta}\dd \mu(g) \right)^{1/2}\leq \epsilon^{c_1\gamma/8},  \]
where the last inequality follows by taking $\delta\lll_{\theta,D,c_1,\gamma} 1$ small to ensure that 
\[\int \norm{g}^{2D\delta}\dd \mu (g)<\epsilon^{-c_1\gamma/4}.\]
Therefore, we have
\begin{align*}
    \int \norm{gv}^{-\delta}\dd \mu (g)&=\int_E \norm{gv}^{-\delta}\dd \mu(g)+\int_{E^c}\norm{gv}^{-\delta}\dd \mu(g)\leq \epsilon^{c_1\gamma/8}+\epsilon^{c_1\delta/4}.
\end{align*}
We take $c=\min\{c_1\gamma/8,c_1\delta/4\}$ to conclude the proof.

\end{proof}

\medskip

For later use, we also need to consider the following probability measure on $G$ given by regular convolutions among $\mu_{i,d}$'s ($i\in \dbracket{1,d}$). Denote by $\ell_i$ the Lyapunov exponent of $\lambda_i$ for $i\in \dbracket{1,d}$. Given any $n_1\in \N$ large enough, we define $n_k\in \N$ for $2\leq k\leq d$ by
\[ n_k=\left\lfloor \frac{n_1\ell_1}{\ell_k}\right\rfloor. \]
We then define a probability measure 
$\overline{\mu}_{n_1}$ on $G$ by
\begin{align}\label{eq: overline mu n1}
\overline{\mu}_{n_1}=\mu_{1,d}^{*n_1}*\cdots *\mu_{d,d}^{*n_d}.
\end{align}
 Note that by Lemma \ref{lemma: product structure}, we have 
\[ \overline{\mu}_{n_1}^{*m}*\overline{\mu}_{n_1}^{*n}=\overline{\mu}_{n_1}^{*(m+n)}, \, \forall \, m,n,n_1\in \N.  \]
The advantage of using $\overline{\mu}_{n_1}$ for $n_1$ large is that the diagonal part of typical element in $\supp \overline{\mu}_{n_1}$ is close to $a(t)$ for $t=e^{n_1 \ell_1}$. The following lemma asserts that contraction property of the averaging operator determined by $\overline{\mu}_{n_1}$ for $n_1$ sufficiently large also holds.

\begin{lemma}\label{lemma: contracting operator of usual convolution}

Let $V$ be a nontrivial irreducible linear representation of $G$. There exist $\delta,c>0$ such that for all $n_1\ggg_{V,\sigma} 1$ and any nonzero $v\in V$, we have
\[ \int \frac{1}{\norm{gv}^{\delta}}\dd \overline{\mu}_{n_1}(g)\leq e^{-cn_1}\norm{v}^{-\delta}. \]
\end{lemma}

\begin{proof}
The proof is similar to that of Lemma \ref{lemma: contracting  operator from stopping time}. Let $V^+\subset V$ be the direct sum of eigenspaces of $a(t)$ ($t>1$) with eigenvalue greater than $1$. Let $c_1>0$ be such that for any $v\in V^+$, we have $\norm{a(t)v}>t^{c_1}\norm{v}$ for $t>1$. 

Let $\pi_+:V\to V^+$ be the orthogonal projection. Given any $v\in V$, we define a map from $\R^d$ to $\R$ by
\[ P_v(\bs)=\norm{\pi_+(u(\bs))v}^2. \]
Then by \cite[Example 1.1 and Lemma 3.1]{Shi21}, $P_v$ is a nonzero polynomial of bounded degree such that $\norm{P_v}\simeq_V 1$ for all $v\in V$ with $\norm{v}=1$. By Corollary \ref{cor: regularity-stopping time for product} (2), for any $n_1\ggg_{(\sigma_i)_{1}^d} 1$, we have
\[ \overline{\mu}_{n_1}\underbrace{\{g: \norm{ P_v(\ttb_g) }<e^{-n_1 \ell_1 c_1/2}\}}_{E_1} \leq e^{-n_1 c_1\gamma/2}, \]
provided that $c_1\lll_{\ell_1} 1$.
On the other hand, by the large deviation principle, for any $M>0$, there exists $c_2=c_2(c_1,M,(\sigma_i)_1^d)>0$ such that for any $k\in \dbracket{1,d}$ and $n_k\ggg_{c_1,(\sigma_i)_1^d,M} 1$, 
\begin{align}\label{eq: LDP for nk}
    \mu_{k,d}^{*n_k}\{g: |\log\ttr_g+n_k\ell_k|>n_k c_1/M\}<e^{-c_2n_1}.
\end{align}
When $n_1\geq \max_k \ell_k M/c_1$, for any $s\in \R$ we have
\begin{align}\label{eq: bound on nklk}
    \abse{s+n_k\ell_k}\leq \frac{n_kc_1}{M} \, \implies \abse{s+n_1\ell_1}\leq \frac{n_1c_1'}{M},
\end{align}
where $c_1'=c_1(1+\max_k \ell_1/\ell_k)$.
Let
\[ E_2=\{g\in \supp \overline{\mu}_{n_1}: |\log \ttr_{g_k}+n_1 \ell_1|>n_1c_1'/M \,\, \text{ for some } k\in \dbracket{1,d} \}, \]
where for $g\in \supp \overline\mu_{n_1}$, $g=\prod_1^d g_k$ for $g_k\in \supp \mu_{k,d}^{*n_k}$.
Then it follows from \eqref{eq: LDP for nk} and \eqref{eq: bound on nklk} that 
\[\overline{\mu}_{n_1}(E_2)\leq e^{-c_2 n_1/2},\]
provided that $n_1\ggg_{c_1',M,(\sigma_i)_1^d} 1$.
Let $E=E_1\cup E_2$. For any $g\notin E_2$, we can write $g=a_g a(e^{n_1\ell_1})u(\ttb_g)$ for some $a_g\in A_{e^{O(n_1c_1'/M)}}$. Hence, provided that $M\ggg_{D,c'_1,\ell_1} 1$, given any $g\in \supp \overline{\mu}_{n_1}\setminus E$ we have
\[ \norm{gv}\geq \norm{a_g}^{-D}\norm{a(e^{n_1\ell_1})\pi_+(u(\ttb_g)v)}\geq e^{-O(Dn_1c'_1/M)} e^{c_1' n_1\ell_1/2}\geq e^{c_1'n_1\ell_1/4}, \]
for some $D=D(V)>1$. By the Cauchy-Schwarz inequality, we have
\begin{align}\label{eq: estimate 1 in CH}
    \int_E \norm{gv}^{-\delta} \dd \overline{\mu}_{n_1}(g)\leq \overline{\mu}_{n_1}(E)^{1/2}\left( \int \norm{g}^{2\delta D} \dd \overline{\mu}_{n_1}(g) \right)^{1/2}\leq  e^{-c_2 n_1/2}, 
\end{align}
provided that $0<\delta\lll_{c'_1,c_2,\gamma} 1$. On the other hand,
\begin{align}\label{eq: estimate 2 in CH}
    \int_{E^c} \norm{gv}^{-\delta}\dd \overline{\mu}_{n_1}(g)\leq e^{-\delta c'_1 n_1\ell_1/4}.
\end{align}
Combining \eqref{eq: estimate 1 in CH} and \eqref{eq: estimate 2 in CH}, the lemma follows.
\end{proof}


\medskip

We record the following useful lemma, which will be needed in the proof of the contraction hypothesis of Benoist-Quint height functions.

\begin{lemma}\label{lemma: increasing b_k}
Let $k\in \dbracket{1,d}$. The map $b_{k}:\R_{>0}\to \R_{\geq 0}$ defined by
\[ b_{k}(\eps)=\sup_{v\in E_1, \norm{v}\geq 1} \varphi_{k,\eps}(v) \]
is non-decreasing, where $E_1= \bigcup_{1}^d\bigwedge^i \R^{d+1}$. Moreover, $\lim_{\eps\to 0^+}b_k(\eps)=0$.
\end{lemma}

\begin{proof}
By definition of $b_k(\eps)$, clearly we have $\lim_{\eps\to 0^+}b_k(\eps)=0$. Given $0<\epsilon_1\leq \epsilon_2$, we need to show $b_k(\eps_1)\leq b_k(\eps_2)$. Let $v\in \bigwedge^i \R^{d+1}$ ($i\in \dbracket{1,d}$) be such that $\norm{v}\geq 1$.  If $\norm{q_0(v)}\geq \eps_1^{\delta_i}$, then $\varphi_{k,\eps_1}(v)=0\leq \varphi_{k,\eps_2}(v)$. If $\norm{q_0(v)}<\eps_1^{\delta_i}\leq \eps_2^{\delta_i}$, then clearly by definition we have $\varphi_{k,\eps_1}(v)\leq \varphi_{k,\eps_2}(v)$. This being true for all $v\in E_1$ with $\norm{v}\geq 1$, the lemma follows.
\end{proof}
Let $C_1>1$ be a constant depending on $(G_k)_{k=1}^d$ such that Mother's inequality in \cite[Proposition 3.1]{BQ4} works for representations of $G_k$ on $E$ for any $k\in \dbracket{1,d}$. By definition of the function $b_k$, we can choose $0<\eta_1<C_1^{-1}$ small enough such that $b_k(\eta_1)<1$ for any $k\in \dbracket{1,d}$. Then by Lemma \ref{lemma: increasing b_k}, for any $\eps_0\in (0,\eta_1)$, we have $b_k(\eps_0)<1$ for every $k\in \dbracket{1,d}$.

\medskip

Let
\begin{align}\label{eq: defn of kappa_1}
    \kappa_1= \min_{\lambda\in P_k^+\setminus\{0\}, k\in \dbracket{1,d}}\delta_{\lambda}^{-1}.
\end{align}
We recall the following lemma from \cite{BQ4}:
\begin{lemma}\cite[Lemma 4.2]{BQ4}\label{lemma: consequence of MI}
 Let $\eta_1>0$ be such that $b_k(\eta_1)<1$ for every $k\in \dbracket{1,d}$. Let $k\in \dbracket{1,d}$. Given $0<\eps_0<\eta_1$ and monomials $u,v,w\in \bigwedge^* \R^{d+1}$ with degree $r\geq 0$, $s>0$ and $t>0$ respectively. Assume that $\varphi_{k,\eps_0}(u\wedge v)\geq 1$ and $\varphi_{k,\eps_0}(u\wedge w)\geq 1$. The following hold:
\begin{itemize}
    \item[(i)] If $r>0$ and $r+s+t<d+1$, then 
    \[ \min(\varphi_{k,\eps_0}(u\wedge v), \varphi_{k,\eps_0}(u\wedge w))\leq (C_1\eps_0)^{\kappa_1/2} \max(\varphi_{k,\eps_0}(u), \varphi_{k,\eps_0}(u\wedge v\wedge w)). \]

    \item[(ii)] If $r=0$ and $r+s+t<d+1$, then 
    \[ \min(\varphi_{k,\eps_0}( v), \varphi_{k,\eps_0}( w))\leq (C_1\eps_0)^{\kappa_1/2}\varphi_{k,\eps_0}(v\wedge w). \]

    \item[(iii)] If $r>0$, $r+s+t=d+1$ and $\norm{u\wedge v\wedge w}\geq 1$, then 
    \[ \min(\varphi_{k,\eps_0}(u\wedge v), \varphi_{k,\eps_0}(u\wedge w)))\leq (C_1\eps_0)^{\kappa_1/2}\varphi_{k,\eps_0}(u). \]

    \item[(iv)] If $r=0$, $r+s+t=d+1$, then $\norm{v\wedge w}<1$.
    
\end{itemize}
    
\end{lemma}

\begin{proof}
The proof of items (i) (ii) (iii) are identical to that of \cite[Lemma 4.2]{BQ4} using the Mother inequality \cite[Proposition 3.1]{BQ4}. The statement of (iv) is different from that in \cite[Lemma 4.2]{BQ4} and we provide a proof here. Assume otherwise that $\norm{v\wedge w}\geq 1$, then either $\norm{v}\geq 1$ or $\norm{w}\geq 1$. Hence, by Lemma \ref{lemma: increasing b_k} and the choice of $\eta_1$, we have either $\varphi_{k,\eps_0}(v)<1$ or $\varphi_{k,\eps_0}(w)<1$. Since $u\wedge v=v$ and $u\wedge w=w$, this contradicts the assumption in the lemma and finishes the proof of (iv).
    
\end{proof}

\medskip

\begin{proposition}\label{prop: contraction of BQ}

Given $k\in \dbracket{1,d}$, $\eps\in (0,1)$ and $C>1$. Let $\mu$ be a probability measure on $G_k$ satisfying

(1) There exists $\delta>0$ such that for any $\lambda\in P_k^+\setminus\{0\}$ and $v\in E^{\lambda}\setminus\{0\}$, 
\[ \int \frac{1}{\norm{gv}^{\delta/\delta_{\lambda}}}\dd \mu(g)\leq  \frac{\eps}{\norm{v}^{\delta/\delta_{\lambda}}}.  \]

(2) $\sup \{\norm{g}: g\in \supp \mu\}\leq C$.

Then there exists $b>0$ depending on $C,\delta$ such that for any $\epsilon_0\lll_{C,G_k} 1$, we have
\[ \int f^{\delta}_{k,\epsilon_0}(gx)\dd \mu(g)\leq \epsilon f^{\delta}_{k,\epsilon_0}(x)+b,\,\, \forall x\in X.  \]

\end{proposition}

\begin{proof}
The proof of this proposition follows  similar lines to that of \cite[Proposition 5.3]{BQ4}. We reproduce the proof here for completeness. 

To lighten the notation, we omit the index $k$ and write $f_{k,\epsilon_0}=f_{\epsilon_0}$ and $\varphi_{k,\eps_0}=\varphi_{\eps_0}$. By assumption, for any nonzero $v\in \bigwedge^i \R^{d+1}$ ($i\in\dbracket{1,d}$), 
\[ \int \varphi_{\epsilon_0}^{\delta}(gv)\dd \mu(g)\leq \min_{\lambda\in P_k^+\setminus\{0\}} \epsilon_0^{\delta\delta_i/\delta_{\lambda}} \int \norm{gq_{\lambda}(v)}^{-\delta/\delta_{\lambda}}\dd \mu(g)\leq \epsilon\varphi_{\epsilon_0}^{\delta}(v). \]
Let 
\[ C_0=\sup \{\norm{g}^{d+1}: g\in \supp \mu \}\leq C^{d+1}. \]
Denote by $\kappa_2=\max_{\lambda\in P_k^+\setminus\{0\}}\delta_{\lambda}^{-1}$. Then for any $\lambda\in P_k^+\setminus\{0\}$, $g\in \supp \mu$ and $v\in \bigwedge^* \R^{d+1}\setminus\{0\}$, we have
\[ C_0^{-1}\norm{q_{\lambda}(v)}\leq \norm{q_{\lambda}(gv)}\leq C_0\norm{q_{\lambda}(v)}, \]
which implies that for any $\epsilon_0>0$,
\begin{align}\label{eq: CH-1}
    C_0^{-\kappa_2}\varphi_{\epsilon_0}(v)\leq \varphi_{\epsilon_0}(gv)\leq C_0^{\kappa_2}\varphi_{\epsilon_0}(v).
\end{align}
We now fix $\eta_1>0$ small enough so that $b_k(\eta_1)<1$ for every $k\in \dbracket{1,d}$, and
\begin{align}\label{eq: CH-2}
C_0^{4\kappa_2/\kappa_1}C_1\epsilon_1<1, 
\end{align}
where $\kappa_1$ is defined in \eqref{eq: defn of kappa_1}, and $C_1\geq 1$ is a constant given as in \cite[Proposition 3.1]{BQ4} depending only on $(G_k)_{1}^d$. To proceed, we let $\epsilon_0\in (0,\eta_1)$. For any $x\in X$, we define
\[ \Psi_{\epsilon_0}(x):=\{v\in V(x): v \text{ is primitive and } \varphi_{\epsilon_0}(v)\geq C_0^{-2\kappa_2}f_{\epsilon_0}(x)\}. \]
\medskip

\underline{Case 1:} $f_{\eps_0}(x)\leq C_0^{2\kappa_2}$. Then for any $g\in \supp \mu$, by \eqref{eq: CH-1} we have $f_{\eps_0}(gx)\leq C_0^{3\kappa_2}$. Hence,
\[ \int f^{\delta}_{\eps_0}(gx)\dd \mu(g)\leq C_0^{3\delta\kappa_2}. \]

\medskip

\underline{Case 2:} $f_{\eps_0}(x)>C_0^{2\kappa_2}$. Then we claim that $\Psi_{\eps_0}(x)$ contains at most one element in $\bigwedge^i \R^{d+1}$ up to sign change for each $i \in \dbracket{1,d}$. The claim is proved by contradiction. Assume not, there exists some $i\in \dbracket{1,d}$ such that $\Psi_{\eps_0}(x)$ contains two non-colinear vectors $v_0,w_0\in \bigwedge^i \R^{d+1}$. We can write $v_0=u\wedge v$ and $w_0=u\wedge w$ for $u,v,w\in V(x)$ such that $v,w\in \bigwedge^j \R^{d+1}$ for some $1\leq j\leq i$ and $u\wedge v\wedge w\in \bigwedge^{i+j}\R^{d+1}\setminus \{0\}$. By assumption, we have $\min\{\varphi_{\eps_0}(u\wedge v),\varphi_{\eps_0}(u\wedge w)\}\geq C_0^{-2\kappa_2}f_{\eps_0}(x)>1$. This enables us to apply Lemma \ref{lemma: consequence of MI}. We have 4 situations to consider, each of which yields a contradiction thanks to Lemma \ref{lemma: consequence of MI}:

(i) If $j<i$ and $j<d+1-i$. Applying Lemma \ref{lemma: consequence of MI} (i), we have
\begin{align*}
    f_{\eps_0}(x)&\leq C_0^{2\kappa_2}\min\{\varphi_{\eps_0}(u\wedge v),\varphi_{\eps_0}(u\wedge w)\}\\
    &\leq \left( C_0^{4\kappa_2/\kappa_1} C_1\epsilon_0 \right)^{\kappa_1/2}\max\{\varphi_{\eps_0}(u),\varphi_{\eps_0}(u\wedge v\wedge w) \}\\
    &< f_{\eps_0}(x),
\end{align*}
where the last inequality follows from \eqref{eq: CH-2}. This is a contradiction.

(ii) $j=i<d+1-i$. In this case $u=1$. Lemma \ref{lemma: consequence of MI} (ii), \eqref{eq: CH-2} and the same computation as in (i) yields $f_{\eps_0}(x)<f_{\eps_0}(x)$.

(iii) $j=d+1-i<i$. In this case $\norm{u\wedge v\wedge w}\geq 1$.  Lemma \ref{lemma: consequence of MI} (iii), \eqref{eq: CH-2} and the same computation as in (i) yields $f_{\eps_0}(x)<f_{\eps_0}(x)$.

(iv) $j=i=d+1-i$. In this case $u=1$ and $\norm{v\wedge w}\geq 1$. This contradicts  Lemma \ref{lemma: consequence of MI} (iv). These 4 cases together finish the proof of the claim.

\medskip

Now for any $v\in V(x)\setminus \Psi_{\eps_0}(x)$ and $g\in \supp \mu$, by \eqref{eq: CH-1} we have
\[ \varphi_{\eps_0}(gv)\leq C_0^{-\kappa_1}f_{\eps_0}(x)=C_0^{-\kappa_1}\max_{w\in \Psi_{\eps_0}(x)}\varphi_{\eps_0}(w)\leq \max_{w\in \Psi_{\eps_0}(x)}\varphi_{\eps_0}(gw). \]
This implies that for any $v\in V(x)$ and $g\in \supp \mu$, we have
\[  \varphi_{\eps_0}(gv)\leq \max_{w\in \Psi_{\eps_0}(x)}\varphi_{\eps_0}(gw). \]
Hence, the claim implies
\[ \int f^{\delta}_{\eps_0}(gx)\dd \mu(g)\leq \sum_{w\in \Psi_{\eps_0}(x)}\int \varphi^{\delta}_{\eps_0}(gw)\dd \mu(g)\leq d\eps f^{\delta}_{\eps_0}(x). \]
We set $b=C_0^{3\delta\kappa_2}$ and combine upper bounds in Cases 1 and 2 to conclude the proof.

\end{proof}

\medskip

Proposition \ref{prop: contraction of BQ} can be combined with Lemma \ref{lemma: contracting  operator from stopping time} to yield:

\begin{corollary}\label{cor: stopping time iteration of CH}
There exist $\eta_1,\delta,c\in (0,1)$ and  $C,b>0$ such that for any $\epsilon_0\in (0,\eta_1)$, any $\eps\in(0,1]$ and  $x\in X$, if $(\eps_i)_1^k\in (0,1)^k$ satisfies
\begin{align}\label{eq: pertubation}
   1\leq \min_i\frac{\eps_i}{\eps}\leq \max_i \frac{\eps_i}{\eps}\leq r_{min}^{-1}, 
\end{align}
where $r_{min}=\min_{\phi\in \supp \lambda_i,1\leq i\leq d} \ttr_{\phi}$.
Then 
\[ \int f_{k,\epsilon_0}^{\delta}(gx)\dd \mu_{1,k}^{\eps_1}*\cdots*\mu_{k,k}^{\eps_k}(g)\leq C \eps^{c}f^{\delta}_{k,\epsilon_0}(x)+b. \]   
\end{corollary}

\begin{proof}
Let $\eps\in (0,1]$ and $(\eps_i)_1^k\in (0,1)^k$ satisfying \eqref{eq: pertubation}. For simplicity,  we denote 
\[ {\mu}^{(\eps_i)_{1}^k}=\mu_{1,k}^{\eps_1}*\cdots*\mu_{k,k}^{\eps_k}. \]

Let $\theta\in (0,1)$  be the constant depending on $G_k$ and $(E^{\lambda})_{\lambda\in P_k^+}$  satisfying the assumption in Lemma \ref{lemma: contracting  operator from stopping time}. Fix $n_0\ggg_{\sigma ,\theta} 1$ such that $e^{-n_0\theta}<r_{min}^d$ and if we set $\eps'=e^{-n_0}$, then Lemma \ref{lemma: contracting  operator from stopping time} holds for $\eps'$ in place of $\eps$ therein. By Lemma \ref{lemma: contracting  operator from stopping time}, there exist $c,\delta\in (0,1)$ such that for any $(\eps_i')_1^k\in (0,1)^k$ satisfying
\begin{align}\label{eq: pertubation 2}
     \eps'^{\theta}\leq \min_i\frac{\eps'_i}{\eps'}\leq \max_i \frac{\eps'_i}{\eps'}\leq \eps'^{-\theta}, 
\end{align}
and any $v\in E^{\lambda}\setminus\{0\}$, we have
\[ \int \frac{1}{\norm{gv}^{\delta/\delta_{\lambda}}}\dd {\mu}^{(\eps_i')_{1}^k}(g)\leq  \epsilon'^{c}\norm{v}^{-\delta/\delta_{\lambda}}. \]
Moreover, for any $(\eps'_i)_1^k$ satisfying \eqref{eq: pertubation 2}, we have 
\[\sup\{\norm{g}: g\in \supp {\mu}^{(\eps_i')_{1}^k}\}\leq \eps'^{-1-O(\theta)}.
\]
Hence, Proposition \ref{prop: contraction of BQ} applies, and we can choose $\eta_1\lll_{\eps',G_k} 1$ such that for any $\eps_0\in (0,\eta_1)$ and  any $(\eps_i')_1^k$ satisfying \eqref{eq: pertubation 2}, 
\begin{align}\label{eq: stopping contraction for f_k}
    \int f^{\delta}_{k,\epsilon_0}(gx)\dd {\mu}^{(\eps_i')_{1}^k}(g)\leq \epsilon'^c f^{\delta}_{k,\epsilon_0}(x)+b,\,\, \forall x\in X,
\end{align}
for some $b>0$ depending on $\eps',\delta,\theta$. We write 
\begin{align}\label{eq: log eps}
    \integer{|\log \eps|}=(p+1)n_0+m 
\end{align}
for $p\in \Z_{\geq -1}$ and $0\leq m\leq n_0-1$. 
If $p=0$ or $-1$, then $\eps \geq \eps'^2$. We may choose $C>1$ depending only on $\eps'$ such that for any $g\in \supp {\mu}^{(\eps_i)_1^k}$ and $y\in X$,
\[ f^{\delta}_{k,\epsilon_0}(gy)\leq Cf^{\delta}_{k,\epsilon_0}(y). \]
Now we assume that $p\geq 1$. In this case, $\eps<\eps'^2<r^{2/\theta}$. Due to \eqref{eq: pertubation}, we can  write
\[ \eps_i=\eps^{1-\theta_i}, \,\, \text{ for  some }\theta_i\in [0,\theta),\,  1\leq i\leq k. \]
Then we denote $(\eps_i')_1^k=(\eps'^{1-\theta_i})_1^k$. Iterating Corollary \ref{cor: prod rw stopping time decomposition}, we have
\begin{align}\label{eq: decomp 1}
   {\mu}^{(\eps_i)_{1}^k}=\int \delta_{g_{p+1}\cdots g_1} \dd {\mu}^{(\eps_i/\prod_{j=1}^p\ttr_{g_{j,i}})_1^k}(g_{p+1})\dd {\mu}^{(\eps'^p_i/\prod_{j=1}^{p-1}\ttr_{g_{j,i}})_{1}^k}(g_p)\cdots  \dd {\mu}^{(\eps'_i)_{1}^k}(g_1),
\end{align}
where $g_j=g_{j,1}\cdots g_{j,k}$ for $g_{j,i}\in \supp \mu_{i,k}^{\eps'^j/\prod_{s=1}^{j-1}\ttr_{g_{s,i}}}$ for $1\leq j\leq p$ and $1\leq i\leq k$. Moreover, thanks to Lemma \ref{lemma: basic prop of stopping time} (1) and the choice of $n_0$ (which guarantees that $\eps'< r_{min}^d$), we have
\begin{align}\label{eq: pertubation 3}
    (\eps'^j/\prod_{s=1}^{j-1}\ttr_{g_{s,i}})_{i=1}^k\in (0,1)^k \text{ satisfies \eqref{eq: pertubation 2} for any } 1\leq j\leq p.
\end{align}
In view of \eqref{eq: log eps}, we also have
\begin{align}\label{eq: last bound}
    (\eps_i/\prod_{j=1}^p\ttr_{g_{j,i}})_1^k\in [\eps'^2,\eps' r_{min}^{-1}]^k.
\end{align}
Therefore, the contraction hypothesis \eqref{eq: stopping contraction for f_k} applies to each component measure ${\mu}^{(\eps_i'^j/\prod_{s=1}^{j-1}\ttr_{g_{s,i}})_1^k}$ for $1\leq j\leq p$ and we have
\begin{align*}
  {\mu}^{(\eps_i)_{1}^k}(f^{\delta}_{k,\eps_0})&\leq e^{\delta Dn_0} \int f^{\delta}_{k,\eps_0}(g_p\cdots g_1 x) \dd {\mu}^{(\eps'^p_i/\prod_{j=1}^{p-1}\ttr_{g_{j,i}})_{1}^k}(g_p)\cdots  \dd {\mu}^{(\eps'_i)_{1}^k}(g_1)\\
    &\leq e^{2\delta Dn_0} e^{-cpn_0}f^{\delta}_{k,\epsilon_0}(x)+e^{2\delta Dn_0} b,
\end{align*}
where the first inequality follows from  \eqref{eq: last bound} and \eqref{eq: Lipschitz} for some $D>0$ depending only on $G$. This finishes the proof of the corollary.
\end{proof}

\medskip

We combine Proposition \ref{prop: contraction of BQ} and Lemma \ref{lemma: contracting operator of usual convolution} to obtain:

\begin{corollary}\label{cor: usual convolution iteration of CH}
There exist $\eta_1,\delta,c\in (0,1)$ such that for any $n_1\ggg_{\sigma} 1$, there is  $b=b(n_1,\delta)>0$ so that for any $n\in \N$, $\epsilon_0\in (0,\eta_1)$ and any $x\in X$, 
\[ \int f^{\delta}_{d,\epsilon_0}(gx)\dd \overline{\mu}_{n_1}^{*n}(g)\leq e^{-cnn_1} f^{\delta}_{d,\epsilon_0}(x)+b, \]
where $\overline{\mu}_{n_1}^{*n}$ is defined in \eqref{eq: overline mu n1}.
\end{corollary}

\begin{proof}
By Lemma \ref{lemma: contracting operator of usual convolution}, we can choose constants $\delta,c\in (0,1)$ such that for all $n_1\ggg_{\sigma} 1$ and any nonzero $v\in E^{\lambda}\setminus\{0\}$, 
\[ \int \frac{1}{\norm{gv}^{\delta/\delta_{\lambda}}}\dd \overline{\mu}_{n_1}(g)\leq e^{-cn_1}\norm{v}^{-\delta/\delta_{\lambda}}. \]
Note that as $\supp \mu_{i,d}$ is finite for any $1\leq i\leq d$, there exists $D=D(\sigma)>1$ such that
\[ \sup\{\norm{g}: g\in \supp \overline{\mu}_{n_1}\}\leq e^{Dn_1\ell_1}. \]
Hence, Proposition \ref{prop: contraction of BQ} applies: there exists $b>0$ depending on $n_1,\delta$ such that for any $\epsilon_0\lll_{n_1,G} 1$,
\begin{align}\label{eq: usual contraction for f_d}
    \int f_{d,\epsilon_0}^{\delta}(gx)\dd \overline{\mu}_{n_1}(g)\leq e^{-cn_1} f_{d,\eps_0}^{\delta}(x)+b,\,\, \forall x\in X.
\end{align}
Then the corollary follows by applying \eqref{eq: usual contraction for f_d} iteratively.
\end{proof}

We record the following immediate consequence of Corollary \ref{cor: usual convolution iteration of CH}.

\begin{corollary}\label{cor: effective recurrence of regular convolution}
There exist $C,c>0$ such that for any $n_1\ggg_{\sigma} 1$ and $b=b(n_1)>0$, the following holds: For any $n\in \N$,  $\rho>0$ and $x\in X$,
\begin{align*}
\overline{\mu}^{*n}_{n_1}*\delta_x\{ \inj\leq \rho \}\leq \rho^c(e^{-cnn_1}\inj(x)^{-C}+b).
\end{align*}

\end{corollary}

\begin{proof}

By Corollary \ref{cor: usual convolution iteration of CH} and the Markov inequality, we have
\[ \overline{\mu}^{*n}_{n_1}*\delta_x\{y: f_{\epsilon_0}(y)\geq \rho^{-1}\}\leq \rho^{\delta}(e^{-cnn_1 }f^{\delta}_{\epsilon_0}(x)+b), \]
for some $\delta,c\in (0,1)$, $b=b(n_1,\delta)>0$ and $\eps_0\lll 1$. Fix such $\eps_0$. Since the action of $G$ on $\bigwedge^* \R^{d+1}$ has no nontrivial invariant vector, Mahler's criterion implies that there exists $C=C(X)>0$ such that 
\[ \inj(y)^{-1/C} \ll_{\epsilon_0} f_{\epsilon_0}(y)\ll_{\epsilon_0} \inj(y)^{-C}, \quad \forall y\in X. \]
The corollary follows.
\end{proof}

\bigskip

\section{Effective recurrence for intertwined random walks}\label{sec: effective recurrence}

The main purpose of this section is to prove the following effective recurrence theorem. First we let 
\begin{align}\label{eq: defn of rmin}
    r_{min}=\min_{\phi\in \supp \lambda_i,i\in \dbracket{1,d}} \ttr_{\phi}>0.
\end{align}
\begin{thm}\label{thm: weighted effective recurrence}
There exist constants $C,c>0$ depending only on $\sigma$ and $G$ such that the following holds. Given parameters $\epsilon_1, \cdots, \epsilon_d\in (0,1)$. Assume that there exists a permutation $S$ on $\dbracket{1,d}$ such that $\eps_{S(i)}\leq \eps_{S(j)}$ for any $1\leq i\leq j\leq d$ and  
\begin{align}\label{eq: condition on epsilon i}
\eps_{S(d)}^d\prod_{j=1}^{d-1}\eps_{S(j)}^{-1}\leq r_{min}^d.
\end{align}
Then for any $x\in X$ and $0<\rho< 1$, we have
\[ \mu_{1,d}^{\eps_1}*\cdots*\mu_{d,d}^{\eps_d}*\delta_x\{\inj<\rho\}<C\rho^c\left(1+(\epsilon_{S(d)}^d \prod_{i=1}^{d-1}\eps^{-1}_{S(i)})^{c}\inj(x)^{-C}\right). \]
\end{thm}
Here condition \eqref{eq: condition on epsilon i} implies that $\epsilon_{i}^d \prod_{j\neq i}\eps^{-1}_{j}< r_{min}^d$ for any $i\in \dbracket{1,d}$. Indeed, \eqref{eq: condition on epsilon i} can be weakened to $\eps_{S(d)}^d\prod_{j=1}^{d-1}\eps_{S(j)}^{-1}<1$. The stronger constraint in \eqref{eq: condition on epsilon i} is stated  for simplicity of the proof.

\medskip

\begin{proof}[Proof of equivalence between Theorem \ref{thm: nondiv} and Theorem \ref{thm: weighted effective recurrence}.]

Given $\bt\in \R^d_{>1}$ with $\integer{\bt}\geq r_{min}^{-d}$, we define parameters $\eps_1,\cdots,\eps_d\in (0,1)$ by setting 
\begin{align*}\label{eq: defn of eps_i using t_i}
    \eps_i=t_i^{-\tfrac{1}{d+1}}\prod_{j=1}^d t_j^{-\tfrac{1}{d+1}}, \,\, \forall\, 1\leq i\leq d. 
\end{align*}
Equivalently, 
\[t_i=\eps_i^{-(d+1)}\prod_{j=1}^d \eps_j, \,\,\forall \, 1\leq i\leq d.\] 
The assumption $\integer{\bt}\geq r_{min}^{-d}$ and the definition of $(\eps_i)_1^d$ ensure that there exists a permutation $S$ on $\dbracket{1,d}$ such that $\eps_{S(i)}<\eps_{S(j)}$ for any $1\leq i\leq j\leq d$ and \eqref{eq: condition on epsilon i} holds for $(\eps_i)_1^d$ and $S$. In particular, we have
\begin{align}\label{eq: t and prod of eps}
\integer{\bt}=\epsilon_{S(d)}^{-d} \prod_{i=1}^{d-1}\eps_{S(i)}.
\end{align}

\medskip

Denote
\[ \lambda=\lambda_{1,d}^{\eps_1}*\cdots*\lambda_{d,d}^{\eps_d} \,\, \text{ and }\,\,  \mu=\mu_{1,d}^{\eps_1}*\cdots*\mu_{d,d}^{\eps_d}.  \]
Let $\rho>0$ and $\chi_i$ be the characteristic function of $\{\inj<\rho^i\}$ for $i=1,2$. By Lemma \ref{lemma: basic prop of stopping time}, for any $g\in \supp \mu$, $g=a_g a(\bt)u(\bs)$ for some $a_g\in A_{O(1)}$ and $\bs \in \supp \lambda*\delta_0$. This and the compactness of $\supp \sigma$ implies that for any $\rho\lll_{\sigma} 1$, $y\in X$, $\bs\in \supp \sigma$ and $g\in \supp \mu$,
\begin{align}\label{eq: implication 1}
\chi_2(a_g^{-1}u(\bs)y)\leq\chi_1(y).
\end{align}
Using $\lambda*\sigma=\sigma$, we obtain for $i=1,2$,
\begin{align}\label{eq: int of chi}
    \int \chi_i(a(\bt)u(\bs)x)\dd \sigma(\bs) &=\int \int \chi_i(a_g^{-1}u(\bs)gx)\dd \mu(g)\dd \sigma(\bs).
\end{align}
Now assume Theorem \ref{thm: weighted effective recurrence} holds. Plug in $i=2$ for \eqref{eq: int of chi}, by \eqref{eq: implication 1} we have
\begin{align}\label{eq: recurrence implies non-div}
    \sigma\{\bs: \inj(a(\bt)u(\bs)x)<\rho^2\}\leq \mu*\delta_x\{\inj<\rho\}.
\end{align}
Applying Theorem \ref{thm: weighted effective recurrence} to \eqref{eq: recurrence implies non-div} and noting \eqref{eq: t and prod of eps}, we obtain Theorem \ref{thm: nondiv} for any $\bt\in \R^d_{>1}$ with $\integer{\bt}\geq r_{min}^{-d}$. For any $\bt\in \R^d_{>1}$ such that $\integer{\bt}<r_{min}^{-d}$, $\integer{r_{min}^{-d}\bt}>r_{min}^{-d}$. Since $a(r_{min}^{-d})$ lies in a compact set, Theorem \ref{thm: nondiv} still holds for possibly larger $C$ and smaller $c$ in the statement of Theorem \ref{thm: nondiv}.

Next we assume Theorem \ref{thm: nondiv} holds. Plug in $i=1$ for \eqref{eq: int of chi}, by \eqref{eq: implication 1} we have
\begin{align}\label{eq: non-div implies recurrence}
    \mu*\delta_x\{\inj<\rho^2\}\leq \sigma\{\bs: \inj(a(\bt)u(\bs)x)<\rho\}.
\end{align}
Applying Theorem \ref{thm: nondiv} to \eqref{eq: non-div implies recurrence}, we obtain Theorem \ref{thm: weighted effective recurrence}.
\end{proof}

\medskip

The following proposition is the first step towards the proof of Theorem \ref{thm: weighted effective recurrence}. It asserts that  convolutions by suitable stopping time measures exhibit recurrence property: 
\begin{proposition}\label{prop: effective recurrence of suitable convolution}
  There exist $C,c>0$ such that for any   $\epsilon\in (0,1]$, $(\delta_i)_1^d\in \R^d_{\geq 1}$ such that $ \max_i \delta_i< r_{min}^{-1}$ and $\rho\lll 1$, the following hold:
\begin{itemize}
    \item[(1)]  Given $k\in \dbracket{1,d-1}$. Let
\[  \eps_i=\begin{cases}
    \delta_i\eps^{k+1} & \text{ if } 1\leq i\leq k,\\
    \delta_i \eps^k & \text{ if } k+1\leq i\leq d.
\end{cases} \]
Then for any $x\in X$,
\begin{align*}
\mu_{1,d}^{\eps_1}*\cdots*\mu_{d,d}^{\eps_d}*\delta_x\{\inj\leq \rho^C \inj(x)^C\}\leq \rho^{c}(\epsilon^{c}+1), 
    \end{align*}

\item[(2)]  Let $\eps'_i=\delta_i \eps^{d+1}$ for $1\leq i\leq d$. Then
\begin{align*}
\mu_{1,d}^{\eps'_1}*\cdots*\mu_{d,d}^{\eps'_d}*\delta_x\{\inj\leq \rho\}\leq \rho^c(\epsilon^{c}\inj(x)^{-C}+1).
\end{align*}

\end{itemize}

\end{proposition}

\begin{proof}
We first prove statement (1) when $1\leq k\leq d-1$.  To simplify our notation, we write $\mu=\mu_{1,d}^{\eps_1}*\cdots*\mu_{d,d}^{\eps_d}$. Note that by Lemma \ref{lemma: basic prop of stopping time} (1), for any $g\in \supp \mu$ we have
\begin{align*}
    g&=a_g \diag(\epsilon^{-1},\cdots,\epsilon^{-1},1,\cdots,1,\epsilon^{k})u(\bs)=a_g a_k(\epsilon^{-(k+1)})u(\bs),
\end{align*}
where $a_k$ is defined as in \eqref{eq: defn of a_k}, $a_g\in A_{O(1)}$, and $s_i$ is distributed with the law $\lambda_i^{\delta_i \epsilon^{k+1}}*\delta_0$ for $1\leq i\leq k$, and $\lambda_i^{\delta_{i}\epsilon^k}*\delta_0$ for $k+1\leq i\leq d$. This motivates us to consider the measure $\mu'$ on $G_k$ defined by 
$$ \mu':=\mu_{1,k}^{\eps^{k+1}\delta_1}*\cdots*\mu_{k,k}^{\eps^{k+1}\delta_k}.$$
Note that for any $h\in \supp  \mu'$,
\begin{align}\label{eq: h in mu'}
    h=a_h a_k(\epsilon^{-(k+1)})u_k(s_1,\cdots,s_k),
\end{align}
where $a_h\in (A_k)_{O(1)}$, and $s_i$ is distributed according to the law $\lambda_i^{\delta_i \epsilon^{k+1}}*\delta_0$ for $1\leq i\leq k$. 

Let $\eta_1$ be the constant as in Corollary \ref{cor: stopping time iteration of CH}. By Mahler's criterion, there exists $C_1>0$ depending only on $X$ such that for any $y\in X$, if $\epsilon_0= \min\{\eta_1/2,C_1^{-1}\inj(y)^{C_1}\}$, then $f_{k,\epsilon_0}(y)\leq 1$. Since $\sup \{\inj(y):y\in X\}<\infty$, we may increase $C_1$ if necessary such that $C_1^{-1}\inj(y)^{C_1}< \eta_1/2$. The assumption on $(\delta_i)_1^k$ allows us to apply Corollary \ref{cor: stopping time iteration of CH}  and obtain
\[ \int f_{k,\epsilon_0}^{\delta}(hy)\dd \mu'(h)\leq C \epsilon^{c}+b, \]
for some $\delta,C,c>0$.
By Lemma \ref{lemma: BQ function vs inj}, if $f_{k,\epsilon_0}(hy)\leq \rho^{-1}$, then $\inj(gy)\geq \epsilon_0^C\rho^C\gg\inj(y)^{C_1 C}\rho^C$. Hence,  the Markov inequality implies
\begin{align}\label{eq: markov bound}
     \mu'\{h: \inj(hy)\leq \inj(y)^{C_1 C}\rho^{2C}\} \leq \mu'\{h: f_{k,\epsilon_0}(hy)>\rho^{-1}\}\leq \rho^{\delta/2}(\epsilon^{c}+1),
\end{align}
where we take $\rho\lll 1$ to absorb $C$ and $b$.  By \eqref{eq: h in mu'}, if $h\in \supp \mu'$ and 
\[ \inj(a_k(\epsilon^{-{(k+1)}})u_k(\ttb_h)y)\leq \inj(y)^{2C_1 C}\rho^{3C}, \]
then
\begin{align}\label{eq: inj of hy}
\inj(hy)\leq \inj(y)^{C_1 C}\rho^{2C},   
\end{align}
provided that $\rho\lll 1$.
As $\ttb_h$ is distributed according to $\lambda_{1,k}^{\eps^{k+1}\delta_1}*\cdots*\lambda_{k,k}^{\eps^{k+1}\delta_k}*\delta_0$, \eqref{eq: markov bound} and \eqref{eq: inj of hy} imply that for any $y\in X$, 
\begin{align}\label{eq: bound of E_y}
\lambda_{1,k}^{\eps^{k+1}\delta_1}*\cdots*&\lambda_{k,k}^{\eps^{k+1}\delta_k}*\delta_0\underbrace{\{\bs: \inj(a_k(\epsilon^{-(k+1)})u_k(\bs)y)\leq \inj(y)^{2C_1 C}\rho^{3C}\}}_{E_y}\nonumber\\
&\leq  \rho^{\delta/2}(\epsilon^{c}+1).
\end{align}
Now we write 
\[ \mu_1=\mu_{1,d}^{\delta_1\epsilon^{k+1}}*\cdots*\mu_{k,d}^{\delta_k\epsilon^{k+1}}\, \text{ and }\, \mu_2=\mu_{k+1,d}^{\delta_{k+1}\epsilon^k}*\cdots*\mu_{d,d}^{\delta_d\epsilon^k}.\]
Given any $y\in X$, we denote
\[ E(y):=\{g\in \supp \mu_1: \ttb_g\in E_y\}, \]
where $E_y$ is defined in \eqref{eq: bound of E_y}. Let $x\in X$ be given, and $\chi$ be the characteristic function of $\{\inj\leq \inj(x)^{3C_1 C}\rho^{4C}\}$. Note that as  $\{\ttb_{g_2}:g_2\in  \supp\mu_2\}$ is compact, $\inj(u(0,\ttb_{g_2})x)\simeq \inj(x)$ for any $g_2\in \supp \mu_2$. We observe that $\chi(g_1g_2x)=0$ for any  $g_2\in \supp \mu_2$ and $g_1\in \supp \mu_1\setminus E(u(0,\ttb_{g_2})x)$, since in this case
\[ \inj(g_1g_2x)=\inj(a_{g_1g_2}a_k(\epsilon^{-(k+1)})u_k(\ttb_{g_1})u(0,\ttb_{g_2})x)>\inj(x)^{3C_1 C}\rho^{4C}. \]
where the inequality follows since $a_{g_1g_2}\in A_{O(1)}$ and provided that $\rho\lll 1$. Therefore, with \eqref{eq: bound of E_y} we have
\begin{align*}
    \int \chi(gx) \dd \mu(g)&=\int \int \chi(g_1g_2x)\dd \mu_1(g_1)\dd\mu_2(g_2)\\
&=\int\int_{E(u(0,\ttb_{g_2})x)}\chi(g_1g_2x)\dd \mu_1(g_1)\dd \mu_2(g_2)\\
    &\leq \rho^{\delta/2}(\epsilon^{c}+1).
\end{align*}
This finishes the proof when $1\leq k\leq d-1$ in Case (1).

\medskip

For (2), the linear representation of $G$ on $\bigwedge^* \R^{d+1}$ does not have nontrivial invariant vectors. We can take $\epsilon_0\lll 1$ depending only on $G$ for the Benoist-Quint function $f_{d,\epsilon_0}$ such that the contraction hypothesis in Corollary \ref{cor: stopping time iteration of CH} holds. Fixing this constant $\eps_0$, by Mahler's criterion, for any $x\in X$ we have $\log f_{d,\epsilon_0}(x)\simeq |\log \inj(x)|$. The estimate then follows by applying Corollary \ref{cor: stopping time iteration of CH} in the case of $k=d$ and the Markov inequality.
\end{proof}

\medskip

The proof of Theorem \ref{thm: weighted effective recurrence} relies on a specific stopping time decomposition of $\mu^{\epsilon_1}_{1,d}*\cdots *\mu_{d,d}^{\eps_d}$ based on Corollary \ref{cor: prod rw stopping time decomposition}, together with the application of Proposition \ref{prop: effective recurrence of suitable convolution}. The specific decomposition we need is captured by functions $\{\xi_{i,j}:\R^d\to \R: 1\leq i,j\leq d\}$ defined as follows: Given $1\leq i,j\leq d$ and $(\eps_p)_1^d\in (0,1)^d$,
\[ \xi_{i,j}\left((\eps_p)_1^d \right)=\begin{cases}
    \eps_{d-j+1}^{d-j+1}\cdot \left( \prod_{s=1}^{d-j}\eps_s^{-1} \right) & \text{ if } 1\leq j\leq d-i+1,\\
    \eps_i\eps_{d-j+1}^{d-j}\cdot \left( \prod_{s=1}^{d-j}\eps_s^{-1} \right)& \text{ if } d-i+2\leq j\leq d.
\end{cases} \]
Here by convention we set $\prod_{s=1}^0 \eps_s^{-1}=1$. Note that $\xi_{i,d}((\eps_p)_1^d)=\eps_i$ and $\xi_{i,1}((\eps_p)_1^d)=\eps_d^d\prod_{1}^{d-1}\eps_s^{-1}$ for any $i\in \dbracket{1,d}$. An application of Lemma \ref{lemma: rw stopping time decomposition} yields the following decomposition:

\begin{lemma}\label{lemma: specific decomposition of random walk}
Let $(\eps_i)_1^d\in (0,1)^d$ be such that $\eps_i\leq \eps_j$ for any $1\leq i\leq j\leq d$ and $\eps_d^d\prod_{1}^{d-1}\eps_j^{-1}<r_{min}^d$, where $r_{min}=\min_{\phi\in \supp\lambda_i, 1\leq i\leq d} \ttr_{\phi}$. Then for any $1\leq i\leq d$, we have
\begin{align}\label{eq: specific decomposition}
     \mu_{i,d}^{\eps_i}=\int \delta_{g_d\cdots g_1} \prod_{j=1}^d \dd \mu_{i,d}^{\xi_{i,j}\left((\eps_p)_1^d \right)/\prod_{s=1}^{j-1}\ttr_{g_s}}(g_j),
\end{align}
where we set $\prod_{s=1}^0 \ttr_{g_s}=1$. Moreover, 
\begin{align}\label{eq: bound for stopping time}
  &\eps_{d-j+1}^{d-j+2}\cdot \eps_{d-j+2}^{-(d-j+2)} \leq \frac{\xi_{i,j}\left((\eps_p)_1^d \right)}{\prod_{s=1}^{j-1}\ttr_{g_s}}\leq r_{min}^{-1}\eps_{d-j+1}^{d-j+2}\cdot\eps_{d-j+2}^{-(d-j+2)} \, \text{ if }\, 2\leq j\leq d-i+1,\nonumber\\
 & \eps_{d-j+1}^{d-j+1}\cdot\eps_{d-j+2}^{-(d-j+1)} \leq \frac{\xi_{i,j}\left((\eps_p)_1^d \right)}{\prod_{s=1}^{j-1}\ttr_{g_s}}\leq r_{min}^{-1} \eps_{d-j+1}^{d-j+1}\cdot \eps_{d-j+2}^{-(d-j+1)} \, \text{ if }\, d-i+2\leq j\leq d.
\end{align}

\end{lemma}

\begin{proof}
We fix some $i\in \dbracket{1,d}$. To simplify the notation, we write $\xi_{i,j}\left( (\eps_p)_1^d \right)=\xi_{i,j}$ for any $j\in \dbracket{1,d}$, and $\mu_{i,d}=\mu$. The assumption $0<\eps_1\leq \cdots\leq \eps_d<1$ guarantees for any $1\leq i,j\leq d$, we have
\begin{align}\label{eq: bound on xi}
    \eps_i \leq \xi_{i,j}\leq \eps_d^d\prod_{s=1}^{d-1}\eps_s^{-1}.
\end{align}
We prove \eqref{eq: specific decomposition} by induction. As $\eps_i\leq \xi_{i,1}$, by Lemma \ref{lemma: rw stopping time decomposition} we have
\[ \mu^{\eps_i}=\int \delta_{g_2g_1}\dd \mu^{\eps_i/\ttr_{g_1}}(g_2)\dd \mu^{\xi_{i,1}}(g_1). \]
Suppose now that we have shown for some $1\leq k\leq d-1$ that
\[ \mu^{\eps_i}=\int \delta_{g_k\cdots g_1}\dd \mu^{\eps_i/\prod_{s=1}^{k-1}\ttr_{g_s}}(g_k) \prod_{j=1}^{k-1} \dd \mu^{\xi_{i,j}/\prod_{s=1}^{j-1}\ttr_{g_s}}(g_j). \]
Since $\eps_i\leq \xi_{i,k}$ which is guaranteed by \eqref{eq: bound on xi}, applying Lemma \ref{lemma: rw stopping time decomposition} we have 
\[ \mu^{\eps_i/\prod_{s=1}^{k-1}\ttr_{g_s}}=\int \delta_{g_{k+1}g_k} \dd \mu^{\eps_i/\prod_{s=1}^{k}\ttr_{g_s}}(g_{k+1})\dd \mu^{\xi_{i,k}/\prod_{s=1}^{k-1}\ttr_{g_s}}(g_k).  \]
Therefore,
\[ \mu^{\eps_i}=\int \delta_{g_{k+1}\cdots g_1}\dd \mu^{\eps_i/\prod_{s=1}^{k}\ttr_{g_s}}(g_{k+1})\prod_{j=1}^{k} \dd \mu^{\xi_{i,j}/\prod_{s=1}^{j-1}\ttr_{g_s}}(g_j). \]
Specifying to $k=d$ and noting that $\xi_{i,d}=\eps_i$, we obtain \eqref{eq: specific decomposition}.

\medskip

We now prove \eqref{eq: bound for stopping time}. To begin with, by Lemma \ref{lemma: basic prop of stopping time} (1) and the assumption $\xi_{i,1}=\eps_d^d\prod_{1}^{d-1}\eps_s^{-1}<r_{min}^d<1$, for any $g_1\in \supp \mu^{\xi_{i,1}}$ we have
\[ r_{min}\xi_{i,1}\leq \ttr_{g_1}< \xi_{i,1}.  \]
This implies that
\begin{align}\label{eq: bound on xi_2}
    \eps_{d-1}^d\eps_d^{-d} \leq \frac{\xi_{i,2}}{\ttr_{g_1}}\leq r_{min}^{-1}\eps_{d-1}^d\eps_d^{-d} 
\end{align}
which shows \eqref{eq: bound for stopping time} when $j=2$.
In general, given any $j\in [3,d]$ and $g_k\in \supp \mu^{\xi_{i,k}/\prod_{s=1}^{k-1}\ttr_{g_s}}$ for $1\leq k\leq j-1$. The assumption $\eps_d^d\prod_{1}^{d-1}\eps_j^{-1}<r_{min}^d$ and \eqref{eq: bound on xi} ensure that $\xi_{i,j-1}/\prod_{s=1}^{j-2}\ttr_{g_s}<1$. As $g_{j-1}\in \supp \mu^{\xi_{i,j-1}/\prod_{s=1}^{j-2}\ttr_{g_s}}$, by Lemma \ref{lemma: basic prop of stopping time} (1) we have
\[ r_{min}\frac{\xi_{i,j-1}}{\prod_{s=1}^{j-2}\ttr_{g_s}}\leq \ttr_{g_{j-1}}\leq \frac{\xi_{i,j-1}}{\prod_{s=1}^{j-2}\ttr_{g_s}}, \]
which implies 
\begin{align}\label{eq: bound on xi_j}
    \frac{\xi_{i,j}}{\xi_{i,j-1}} \leq \frac{\xi_{i,j}}{\prod_{s=1}^{j-1}\ttr_{g_s}}\leq r_{min}^{-1} \frac{\xi_{i,j}}{\xi_{i,j-1}}. 
\end{align}
Plugging in the definition of $\xi_{i,j}$ in \eqref{eq: bound on xi_j} yields \eqref{eq: bound for stopping time} for $3\leq j\leq d$.
\end{proof}

\medskip

With Proposition \ref{prop: effective recurrence of suitable convolution} and Lemma \ref{lemma: specific decomposition of random walk}, we are ready to present

\begin{proof}[Proof of Theorem \ref{thm: weighted effective recurrence}]
First we make a reduction by claiming that without loss of generality, we may assume parameters $\eps_1,\cdots,\eps_d\in (0,1)$ satisfy $0<\epsilon_1\leq \cdots\leq \epsilon_d<1$ and \eqref{eq: condition on epsilon i} holds with $S$ being the identity permutation. Indeed, assume that the theorem has been proved for all parameters $\eps_1,\cdots,\eps_d\in (0,1)$ satisfying both $0<\eps_1\leq \cdots\leq \eps_d<1$ and \eqref{eq: condition on epsilon i} (with $S$ being identity permutation), and any product-type self-affine measure $\sigma$. Now given $(\eps_i)_1^d$ as in the theorem and let $S$ be the given permutation satisfying $\eps_{S(i)}\leq \eps_{S(j)}$ for $1\leq i< j\leq d$ and \eqref{eq: condition on epsilon i}. By abuse of notation, we denote by $S:\R^d\to \R^d$ the map defined by $S((x_i)_1^d)=(x_{S(i)})_1^d$.  Note that the measure $S^{-1}_*\sigma$ is a product-type self-affine measure as well ($S^{-1}_* \sigma=\otimes_{i=1}^d\sigma_{S(i)}$ for $\sigma=\otimes_1^d \sigma_i$), and the stopping time measure driven by $S^{-1}_*\sigma$ and $(\eps_{S(i)})_1^d$ is 
\begin{align}\label{eq: equality after permutation}
\mu^{\eps_{S(1)}}_{S(1),d}*\cdots*\mu_{S(d),d}^{\eps_{S(d)}}=\mu_{1,d}^{\eps_1}*\cdots*\mu_{d,d}^{\eps_d},
\end{align}
where the equality follows by Lemma \ref{lemma: product structure}. 
By assumption, there exist $C,c>0$ depending on $S_*\sigma$ and $G$ such that for any $x\in X$ and $0<\rho<1$, 
\[\mu^{\eps_{S(1)}}_{S(1),d}*\cdots*\mu_{S(d),d}^{\eps_{S(d)}}*\delta_x\{\inj<\rho\}<C\rho^c\left(1+(\epsilon_{S(d)}^d \prod_{i=1}^{d-1}\eps^{-1}_{S(i)})^{c}\inj(x)^{-C}\right).\]
Then the claim follows from \eqref{eq: equality after permutation} and the fact that the permutation group of $\dbracket{1,d}$ is finite.

\medskip

Now let $0<\eps_1\leq \cdots\leq \eps_d<1$ be given such that \eqref{eq: condition on epsilon i} holds for $S=\Id$. To simplify the notation, we write
\[\widetilde{\mu}= \mu_{1,d}^{\eps_1}*\cdots*\mu_{d,d}^{\eps_d}.\]
By Lemma \ref{lemma: specific decomposition of random walk} and Corollary \ref{cor: prod rw stopping time decomposition}, we can decompose
\begin{align}\label{eq: decomposition of widetilde mu}
    \widetilde{\mu}= \int \delta_{g_d\cdots g_1}\dd \widetilde{\mu}_{d,(g_{d-1},\cdots,g_1)}(g_d)\cdots \dd\widetilde{\mu}_1(g_1),
\end{align}
where for each $1\leq j\leq d$,
\[ \widetilde{\mu}_{j,(g_{j-1},\cdots,g_1)}=\mu_{1,d}^{\xi_{1,j}((\eps_p)_1^d)/\prod_{s=1}^{j-1}\ttr_{g_{s,1}}}*\cdots *\mu_{d,d}^{\xi_{d,j}((\eps_p)_1^d)/\prod_{s=1}^{j-1}\ttr_{g_{s,d}}}, \]
and $g_{j}=g_{j,1}\cdots g_{j,d}$ with $g_{j,k}\in \supp \mu_{k,d}^{\xi_{k,j}((\eps_p)_1^d)/\prod_{s=1}^{j-1}\ttr_{g_{s,k}}}$. We define a sequence of good sets as follows. Let 
\[ E_1(x)=\{g_1\in \supp \widetilde{\mu}_1: \inj(g_1x)>\rho\}. \]
In general, let $1\leq  k\leq d-1$. Suppose we are given $g_j\in \supp \widetilde{\mu}_{j,(g_{j-1},\cdots,g_1)}$ for $1\leq j\leq k$, define
\begin{align*}
    E_{k+1}(&g_k\cdots g_1 x)\\
    &=\{g_{k+1}\in \supp \widetilde{\mu}_{k+1,(g_k,\cdots,g_1)}: \inj(g_{k+1}g_k\cdots g_1x)>\rho^C \inj(g_k\cdots g_1x)^C\}.
\end{align*}
Note that \eqref{eq: bound for stopping time} guarantees that the stopping time data in the definition of $\widetilde{\mu}_j(g_{j-1},\cdots,g_1)$ for each $1\leq j\leq d$ fits into the one in Proposition \ref{prop: effective recurrence of suitable convolution}. Hence, the application of Proposition \ref{prop: effective recurrence of suitable convolution} (1) to $\widetilde{\mu}_{k+1,(g_k,\cdots,g_1)}$ for $k\in\dbracket{1,d-1}$ and (2) to $\widetilde{\mu}_1$ yield 
\begin{align}\label{eq: lower bound for good set}
    \widetilde{\mu}_{k+1,(g_{k},\cdots,g_1)}(E_{k+1}(g_k\cdots g_{1}x))&>1-\rho^c (\eps_{d-k}^c\eps_{d-k+1}^{-c}+1), \, \forall \, k\in \dbracket{1,d-1}; \nonumber\\
   \text{ and } \,\, \widetilde{\mu}_1(E_1(x))&>1-\rho^c(\eps_d^c\prod_{s=1}^{d-1}\eps_s^{-c}\inj(x)^{-C} +1 ).
\end{align}
Moreover, a direct induction process shows that there exists $C_d>0$ such that if $g_k\in E_{k}(g_{k-1}\cdots g_1x)$ for each $1\leq k\leq d$, then 
\begin{align}\label{eq: low bound for inj}
    \inj(g_d \cdots g_1 x)>\rho^{C_d}.
\end{align}
Let $\chi=1_{\{\inj\leq \rho^{C_d}\}}$. Then \eqref{eq: decomposition of widetilde mu}, \eqref{eq: lower bound for good set} together with \eqref{eq: low bound for inj} imply 
\begin{align*}
    \widetilde{\mu}*\delta_x(\chi)&\leq \rho^c(\eps_d^c\prod_{s=1}^{d-1}\eps_s^{-c}\inj(x)^{-C}+1)+\sum_{k=1}^{d-1}\rho^c(\eps_{d-k}^c\eps_{d-k+1}^{-c}+1)\\
    &\leq C \rho^{c/2}(\eps_d^c\prod_{s=1}^{d-1}\eps_s^{-c}\inj(x)^{-C}+1),
\end{align*}
where the last inequality follows from the assumption that $\eps_{d-k}\eps_{d-k+1}^{-1}\leq 1$ for $1\leq k\leq d-1$, and provided that $\rho\lll 1$ and $C\ggg 1$. This finishes the proof.

\bigskip

\end{proof}

\bigskip

\section{Proof of Theorem \ref{thm: weighted effective equid}}\label{sec: proof of weighted equid}

The aim of this section is to apply Theorem \ref{thm: weighted effective recurrence} and the following unweighted effective equidistribution theorem to  prove Theorem \ref{thm: weighted effective equid}. Recall that given any $t>0$, we have defined the one-parameter diagonal flow
\begin{align}\label{eq: defn of one parameter}
    a(t)=a(t,\cdots,t)=\diag(t^{\tfrac{1}{d+1}},\cdots,t^{\tfrac{1}{d+1}},t^{-\tfrac{d}{d+1}}). 
\end{align}
\begin{thm}\label{thm: unweighted effective equid}
Let $\sigma$ be a product-type self-affine measure on $\R^d$. Then for every $x\in X$, $t>1$ and $f\in B^\infty_{\infty, l}(X)$, we have
\[\abse{\int_{\R^d} f\bigl(a(t) u(\bs) x\bigr) \dd \sigma(\bs) - m_{X}(f) } \leq C \cS_{\infty, l}(f) \inj(x)^{-1} t^{-c},\]
where the constant $C,c>0$ only depend on $\sigma$, and $l=\lceil \frac{1}{2}\dim \SO(d+1)\rceil$.
\end{thm}
Since $\sigma$ is not self-similar in general, Theorem \ref{thm: unweighted effective equid} does not follow from \cite[Theorem 1.2]{BHZ25}. However, the product structure of $\sigma$ enables us to adapt the methods employed in \cite{BHZ25} to prove Theorem \ref{thm: unweighted effective equid}. This will be done in Section \ref{sec: unweighted equid}.

\medskip

\begin{proof}[Proof of Theorem \ref{thm: weighted effective equid} assuming Theorem \ref{thm: unweighted effective equid}.]

To simplify the notation, we write $\cS(f)=\cS_{\infty,l}(f)$. We make  some reductions before proceeding to the proof. First we claim that it is enough to prove Theorem \ref{thm: weighted effective equid} for $\bt\in \R^d_{> 1}$ satisfying
\begin{align}\label{eq: special bt}
    t_1\geq \cdots\geq t_d
\end{align}
Indeed, suppose that the theorem has been shown for any self-affine measure $\sigma$ of product type and for any $\bt$ satisfying \eqref{eq: special bt}. Let $\bt\in \R^d_{>1}$. We can find a permutation $S:\{1,\cdots,d\}\to \{1,\cdots,d\}$ such that $S(\bt)$ satisfies \eqref{eq: special bt}. Here $S:\R^d\to \R^d$ is a  map induced by the permutation $S$:
\[ S(x_1,\cdots,x_d)=(x_{S(1)},\cdots, x_{S(d)}), \, \forall \,(x_1,\cdots,x_d)\in \R^d. \]
By abuse of notation, we also denote by $S$ the Weyl group element in $G$ corresponding to this permutation such that $S a(\bt)S^{-1}=a(S(\bt))$. Observe that if $\sigma$ is a self-affine measure of product type, so is $S_*\sigma$.
Then for every $x\in X$ and $f\in B_{\infty,l}^{\infty}(X)$ with $m_X(f)=0$, we have
\begin{align*}
    \abse{\int_{\R^d} f\bigl(a(\bt) u(\bs) x\bigr) \dd \sigma(\bs) }&=\abse{\int_{\R^d} f\bigl(S^{-1} a(S(\bt))S u(\bs) x\bigr) \dd \sigma(\bs) }\\
    &=\abse{\int_{\R^d} f\circ S^{-1}\bigl(a(S(\bt)) u(S(\bs)) Sx\bigr) \dd \sigma(\bs) }\\
    &=\abse{\int_{\R^d} f\circ S^{-1}\bigl(a(S(\bt)) u(\bs) Sx\bigr) \dd S_*\sigma(\bs) }\\
    &\leq C \cS(f\circ S^{-1}) \inj(Sx)^{-1} \integer{\bt}^{-c},
\end{align*}
for some $C,c>0$ depending on $S_*\sigma$. Here for the last inequality we use that $\integer{\bt}=\integer{S(\bt)}$. Then the theorem for $\bt\in \R^d_{>1}$ follows by noting that the Weyl group of $G$ is finite.

\medskip

Now we prove the theorem for $\bt\in \R^d_{>1}$ satisfying \eqref{eq: special bt}. Without loss of generality, we may assume $t_d>r_{min}^{-2d(d+1)}$ where $r_{min}$ is defined as in \eqref{eq: defn of rmin}. We decompose $a(\bt)$ into
\[ a(\bt)=a(t_d^{\tfrac{1}{2}})\cdot a(\bt)a(t_d^{-\tfrac{1}{2}}). \]
For our convenience, note that
\[a(\bt)a(t_d^{-\tfrac{1}{2}})=\diag(t_1^{\tfrac{1}{d}}t_d^{-\tfrac{1}{2(d+1)}},t_2^{\tfrac{1}{d}}t_d^{-\tfrac{1}{2(d+1)}},\cdots,t_d^{\tfrac{1}{2(d+1)}},t_d^{\tfrac{d}{2(d+1)}}\prod_{i=1}^dt_i^{-\tfrac{1}{d+1}}).  \]
Let $\epsilon_1,\cdots,\epsilon_d\in (0,1)$ be such that 
\begin{align}\label{eq: choice of eps i}
     \epsilon_1^{\tfrac{1}{d+1}}\cdots \epsilon_{k-1}^{\tfrac{1}{d+1}}\epsilon_k^{-\tfrac{d}{d+1}}\epsilon_{k+1}^{\tfrac{1}{d+1}}\cdots \epsilon_d^{\tfrac{1}{d+1}}=t_k^{\tfrac{1}{d}}t_d^{-\tfrac{1}{2(d+1)}}, \,\, \text{ for any } 1\leq k\leq d.
\end{align}
In particular, $\epsilon_1\cdots \epsilon_{d-1}\epsilon_d^{-d}=t_d^{\tfrac{1}{2(d+1)}}>r_{\min}^{-d}$. Moreover, as $\bt$ satisfies \eqref{eq: special bt}, \eqref{eq: choice of eps i} implies  that $0<\epsilon_1\leq \cdots \leq \epsilon_d<1$ and $(\epsilon_1,\cdots,\epsilon_d)$ satisfies \eqref{eq: condition on epsilon i} with $S$ being the identity permutation. Now we consider the measure 
\[ \mu=\mu_{1,d}^{\epsilon_1}*\cdots*\mu_{d,d}^{\eps_d} \, \text{ and }\, \lambda=\lambda_{1,d}^{\eps_1}*\cdots *\lambda_{d,d}^{\eps_d}. \]
Without loss of generality, we may assume $m_X(f)=0$. As $\sigma=\lambda*\sigma$ (see Lemma \ref{lemma: stationarity w.r.t stoppoing time}) and recall that $g=a_ga(\bt)a(t_d^{-1/2})u(\ttb_g)$ when $g\in \supp \mu$ for some $a_g\in A_{O(1)}$,  by \eqref{eq: choice of eps i} we have
\begin{align*}
    \int f(a(\bt)u(\bs)x)\dd \sigma(\bs)&=\int\int f(a(\bt)u(\phi(\bs))x)\dd \sigma(\bs)\dd \lambda(\phi)\\
    &=\int \int f(a_{g}^{-1} a(t_d^{\tfrac{1}{2}})u(\bs)gx) \dd \sigma(\bs) \dd\mu(g).
\end{align*}
Let $\rho=t_d^{-c_1}$ for some $c_1>0$ to be specified later. Consider 
\[ E_{\rho}=\{g: \inj(gx) <\rho\}. \]
By Theorem \ref{thm: weighted effective recurrence}, we have 
\[ \mu(E_{\rho})<\rho^{c_2}(1+t_d^{-c_2}\inj(x)^{-C}), \]
for some $C,c_2>0$ depending on $\sigma$. 
We first assume that $t_d^{-c_2}\inj(x)^{-C}\leq 1$. For any $g\in \supp \mu\setminus E_{\rho}$, applying Theorem \ref{thm: unweighted effective equid} we obtain some $c_3>0$ depending on $\sigma$ such that
\begin{align*}
    \abse{\int f(a_{g}' a(t_d^{\tfrac{1}{2}})u(\bs)gx)\dd \sigma(\bs)}\ll \cS(f\circ a'_g)\inj(gx)^{-1} t_d^{-c_3}\leq \cS(f)t_d^{-c_3/2},
\end{align*}
where the last inequality holds by taking $c_1<c_3/2$. Therefore, 
\begin{align*}
    \abse{ \int f(a(\bt)u(\bs)x)\dd \sigma(\bs)}&\leq \int_{E_{\rho}\cup E_{\rho}^c} \abse{\int f(a_g^{-1}a(t_d^{\tfrac{1}{2}})u(\bs)gx)\dd \sigma(\bs)} \dd \mu(g)\\
    &\ll \mu(E_{\rho})\cS(f)+t_d^{-c_3/2}\cS(f)\\
    &\ll (t_d^{-c_1c_2}+t_d^{-c_3/2})\cS(f)\\
    &\ll t_d^{-c_1'} \inj(x)^{-1}\cS(f),
\end{align*}
where $c_1'=\min\{c_1c_2,c_3/2\}$. When $t_d^{-c_2}\inj(x)^{-C}\geq  1,$
we have $t_d^{-c_2/C} \inj(x)^{-1}\geq 1$. Let $c_2'=c_2/C$, then
\[ \abse{ \int f(a(\bt)u(\bs)x)\dd \sigma(\bs)}\ll t_d^{-c_2'}\inj(x)^{-1}\cS(f)  \]
is trivially true. Taking $c=\min\{c_1',c_2'\}$ finishes the proof of Theorem \ref{thm: weighted effective equid} for $\bt$ satisfying \eqref{eq: special bt}, whence for any $\bt$ by the reduction in the beginning of the proof.

\end{proof}

\bigskip

\section{Proof of Theorem \ref{thm: unweighted effective equid}}\label{sec: unweighted equid}
In this section, we prove Theorem \ref{thm: unweighted effective equid}. To do this, we follow and adapt the three main steps in \cite{BHZ25}: positive dimension of initial random walks, dimension bootstrap, and from high dimension to equidistribution. Recall that for any $n_1\in \N$, we have defined a probability measure $\overline{\mu}_{n_1}$ in \eqref{eq: overline mu n1}. The main reason why we can adapt the methods in \cite{BHZ25} in the setting of this paper is that the random walk driven by $\overline{\mu}_{n_1}$ on $X$ is more or less similar to the one induced by a self-similar measure in \cite{BHZ25}. 

\medskip

\subsection{Positive dimension} We show that given $x\in X$, for sufficiently large $n_1,n\in \N$, the measure $\overline{\mu}_{n_1}^{*n}*\delta_x$ has certain Frostman-type positive dimension:
\begin{proposition}\label{prop: positive dimension}
There exist $A,\kappa>0$ such that for any $n_1\ggg_{\sigma} 1$, $\rho\lll_{n_1,\kappa,\sigma} 1$, $x,y\in X$,
\[ \overline{\mu}_{n_1}^{*n}*\delta_x (B_{\rho}y)\leq \rho^{\kappa}, \]
provided that $nn_1\geq |\log \rho|+A|\log \inj x|$. 
\end{proposition}
The proof of this proposition follows similarly from that of \cite[Proposition 6.1]{BHZ25}. 

\begin{proof}
Let $\rho\in (0,1)$ and assume that $nn_1\geq |\log \rho|$. Let $\kappa>0$ be a parameter to be specified later and assume by contradiction that $\overline{\mu}_{n_1}^{*n}*\delta_x(B_{\rho}y)>\rho^{\kappa}$. Let $\alpha=\tfrac{1}{10(1+\ell_1+\max_k \ell_1/\ell_k)}$ and $m=\integer{ \alpha|\log \rho|/n_1 }$. For any $z\in X$, define a set 
\[ E_z=\{g: gz\in B_{\rho}y\}. \]
Let $F=\{z: \overline{\mu}_{n_1}^{*m}(E_z)>\rho^{2\kappa}\}.$ Then we have
\begin{align*}
    \rho^{\kappa}\leq \overline{\mu}_{n_1}^{*n}*\delta_x(B_{\rho}y)&=\int \overline{\mu}_{n_1}^{*n}(E_z)\dd \overline{\mu}_{n_1}^{*(n-m)}*\delta_x(z)\\
    &\leq \overline{\mu}_{n_1}^{*(n-m)}*\delta_x(F)+\rho^{2\kappa}(1-\overline{\mu}_{n_1}^{*(n-m)}*\delta_x(F)),
\end{align*}
which implies that $\overline{\mu}_{n_1}^{*(n-m)}*\delta_x(F)>\rho^{2\kappa}$ provided that $\rho\lll_{\kappa} 1$.

Now we argue that $\inj(z)$ is very small for $z\in F$. Fix $z\in F$ in the rest of the argument. By the large deviation principle, there exists $\delta=\delta(\sigma)>0$ such that if $mn_1\ggg_{\sigma} 1$, then for any $k\in \dbracket{1,d}$,
\begin{align}\label{eq: LDP in pos dim}
     \mu_{k,d}^{*mn_k}\{g: |\log \ttr_g+mn_k\ell_k|<mn_k\}\geq 1-e^{-mn_k \delta}.
\end{align}
Assume $n_1\geq \max_{k}(\ell_k, 2\ell_k/\ell_1)$. By the assumption on $n_1$, if $|\log \ttr_g+mn_k\ell_k|<mn_k$, then $|\log \ttr_g+mn_1\ell_1|<mn_1(1+ \max_k \ell_1/\ell_k)$. Recall that for any $g\in \supp \overline{\mu}_{n_1}^{*m}$, we can write $g=\prod_1^d g_k$ where $g_k\in \supp \mu_{k,d}^{*mn_k}$.
Hence, by \eqref{eq: LDP in pos dim} we have
\[ \overline{\mu}_{n_1}^{*m}\underbrace{\{g:|\log \ttr_{g_k}+mn_1\ell_1|>mn_1 (1+ \max_k \ell_1/\ell_k), \, \exists k\in \dbracket{1,d}\}}_{E_1}\leq  e^{-mn_1\delta'}\leq \rho^{\alpha \delta'/2},\]
for $\delta'=\min_k \tfrac{\ell_1 \delta}{2\ell_k}$, where the last inequality holds if $\alpha|\log \rho|/2>n_1$.

Now let $C>1$ be a parameter to be specified later. Let $K$ be a fixed compact set such that $\supp \sigma\subset K$. We divide 
\[[-mn_1(\ell_1+1+\max_k \ell_1/\ell_k),-mn_1(\ell_1-1-\max_k \ell_1/\ell_k)]^d\times K\subset \R^{2d}\]
into disjoint cubes $(S_i=S_{i,1}\times S_{i,2})_{i\in I}$ of length $\rho^{C\kappa}$. We have 
\[ \#I \ll_{K}  (mn_1 (1+\max_k\frac{\ell_1}{\ell_k}))^d \rho^{-2dC\kappa}\ll \rho^{-3dC\kappa}, \]
provided that $\rho\lll_{\kappa} 1$.
By the pigeonhole principle, given $z\in F$, there exists at least one $i_0\in I$ such that 
\begin{align}\label{eq: bound for E2}
     \overline{\mu}^{*m}_{n_1}\underbrace{\{g\in E_z\setminus E_1: \log \ttr_g\in S_{i_0,1}, \, \ttb_g\in S_{i_0,2}\}}_{E_2}\geq \frac{\rho^{2\kappa}-\rho^{\alpha \delta'/2}}{\rho^{-3dC\kappa}}\geq \rho^{4dC\kappa}, 
\end{align}
provided $\kappa\lll_{\delta'} 1$. Let $\gamma\lll 1$ be the constant given by Corollary \ref{cor: regularity-stopping time for product} (2). Then \eqref{eq: bound for E2} implies there exist $g_1,g_2\in E_2$ such that
\begin{align}\label{eq: difference for ttb}
    \norm{\ttb_{g_2}-\ttb_{g_1}}\geq \rho^{4\gamma^{-1}dC\kappa},
\end{align}
provided that $\kappa\lll_{C,\gamma} 1$, which guarantees $\rho^{4\gamma^{-1}dC\kappa}\geq \rho^{\alpha/2}> e^{-mn_1}$. Note that the choice of $\alpha$ guarantees that $\norm{\Ad(g_1^{-1})}\leq \rho^{-1/2}$ when $g_1\in E_2$. This implies
\begin{align}\label{eq: bound for z g_1^{-1}g_2z}
    \dist(z,g_1^{-1}g_2 z)\leq \norm{\Ad(g_1^{-1})}\dist(g_1z,g_2z)\ll \rho^{1/2}.
\end{align}
On the other hand,
\[ g_1^{-1}g_2=u(-\ttb_{g_2}) u(\ttb_{g_2}-\ttb_{g_1}) a(\ttr_{g_1}\ttr_{g_2}^{-1})u(\ttb_{g_2}). \]
The fact that $g_1,g_2\in E_2$ and \eqref{eq: difference for ttb} yield
\begin{align}\label{eq: bound for g_1^-1 g_2}
    \rho^{1/4}<\rho^{4\gamma^{-1}dC\kappa} \leq \dist(g_1^{-1}g_2,\Id)\simeq \norm{\ttb_{g_2}-\ttb_{g_1}}+\max_{k\in \dbracket{1,d}}|\log(\ttr_{g_{1,k}}\ttr_{g_{2,k}}^{-1})| \leq 10\rho^{C\kappa}, 
\end{align}
where the first inequality holds provided that $\kappa\lll_{C,\gamma} 1$. Hence, \eqref{eq: bound for z g_1^{-1}g_2z} together with \eqref{eq: bound for g_1^-1 g_2} imply 
\[ \inj(z)\leq \rho^{1/2}+10\rho^{C\kappa}\leq 11\rho^{C\kappa}, \]
provided that $\kappa \lll_{C} 1$.
Therefore, $F\subset \{\inj\leq 11\rho^{C\kappa}\}$, and then by Corollary \ref{cor: effective recurrence of regular convolution} we have
\[ \rho^{2\kappa}<\overline{\mu}_{n_1}^{*(n-m)}*\delta_x\{\inj\leq 11\rho^{C\kappa}\}\leq \rho^{cC\kappa}(e^{-c(n-m)n_1}\inj(x)^{-C_1}+b), \]
for some $C_1,c>0$ and $b=b(n_1)>0$.
This is a contradiction if $C\ggg_c 1$, $(n-m)n_1\ggg |\log \inj x|$ and $\rho\lll_{n_1} 1$. This finishes the proof.
\end{proof}

\bigskip

\subsection{Dimension bootstrap}

The aim here is to prove  Proposition \ref{prop: high dimension} stating that after finitely many steps of convolutions by $\overline{\mu}_{n_1}$ for large enough $n_1\in \N$, a Borel measure on $X$ with positive initial dimension will gain dimension sufficiently close to that of the ambient space. We first recall the definition of robust measures in \cite{BH24,BHZ25,BH25}.

\begin{definition}\label{defn: robust measure}
   Given $\alpha\in (0,1]$, $\tau\geq 0$ and $I\subset [0,1]$. Let $\nu$ be a Borel measure on $X$. Then $\nu$ is said to be $(\alpha,\mathcal{B}_I,\tau)$-robust if $\nu=\nu_1+\nu_2$, where $\nu_2(X)\leq \tau$, $\nu_1$ satisfies $\nu_1\{\inj<\sup I\}=0$ and
    \[ \nu_1(B_{\rho}y)\leq \rho^{\alpha\dim X},\,\, \forall \rho\in I, \, y\in X. \]
We say that $\nu$ is $(\alpha,\mathcal{B}_{\rho},\tau)$-robust if $I=\{\rho\}$ is a singleton.
\end{definition}

\begin{proposition}\label{prop: high dimension}
Given $0<\kappa\lll 1$. There exist $M\ggg_{\kappa} 1$ and $A\ggg_{\sigma} 1$ such that given any $n_1\ggg_{\kappa} 1$, $\eta\lll_\kappa 1$, $\rho\lll_{n_1} 1$ and $x\in X$, if $n\in \N$ satisfies  $nn_1\geq M |\log \rho|+A|\log \inj x|$, then the measure $\overline{\mu}_{n_1}^{*n}*\delta_x$ is $(1-\kappa,\mathcal{B}_{\rho},\rho^{\eta})$-robust.
\end{proposition}

We follow the strategy of \cite{BHZ25} to prove Proposition \ref{prop: high dimension}. The dimension increment in Proposition \ref{prop: high dimension} relies on a generalized multi-slicing theorem of B\'enard and He \cite{BH24,BH25}, where a mild non-concentration condition $(\MNC)$ plays a crucial role. 

\medskip

Let $D\geq 2$ and $1\leq k\leq D-1$ be integers. Denote by $\Gr(\R^D,k)$ the Grassmannian of all $k$-dimensional linear subspaces of $\R^D$. Note that $\Gr(\R^D,k)$ can be viewed as a subset of the vector space $\bigwedge^k \R^D$ endowed with a norm $\norm{\cdot}$. Given $F_i\in \Gr(\R^D,k_i)$ for some $k_i\in \N$ ($i=1,2$) such that $k_1+k_2\leq D$, we denote $\norm{F_1\wedge F_2}=\norm{v_1\wedge v_2}$ where $v_i\in \bigwedge^{k_i} \R^D$ is the representative of $F_i$ with $\norm{v_i}=1$.
\begin{definition}\label{defn: MNC}
     Let $\Xi$ be a probability measure on $\Gr(\R^D,k)$. Given parameters $C,c,\rho>0$, $\Xi$ is said to satisfy the \textit{mild non-concentration property} (MNC) with parameters $(\rho,C,c)$ if there exist $q,m\in \N$ such that $D=qk+m$ and for every $W\in \Gr(\R^D,m)$, $r\geq \rho$,
    \[ \Xi^{\otimes q}\{(F_i)_1^q: \norm{F_1\wedge\cdots \wedge F_q\wedge W}\leq r\}\leq C r^c. \]
$\Xi$ satisfies $(\MNC)^{\perp}$ with parameters $(\rho,C,c)$ if the push-forward measure of $\Xi$ under $F\mapsto F^{\perp}$ satisfies $(\MNC)$ with parameters $(\rho,C,c)$. 
\end{definition}

Specified to our setting, we consider the Lie algebra $\mathfrak{g}=\Lie(G)\cong\R^D$ with $D=d^2+2d$. Let $A'=\{a(t):t>0\}$. We can decompose  $\mathfrak{g}=\mathfrak{g}_-\oplus \kg_0\oplus \kg_+$ into sum of eigenspaces with eigenvalue greater than $1$, equal to $1$ and less than $1$ respectively with respect to the action of $a(t)$ ($t>1$). More precisely, $\kg_+=\Lie U$, $\kg_-=\Lie U^-$ where $U^-$ is the transpose of $U$ and $\kg_0$ is the Lie algebra of the centralizer of $A'$. Denote by $\kg_{\leq 0}=\kg_-\oplus \kg_0$. Note that $\dim \kg_+=d$ and $\dim \kg_{\leq 0}=d^2+d$. 

Let $P'=A'U$. We consider maps $V_1:P'\to \Gr(\kg,d)$ and $V_2:P'\to\Gr(\kg,d^2+d)$ defined respectively by
\begin{align*}
     V_1(g)=\Ad(u(-\ttb_g))\kg_-,\quad  V_2(g)=\Ad(u(-\ttb_g))\kg_{\leq 0}.
\end{align*}

\begin{proposition}\label{prop: MNC}
There exist parameters $C,c>0$ depending on $\sigma$ only such that the following holds. Given $n_1,n\in \N$ such that $n_1\ggg_{\sigma} 1$. Let $\Xi_1$ be the push-forward probability measure on $\Gr(\kg,d)$  of $\overline{\mu}_{n_1}^{*n}$ under the map $g\mapsto V_1(g)$. Then $\Xi_1$ satisfies $(\MNC)$ with parameters $(e^{-nn_1},C,c)$. 

Let $\Xi_2$ be the push-forward probability measure on $\Gr(\kg, d^2+d)$  of $\overline{\mu}_{n_1}^{*n}$ under the map $g\mapsto V_2(g)$. Then $\Xi_2$ satisfies $(\MNC)^{\perp}$ with parameters $(e^{-nn_1},C,c)$.
\end{proposition}

\begin{proof}
For the first assertion, it suffices to show that for any $W\in \Gr(\kg,d)$ and $r\geq e^{-nn_1}$, we have
\[ (\overline{\mu}_{n_1}^{*n})^{\otimes d+1}\{(g_i)_{i=1}^{d+1}:\, \norm{V_1(g_1)\wedge\cdots \wedge V_1(g_{d+1})\wedge W}\leq  r\}\leq C r^c, \]
for some constants $C,c>0$ depending only on $\sigma$. For this, invoking \cite[Proposition 7.6]{BHZ25}, there exists $(u_i)_{i=1}^{d+1}\in U^{d+1}$ such that 
\[ \Ad(u_{1})\kg_{-}\oplus\Ad(u_{2})\kg_{-}\oplus\dots \oplus \Ad(u_{d+1})\kg_{-}\oplus W=\kg. \]
Hence, the map 
\[ P_W\left((\bs_i)_{i=1}^{d+1} \right)=\norm{\Ad(u(-\bs_1))\kg_-\wedge\cdots\wedge \Ad(u(-\bs_{d+1}))\kg_-\wedge W  }^2 \]
is a non-zero polynomial of bounded degree. As $W$ varies in a compact set, we have $\norm{P_W}\simeq 1$ for any $W\in \Gr(\kg,d)$. By Corollary \ref{cor: regularity for product 2}, we have
\begin{align}\label{eq: bound in MNC}
    (\overline{\mu}_{n_1}^{*n})^{\otimes d+1}\{(g_i)_{i=1}^{d+1}: |P_W((\ttb_{{g_i}})_1^{d+1})|\leq r\}\leq C r^c,
\end{align}
provided that $r\geq \max_{k\in \dbracket{1,d}} e^{-nn_k}$. Note that $n_k\geq n_1\ell_1/\ell_k -1$ for any $k\in \dbracket{1,d}$. Provided $n_1\geq \max_k 2\ell_k/\ell_1$, we have $n_k\geq \tfrac{n_1\ell_1}{2\ell_k}$ for any $k\in \dbracket{1,d}$. Hence, if $r\geq e^{-nn_1\min_k\tfrac{\ell_1}{2\ell_k}}$, then \eqref{eq: bound in MNC} holds, whence the first assertion by possibly replacing $c$ by $c\min_k\tfrac{\ell_1}{2\ell_k}$.

Using $(\MNC)$ of $\Xi_1$, the proof of the second assertion on $\Xi_2$ then follows from that in the proof of \cite[Corollary 7.5]{BHZ25}.
\end{proof}

\medskip

Let $\mathcal{F}$ be the collection of partial flags $\sV=(V_1,V_2)$ with $V_1\in \Gr(\kg, d)$ and $V_2\in \Gr(\kg, d^2+d)$ such that $V_1\subset V_2$. Given a vector $\bt=(t_1,t_2,t_3)\in \R^3$ such that $0\leq t_1<t_2<t_3\leq 1$, $\rho\in (0,1)$ and  $\sV=(V_1,V_2)\in \mathcal{F}$, we denote 
\[B^{\sV}_{\rho^{\bt}}=B^{V_1}_{\rho^{t_1}}+ B^{V_2}_{\rho^{t_2}}+ B^{\kg}_{\rho^{t_3}},\]
which represents a Euclidean box carried by $\sV$. We record the following supercritical multislicing theorem from \cite{BHZ25}: 

\begin{proposition}\cite[Proposition 7.14]{BHZ25} \label{prop: multislicing}
Let $\mathcal{F}$, $\bt$ be as above and $D=d^2+2d$. Let $c, \eps,\rho >0$.

Let $\Xi$ be a probability measure on $\cF$. Assume that for each $i=1,2$, the distribution of the component $V_{i}$ as $\sV$ distributed with law $\Xi$ satisfies either $(\MNC)$ or $(\MNC)^\perp$ with parameters $(\rho, \rho^{-\eps}, c)$.

Let $\nu$ be a Borel measure on $B^{\kg}_{1}$ of mass at most $\rho^{-\eps}$, and such that for some $\alpha\in [c, 1-c]$, for all $v\in \kg$, all $r\in [\rho^{t_3}, \rho^{t_1}]$, we have
$$\nu(B^{\kg}_{r}+v)\leq\rho^{-\eps} r^{D\alpha}.$$

If $\eps, \rho\lll_{D, \bt, c} 1$, then there exists an event $\cE \subseteq \cF$ such that $\Xi(\cE)\leq \rho^{\eps}$ and for $\sV\in\cF \setminus \cE$, there is a set $A_{\sV} \subseteq \kg$ with $\nu(\kg \setminus A_{\sV}) \leq \rho^{\eps}$ and such that for every $v \in \kg$,
\[ \nu_{| A_{\sV}} \left( B^{\sV}_{\rho^{\bt}}+v \right)\leq \leb\left( B^{\sV}_{\rho^{\bt}} \right)^{\alpha+\eps} . \]

\end{proposition}

\medskip

The linearizing lemma below is useful when we transfer dimension of certain measure between $\kg$ and $X$:
\begin{lemma}[{\cite[Lemma 6.3]{BH25}}]\label{lemma: linearizing}
Let $0<\delta \lll1$. There exists a measurable map $\varphi : \{\inj \geq \delta\}\rightarrow B^\kg_{1}$ satisfying the following.
\begin{itemize}
\item[1)] For every $r\in (0, \delta)$, $v\in \kg$, the preimage $\varphi^{-1}(B^\kg_{r}+ v)$ is covered by $O(1)$ many balls $(B_{r}x)_{x\in X}$.

\item[2)] For every $r\in (0, \delta)$, $g\in G$ such that $B^\kg_{\delta^2}\subseteq \Ad(g)B^\kg_{r}\subseteq B^\kg_{\delta}$, and $x\in X$, the translate $gB_{r}x \cap \{\inj \geq \delta\}$ is covered by $O(1)$ many preimages of boxes $(\varphi^{-1}(\Ad(g)B^\kg_{r}+ v))_{v\in \kg}$.
\end{itemize}
\end{lemma}

\medskip

We prove the following dimension increment statement by adapting the proof of \cite[Proposition 7.16]{BHZ25}.

\begin{proposition}\label{prop: dimension bootstrap}

Given parameters $\kappa,\epsilon,\rho\in (0,1/10)$, $\alpha\in [\kappa,1-\kappa]$ and $\tau\geq 0$. Let $\nu$ be a Borel measure on $X$ of mass at most $1$ and $n_1\ggg_{\sigma} 1$. Assume that $\nu$ is $(\alpha,\mathcal{B}_{[\rho^{2/3},\rho^{1/3}]},\tau)$-robust and $\supp \nu \subset \{\inj\geq \rho^{\epsilon}\}$. If $\epsilon,\rho\lll_{\kappa} 1$ and $n_1\geq \max_{k\in \dbracket{1,d}} \ell_k/\epsilon$, then $\overline{\mu}_{n_1}^{*n_{\rho}}*\nu$ is $(\alpha+\epsilon,\mathcal{B}_{\rho^{1/2}}, \tau+\rho^{\epsilon})$-robust for $n_{\rho}=\integer{\tfrac{|\log \rho|}{10n_1 \ell_1}}$.

\end{proposition}

\begin{proof}
Without loss of generality, we assume that $\tau=0$. We write $n=n_{\rho}$ for simplicity. Applying Corollary \ref{cor: effective recurrence of regular convolution}, we have 
\begin{align*}
    \overline{\mu}_{n_1}^{*n}*\nu\{\inj\leq \rho^{1/2}\}=\int \overline{\mu}_{n_1}^{*n}*\delta_x\{\inj\leq \rho^{1/2}\} \dd \nu(x)\leq \rho^{c/2}(e^{-nn_1 c}\rho^{-C\epsilon}+b(n_1)),
\end{align*}
for some $C,c>0$ depending only on $\sigma,G$. If $\epsilon\lll_{C,c,\ell_1} 1$, we then have 
\[ \overline{\mu}_{n_1}^{*n}*\nu\{\inj\leq \rho^{1/2}\}\leq \frac{\rho^{\epsilon}}{2}. \]
Hence, it is enough to show that $\overline{\mu}_{n_1}^{*n}*\nu=\nu'+\nu''$ for Borel measures $\nu'$ and $\nu''$ satisfying $\nu''(X)\leq \rho^{\epsilon}/2$ and 
\begin{align}\label{eq: noncon of nu'}
    \sup_{y\in X} \nu'(B_{\rho^{1/2}}y)\leq \rho^{\tfrac{1}{2} (\alpha+\epsilon)\dim X}. 
\end{align}
Indeed, the proposition follows by setting $\nu_1=\nu'|_{\{\inj\geq \rho^{1/2}\}}$ and $\nu_2=\nu''+\nu'|_{\{\inj<\rho^{1/2}\}}$.

To do this, by the large deviation principle, there exists $\gamma_1=\gamma_1(\sigma,\epsilon)>0$ such that for any $k\in \dbracket{1,d}$,
\begin{align*}
    \mu_k^{*nn_k}\{g: |\log \ttr_g+nn_k\ell_k|>nn_k\epsilon\}<e^{-\gamma_1 nn_k}.
\end{align*}
As $n_1>\max_{k\in \dbracket{1,d}}\ell_k/\epsilon$, if 
\[ |\log \ttr_g+nn_k\ell_k|\leq nn_k\epsilon,\]
then 
\[ |\log \ttr_g+nn_1\ell_1|\leq (1+\max_{k\in \dbracket{1,d}} \ell_1/\ell_k) n n_1\epsilon. \]
Hence, by replacing $\gamma_1$ by $\gamma_1 \ell_1/4$ and requiring $n\ggg_{\gamma_1} 1$, letting
\[E_1=\{g: |\log \ttr_{g,k}+nn_1\ell_1|>(1+\max_{k\in \dbracket{1,d}} \ell_1/\ell_k) n n_1\epsilon, \,\text{ for some }k\in \dbracket{1,d}\}, \]
we have $ \overline{\mu}_{n_1}^{*n}(E_1)< e^{-\gamma_1 n n_1}$. Here we write $g=a(\ttr_g)u(\ttb_g)$ for $\bttr_g=(\ttr_{g,k})_{k=1}^d$ and $g\in \supp \overline{\mu}^n_{n_1}$. Let $\bt=(\tfrac{2}{5},\tfrac{1}{2},\tfrac{3}{5})$. For any $g\in \supp \overline{\mu}_{n_1}^{*n}\setminus E_1$,  we can write $g=a_g a(e^{nn_1\ell_1})u(\ttb_g)$ for some $a_g\in A_{\rho^{-O(\eps)}}$, hence
\begin{align}\label{eq: close companion}
\Ad(g^{-1})B^{\kg}_{\rho^{1/2}}&=\Ad(u(-\ttb_g))\Ad(a_g^{-1}a(e^{-nn_1\ell_1}))B^{\kg}_{\rho^{1/2}}\nonumber\\
    &\overset{O(1)}{\simeq} \Ad(u(-\ttb_g))\Ad(a_g^{-1})\left(B^{\kg_-}_{\rho^{2/5}} \oplus B^{\kg_0}_{\rho^{1/2}}\oplus B^{\kg_+}_{\rho^{3/5}} \right)\nonumber\\
    &\overset{\rho^{-O(\epsilon)}}{\simeq}\Ad(u(-\ttb_g))\left(B^{\kg_-}_{\rho^{2/5}} \oplus B^{\kg_0}_{\rho^{1/2}}\oplus B^{\kg_+}_{\rho^{3/5}} \right)\nonumber\\
    &\overset{\rho^{-O(\epsilon)}}{\simeq}  B^{\Ad(u(-\ttb_g))\kg_-}_{\rho^{2/5}}+ B^{\Ad(u(-\ttb_g))\kg_{\leq 0}}_{\rho^{1/2}}+ B^{\kg}_{\rho^{3/5}}=B^{\sV_g}_{\rho^{\bt}},
\end{align}
where $a'_g\in \mathrm{A}_{\rho^{-O(\epsilon)}}$, and the third relation follows from the fact that for any $a\in \mathrm{A}$, $\Ad(a)\kg_*=\kg_*$ for $*=-,0,+$. Here for two sets $A,B\subset \kg$, $A\overset{\rho^{-O(\eps)}}{\simeq} B$ means that $A$ can be covered by $\rho^{-O(\eps)}$-many translates of $B$ and vice versa.

Consider the partial flag $\mathcal{F}$ on $\kg$ consisting of $\sV=(V_1,V_2)$ for $V_1\in \Gr(\kg,d)$, $V_2\in \Gr(\kg,d^2+d)$ such that $V_1\subset V_2$. 
Consider the map from $\supp \overline{\mu}_{n_1}^{*n}$ to $\mathcal{F}$ given by
\[ g\mapsto \sV_g=(V_1(g),V_2(g))=(\Ad(u(-\ttb_g)\kg_-, \Ad(u(-\ttb_g))\kg_{\leq 0}). \]
Let $\Xi$ be the push-forward probability measure of $\overline{\mu}_{n_1}^{*n}$ under this map. Then Proposition \ref{prop: MNC} guarantees that the distribution of $V_1(g)$ and $V_2(g)$ for $g$ obeying the law $\overline{\mu}_{n_1}^{*n}$ satisfies $(\MNC)$ and $(\MNC)^{\perp}$ with parameters $(e^{-nn_1},C,c )$ respectively. Let $\varphi: \{\inj\geq \rho^{1/3}\}\to B^{\kg}_1$ be the map given by Lemma \ref{lemma: linearizing}, and $\tilde{\nu}=\varphi_*(\nu)$. As $\nu$ is $(\alpha,\mathcal{B}_{[\rho^{2/3},\rho^{1/3}]},0)$-robust, for any $r\in [\rho^{3/5},\rho^{2/5}]\subset [\rho^{2/3},\rho^{1/3}]$, we have
\[ \sup_{y\in X} \nu(B_r y)\leq r^{\alpha\dim X}. \]
By Lemma \ref{lemma: linearizing} (1),
\[ \tilde{\nu}(B^{\kg}_r+v)\leq \rho^{-\epsilon} r^{\alpha\dim X}, \,\, \forall v\in \kg. \]
We now apply the multislicing Proposition \ref{prop: multislicing} to the partial flag $\cF$ endowed with the measure $\Xi$. If $\epsilon_0,\rho\lll_{c} 1$, then there exists $\cE\subset \supp \overline{\mu}_{n_1}^{*n}$ with $\overline{\mu}_{n_1}^{*n}(E_2)\leq \rho^{\epsilon_0}$, and for any $g\in \supp \overline{\mu}_{n_1}^{*n}\setminus \cE$, there exists a subset $\tilde{A_g}\subset \kg$ with $\tilde{\nu}(\kg\setminus \tilde{A_g})\leq \rho^{\epsilon_0}$ such that 
\[ \tilde{\nu}_{|\tilde{A_g}}(B^{\sV_g}_{\rho^{\bt}}+v)\leq \leb(B^{\sV_g}_{\rho^{\bt}})^{\alpha+\epsilon_0}=\rho^{\tfrac{1}{2}(\alpha+\epsilon_0)\dim X}, \, \forall v\in \kg. \]
As $B^{\sV_g}_{\rho^{\bt}}\overset{\rho^{-O(\epsilon)}}{\simeq} \Ad(g^{-1})B^{\kg}_{\rho^{1/2}}$ for $g\in \supp\overline{\mu}_{n_1}^{*n}\setminus E_1$ guaranteed by \eqref{eq: close companion}, we have
\[ \tilde{\nu}_{|\tilde{A_g}}(\Ad(g^{-1})B^{\kg}_{\rho^{1/2}}+v)\leq \rho^{\tfrac{1}{2}(\alpha+\epsilon_0-O(\eps))\dim X}. \]
Put $A_g=\varphi^{-1}(\tilde{A_g})$. As $B^{\kg}_{\rho^{2/3}}\subset \Ad(g^{-1})B^{\kg}_{r}\subset B^{\kg}_{\rho^{1/3}}$ for any $g\in \supp\overline{\mu}_{n_1}^{*n}\setminus E_1$ and $r\in [\rho^{3/5},\rho^{2/5}]$ provided that $\epsilon\lll 1$, by Lemma \ref{lemma: linearizing} (2) we have
\begin{align}\label{eq: noncon of nu}
    \nu_{|A_g}(g^{-1}B_{\rho^{1/2}}y)\leq \rho^{-\epsilon}\rho^{\tfrac{1}{2}(\alpha+\epsilon_0-O(\eps))\dim X}\leq \rho^{\tfrac{1}{2}(\alpha+\epsilon_0-O(\eps))\dim X}\leq \rho^{\tfrac{1}{2}(\alpha+\epsilon)\dim X},
\end{align}
provided that $\epsilon\lll \eps_0$.
At the end, we define
\[ \nu'=\int_{G\setminus E_1\cup E_2} g_{\star}\nu_{|A_g}\dd \overline{\mu}_{n_1}^{*n}(g). \]
Then the non-concentration property of $\nu'$ in \eqref{eq: noncon of nu'} follows from \eqref{eq: noncon of nu}. The proof is complete.
\end{proof}

\medskip

We invoke the following lemma:
\begin{lemma}\cite[Lemma 4.8]{BHZ26}\label{Lemma: combinatorial robust} 
    Let $\alpha,s,\rho\in (0,1]$, $\tau\in \mathbb{R}^+$ be parameters. If $\nu$ is $(\alpha,\mathcal{B}_{r},\tau)$-robust for all $r\in [\rho,\rho^s]$, then for any $\eps\in (0,\alpha)$, the measure $\nu$ is $(\alpha-\eps,\mathcal{B}_{[\rho,\rho^s]},\lceil \frac{\log s}{\log(1-\eps)}\rceil \tau)$-robust.
\end{lemma}

\medskip

\begin{proof}[Proof of Proposition \ref{prop: high dimension}.]

Applying Corollary \ref{cor: effective recurrence of regular convolution}, there exist $A,c>0$ such that for any $n_1\ggg_{\sigma} 1$, $x,y\in X$, and $M\geq 1$, if $nn_1 \geq A|\log \inj (x)|$ then
\begin{align}\label{eq: nondiv bound}
\overline{\mu}_{n_1}^{*n}*\delta_x\{\inj\leq \rho^{1/M}\}\leq \rho^{c/M}(1+b(n_1))\leq \rho^{2\kappa/M}(1+b(n_1)),
\end{align}
where the last inequality holds provided $\kappa<c/2$. Applying Proposition \ref{prop: positive dimension}, provided that $\kappa\lll 1$ we obtain for any $r\in [\rho^{M},\rho^{1/M}]$,
\begin{align}\label{eq: posdim in a range}
\overline{\mu}_{n_1}^{*n}*\delta_x(B_r y)\leq r^{\kappa \dim X}, 
\end{align}
provided that $nn_1\geq M|\log \rho|+A|\log \inj(x)|$ and $\rho\lll_{\kappa,\sigma,n_1} 1$. We now fix $\kappa\lll 1$ such that both \eqref{eq: nondiv bound} and \eqref{eq: posdim in a range} hold. 

Let $\epsilon_0,\rho_0\in (0,1/2)$ be constants depending only on $\kappa$ such that conclusions in Proposition \ref{prop: dimension bootstrap} hold for all $\alpha\in [\kappa,1-\kappa]$, $\epsilon\leq \epsilon_0$, $\rho\leq \rho_0$ and $n_1\geq \max_{k\in \dbracket{1,d}}\ell_k/\epsilon_0$. Fix $\epsilon=\eps_0/2$, $K=\integer{ \tfrac{1-2\kappa}{\epsilon}}+1$ and $M=(1+\tfrac{3}{10\ell_1})(\tfrac{3}{2})^K$. We may further assume $\eps_0<\kappa$. We fix $n_1\geq \max_{k\in \dbracket{1,d}}\ell_k/\epsilon_0$ and assume $\rho\lll_{\kappa,M,\epsilon} 1$ to ensure that $b(n_1)\leq \rho^{-\kappa/M}$. Hence, \eqref{eq: nondiv bound} and \eqref{eq: posdim in a range} imply that $\overline{\mu}_{n_1}^{*n}*\delta_x$ is $(\kappa,\mathcal{B}_{[\rho^{M},\rho^{1/M}]}, \rho^{\epsilon_0/M})$-robust if $nn_1\geq M|\log \rho|+A|\log \inj(x)|$.

\medskip

Now we aim to prove that for any $0\leq k\leq K$, if 
\begin{align}\label{eq: bound tk}
     nn_1\geq t_k:=(1+\frac{k}{2\ell_1})M|\log\rho|+A|\log\inj(x)|, 
\end{align}
then the measure
\begin{align}\label{eq: induction for dim}
\overline{\mu}_{n_1}^{*n}*\delta_x\,  \text{ is }\, (\kappa+k\epsilon,\mathcal{B}_{[\rho^{(3/4)^kM}, \rho^{(3/2)^k/M}]}, O_{\eps,k,M}(\rho^{\eps/M}))\text{-robust}.
\end{align}
Specifying to the case $k=K$, we have
\[ \overline{\mu}_{n_1}^{*n}*\delta_x\,  \text{ is }\, (1-\kappa,\mathcal{B}_{[\rho^{(3/4)^KM}, \rho^{(3/2)^K/M}]}, O_{\eps,K,M}(\rho^{\eps/M}))\text{-robust}. \]
The choice of $M$ guarantees that $\rho\in [\rho^{(3/4)^KM}, \rho^{(3/2)^K/M}]$, whence the proposition.

We prove \eqref{eq: bound tk} implying \eqref{eq: induction for dim} by induction. The base case $k=0$ follows by \eqref{eq: nondiv bound} and \eqref{eq: posdim in a range} as mentioned above. Assume the hypothesis holds for some $0\leq k\leq K-1$. Assume $nn_1\geq t_{k+1}$. Given $r\in [\rho^{\tfrac{3^{k+1}M}{2\cdot 4^k}},\rho^{\tfrac{3^k}{2^k M}}]$, let $n'=\integer{ \tfrac{|\log r|}{10n_1\ell_1} }$. Then $n'n_1\leq \frac{M}{2\ell_1}|\log \rho|$. Hence, the assumption on $n$ ensures that
\[ \overline{\mu}_{n_1}^{*(n-n')}*\delta_x \,  \text{ is }\, (\kappa+k\epsilon,\mathcal{B}_{[\rho^{(3/4)^kM}, \rho^{(3/2)^k/M}]}, O_{\eps,k,M}(\rho^{\eps/M}))\text{-robust}. \]
As $[r^{3/2},r^{1/3}]\subset [\rho^{(3/4)^k}M, \rho^{(3/2)^k/M}]$, we have
\[\overline{\mu}_{n_1}^{*(n-n')}*\delta_x \,  \text{ is }\, (\kappa+k\epsilon,\mathcal{B}_{[r^{2/3},r^{1/3}]}, O_{\eps,k,M}(\rho^{\eps/M}))\text{-robust}.\]
Note that Corollary \ref{cor: effective recurrence of regular convolution} implies that
\[ \overline{\mu}_{n_1}^{*(n-n')}*\delta_x\{\inj\leq r^{\epsilon_0}\}\leq \rho^{\epsilon_0/M}, \]
provided that $M\ggg_{c} 1$ for $c>0$ in  Corollary \ref{cor: effective recurrence of regular convolution} (which can be guaranteed by taking $\eps\lll_c 1$ in the beginning). 
This together with the fact that $\kappa+k\epsilon\leq 1-\kappa$ enables us to apply Proposition \ref{prop: dimension bootstrap} and the induction hypothesis to obtain that
\[ \overline{\mu}_{n_1}^{*n'}*(\overline{\mu}_{n_1}^{(n-n')}*\delta_x)|_{\{\inj\geq r^{\eps_0}\}} \,\text{ is }\, (\kappa+(k+2)\eps),\mathcal{B}_{r^{1/2}}, O_{\eps,k+1,M}(\rho^{\eps/M}))\text{-robust}. \]
Applying Lemma \ref{Lemma: combinatorial robust}, 
\[ \overline{\mu}_{n_1}^{n}*\delta_x \,\text{ is }\, (\kappa+(k+1)\eps),\mathcal{B}_{[\rho^{(3/4)^{k+1}M}, \rho^{(3/2)^{k+1}/M}]}, O_{\eps,k+1,M}(\rho^{\eps/M}))\text{-robust}. \]
This finishes the proof of the induction, whence the proposition.
\end{proof}

\medskip

\subsection{From high dimension to equidistribution} Recall the definition of the multi-parameter (one-parameter, resp.) diagonal flow $a(\bt)$ for $\bt\in \R^d_{>0}$ in \eqref{eq: defn of multi parameter} ($a(t)$ for $t\in \R_{>0}$ in \eqref{eq: defn of one parameter}, resp.).
Given any $\bt\in \R^d_{>0}$ and $t\in \R_{>0}$, we define operators $\eta_{\bt}:L^2(X)\to L^2(X)$ by 
\[ \eta_{\bt}(f)(x)=\int f(a(\bt)u(\bs)x)\dd \sigma(\bs), \]
and $\eta_t:L^2(X)\to L^2(X)$ by 
\[ \eta_t(f)(x)=\int f(a(t)u(\bs)x)\dd \sigma(\bs). \]

We invoke the following proposition:

\begin{proposition}\label{prop: spectral gap}\cite[Proposition 8.1]{BHZ25}
There exist $\kappa_0,\rho_0>0$ such that the following holds for all $\rho\in (0,\rho_0]$, $\kappa\in (0,\kappa_0]$ and $\tau\geq 0$.

Let $\nu$ be a $(1-\kappa,\mathcal{B}_{\rho},\tau)$-robust measure on $X$ such that $\nu(X)\leq 1$. Let $l=\lceil \tfrac{1}{2} \dim\SO_{d+1} \rceil$. Then for all $t\in [\rho^{-1/4},\rho^{-1/2}]$ and any $f\in B^{\infty}_{\infty,l}(X)$ with $m_X(f)=0$, we have
\[ \abse{\eta_t*\nu(f)}\leq (\rho^{\kappa}+\tau)\cS_{\infty,l}(f). \]
    
\end{proposition}

\begin{proof}
Using the non-concentration property of $\sigma$ as in Corollary \ref{cor: regularity-stopping time for product} (3), the proof of this proposition is identical to that of 
\cite[Proposition 8.1]{BHZ25}.

\end{proof}

Note that any $g\in \supp \overline{\mu}_{n_1}^{*n}$ can be written as $g=a(e^{nn_1\ell_1}\bt_g) u(\ttb_g)$ for some $\bt_g\in \R^d_{>0}$.

\begin{lemma}\label{lemma: expanding trajectory and random walk}
For any $t>0$ and $n,n_1\in \N$, we have 
\[ \eta_t=\int \eta_{t e^{-nn_1\ell_1}\bt_g^{-1}}*\delta_g \dd \overline{\mu}_{n_1}^{*n}(g),\]
where we write $\bt_g^{-1}=(t_{g,1}^{-1},\cdots,t_{g,d}^{-1})$ for $\bt_g=(t_{g,1},\cdots,t_{g,d})$.
\end{lemma}

\begin{proof}
By definition of $\eta_t$ and that $\sigma=\lambda_{1,d}^{*nn_1}*\cdots*\lambda_{d,d}^{*nn_k}*\sigma$, we have for any $f\in L^2(X)$ and $x\in X$,
\begin{align*}
    \eta_t(f)(x)&=\int\int f(a(t)u(\phi(\bs))x)\dd \sigma(\bs)\dd\lambda_{1,d}^{*nn_1}*\cdots*\lambda_{d,d}^{*nn_k}(\phi)\\
    &=\int \int f(a(te^{-nn_1\ell_1}\bt_g^{-1})u(\bs)gx)\dd\sigma(\bs)\dd\overline{\mu}_{n_1}^{*n}(g).
\end{align*}
The lemma follows.
    
\end{proof}

\medskip

\begin{proof}[Proof of Theorem \ref{thm: unweighted effective equid}]

Note that for any $x\in X$, $f\in B^{\infty}_{\infty,l}(X)$ and $\bt,\bs\in \R^d_{>0}$, we have
\begin{align}\label{eq: ued eq-1}
    \abse{ f(a(\bt)x)-f(a(\bs)x) }\ll_d \max_i|\log t_i s_i^{-1}|\cS_{\infty,1}(f).
\end{align}
Recall that by Proposition \ref{prop: high dimension}, for any $\kappa>0$ small enough, there exists $C\ggg_{\kappa}1$ such that for any $n_1\ggg_{\kappa} 1$, any $\eta,\rho\lll_{\kappa} 1$, as long as $nn_1>C( |\log \rho|+|\log \inj x|)$, the probability measure $\nu:=\overline{\mu}_{n_1}^{*n}*\delta_x$ is $(1-\kappa,\mathcal{B}_{\rho},\rho^{\eta})$-robust. We fix this $\kappa$ and $\eta$ from the beginning.

Let $\gamma,\alpha>0$ be small parameters to be determined later.  Given $t\ggg 1$, we let 
\[ t=\rho^{-\tfrac{3}{8}-2C\ell_1} \,\text{ and } \, n=\lfloor \frac{2C|\log \rho|}{n_1} \rfloor.\]
Let us first assume that $|\log \inj (x)|<|\log \rho|/2$, or equivalently $|\log \inj x|<\tfrac{1}{3/4+4C\ell_1} |\log t|$. Then $\nu=\overline{\mu}_{n_1}^{*n}*\delta_x$ is $(1-\kappa,\mathcal{B}_{\rho},\rho^{\eta})$-robust. Consider the discrete set
\[ \mathcal{R}=\{(1+\rho^{\alpha})^k:k\in \Z\}, \]
and the corresponding $(1+\rho^{\alpha})$-adic partitioning of $\R_{>0}$:
\[ \R_{>0}=\bigcup_{r\in \mathcal{R}} R_r, \]
where $R_r=[r,r(1+\rho^{\alpha}))$.  For each $\br=(r_1,\cdots,r_d)\in \mathcal{R}^d$, we denote by $R_{\br}=R_{r_1}\times\cdots\times R_{r_d}$, and
\[ P'_{\br}=\{g\in \supp \overline{\mu}_{n_1}^{*n}: e^{-nn_1\ell_1}\bt_g^{-1}\in R_{\br}\, \text{ for } g=a(\bt_g e^{nn_1\ell_1})u(\ttb_g)  \}. \]
For $\br\in \mathcal{R}^d$, we consider the Borel measure
\[ \nu_{\br}:=\overline{\mu}_{n_1}^{*n}|_{P'_{\br}}*\delta_x\leq \nu, \]
which is also $(1-\kappa,\mathcal{B}_{\rho},\rho^{\eta})$-robust. By \eqref{eq: ued eq-1}, for any $g\in P'_{\br}$, $x\in X$,
\[ \abse{f(a(\br)x)-f(a( e^{-nn_1\ell_1}\bt_g^{-1})x)}\ll_d \log(1+\rho^{\alpha})\cS_{\infty,1}(f)\ll \rho^{\alpha}\cS_{\infty,1}(f).  \]
Assume $m_X(f)=0$, by Lemma \ref{lemma: expanding trajectory and random walk} we have
\begin{align}\label{eq: ued eq-2}
    \abse{\eta_t*\delta_x(f)}&\leq \sum_{\br\in \mathcal{R}^d}\abse{\int_{P'_{\br}} \eta_{t e^{-nn_1\ell_1}\bt_g^{-1}}*\delta_{gx}(f) \dd \mu^{*\Tilde{n}}(g) }\nonumber\\
    &\leq \sum_{\br\in \mathcal{R}^d}\abse{\eta_{t\br}*\nu_{\br}(f) }+\rho^{\alpha}\cS_{\infty,1}(f).
\end{align}
Provided $n_1>\max_{k\in \dbracket{1,d}} \ell_k/\gamma$, for any $n$ we have
\begin{align}\label{eq: inclusion}
    [-nn_k(\ell_k+\gamma),-nn_k(\ell_k-\gamma)]\subset [-nn_1(\ell_1+C_1\gamma), -nn_1(\ell_1-C_1\gamma)],
\end{align}
where $C_1=1+\max_k \ell_1/\ell_k$.
Let 
\[\tilde{\mathcal{R}}=\{\br\in \mathcal{R}^d: \log r_k\in [-nn_1(\ell_1+C_1\gamma), -nn_1(\ell_1-C_1\gamma)], \,\forall k\in \dbracket{1,d}\}.\]
Provided that $n_1>\max_{k\in \dbracket{1,d}} \ell_k/\gamma$, the large deviation principle and \eqref{eq: inclusion} imply  there exists $\gamma_1\lll_{\gamma} 1$ such that 
\begin{align}\label{eq: ued eq-3}
    \overline{\mu}_{n_1}^{*n}(\bigcup_{\br\in \tilde{\mathcal{R}}} P'_{\br} )\geq 1-e^{-nn_1\gamma_1}.
\end{align}
Combining \eqref{eq: ued eq-2}, \eqref{eq: ued eq-3}, we have
\begin{align}
     \abse{\eta_t*\delta_x(f)}\ll \sum_{\br\in \tilde{\mathcal{R}}}\abse{\eta_{t\br}*\nu_{\br}(f) }+(\rho^{\alpha}+e^{-nn_1\gamma_1})\cS_{\infty,1}(f)
\end{align}
It now remains to bound $\abse{\eta_{t\br}*\nu_{\br}(f)}$ for $\br\in \tilde{\mathcal{R}}$. Note that 
\begin{align}\label{eq: rewritten}
    \eta_{t\br}*\nu_{\br}(f)=\eta_{te^{-nn_1\ell_1}}*\nu_{\br}(f\circ a(\br')),
\end{align}
for some $\br'\in [e^{-nn_1C_1\gamma},e^{nn_1C_1\gamma}]^d$.  The choice of $t,n,n_1$ guarantees that $te^{-nn_1\ell_1}\in [\rho^{-1/4},\rho^{-1/2}]$. Hence, by Proposition \ref{prop: spectral gap} and \eqref{eq: rewritten}, we have
\[ \abse{ \eta_{t\br}*\nu_{\br}(f)}\ll (\rho^{\kappa}+\rho^{\eta})\cS_{\infty,l}(f\circ a(\br'))\ll e^{nn_1O(\gamma)}(\rho^{\kappa}+\rho^{\eta})\cS_{\infty,l}(f). \]
Hence, 
\begin{align*}
    \sum_{\br\in \tilde{\mathcal{R}}}\abse{\eta_{t\br}*\nu_{\br}(f) }&\ll \#\tilde{\mathcal{R}}\cdot e^{n_1nO(\gamma)}(\rho^{\kappa}+\rho^{\eta})\cS_{\infty,l}(f)\nonumber\\
    &\ll \left(\frac{nn_1\gamma}{\rho^{\alpha}}\right)^d e^{nn_1O(\gamma)}(\rho^{\kappa}+\rho^{\eta})\cS_{\infty,l}(f).
\end{align*}
We choose $\gamma\lll_{\kappa,\eta} 1$ such that 
\[ e^{nn_1O(\gamma)}(\rho^{\kappa}+\rho^{\eta})\leq\rho^{-2CO(\gamma)}(\rho^{\kappa}+\rho^{\eta})<\rho^{\kappa/2}+\rho^{\eta/2}. \]
Then we choose $\alpha\lll_{\kappa,\eta} 1$ such that 
\[ \rho^{-2\alpha}(\rho^{\kappa/2}+\rho^{\eta/2})<\rho^{\kappa/4}+\rho^{\eta/4}. \]
Finally choose $c\lll_{\alpha,\gamma} 1$ to conclude that 
\begin{align}\label{eq: bound on error when inj large}
    \abse{\eta_t*\delta_x(f)}\ll t^{-c}\inj(x)^{-1}\mathcal{S}_{\infty,l}(f), 
\end{align}
this finishes the proof of the theorem when $|\log\inj x|\leq\tfrac{1}{3/4+4C\ell_1} |\log t|$. When $|\log\inj x|\geq \tfrac{1}{3/4+4C\ell_1} |\log t|$, we may choose $c\lll_C 1$  such that $t^{-c}\inj(x)^{-1}\geq 1$. Then the bound of the error term in \eqref{eq: bound on error when inj large} is trivially true.
\end{proof}

\bigskip

\appendix

\section{Double equidistribution}\label{sec: double equid}

Let $\sigma$ be a product-type self-affine measure on $\R^d$. We say that $\bomega=(\omega_i)_1^d\in (0,1)^d$ is a weighted vector if $\sum_{1}^d \omega_i=1$. Given a weighted vector, we define the one-parameter diagonal flow along $\bomega$ by
\[ a_{\bomega}(t)=\diag(t^{\tfrac{d\omega_1}{d+1}},\cdots,t^{\tfrac{d\omega_d}{d+1}},t^{-\tfrac{d}{d+1}}), \,\, \forall \, t\in \R_{>0}. \]
Given $f_1,f_2\in B^{\infty}_{\infty,l}(X)$, weighted vectors $\bomega,\bv$ and $t_2, t_1>1$, we define the double equidistribution coefficient 
\begin{equation*} \label{double-eq-function}
\Delta^{\bomega,\bv}_{f_1,f_2}(t_1,t_2) := \abse{\int_{\R^d} f_{1}\bigl(a_{\bomega}(t_1)u(\bs)x_{0}\bigr) f_{2} \bigl(a_{\bv}(t_2)u(\bs)x_{0}\bigr) \dd \sigma(\bs) - m_{X}(f_{1})m_{X}(f_{2})}.
\end{equation*}
where $x_0\in X$ is the identity coset. 
\medskip

The following proposition is  essentially \cite[Theorem 1.7]{AG26}, which is a weighted generalization of \cite[Proposition 8.3]{BHZ25}. The main differences here include that we allow weighted vectors $\bomega,\bv$ to vary in the whole $(0,1)^d$, and we relax the assumption in \cite[Theorem 1.7]{AG26}  that $\sigma$ is a product of self-similar measures on $\R$ with equal contraction ratio. To state this proposition, we follow some convenient notations in \cite[Theorem 1.7]{AG26}.

\begin{proposition}\label{prop: double equid}
There exist $C,c>0$ depending only on $\sigma$ such that for any weighted vectors $\bomega,\bv$, all $t_1,t_2>1$ with $t_2\geq t_1$ and $f_1,f_2\in B_{\infty,l}^{\infty}(X)$, we have
\begin{align*}
    \Delta^{\bomega,\bv}_{f_1,f_2}(t_1,t_2)\leq Ct_1^{-c\integer{\bomega}}\cS_{\infty,l}(f_1)\abse{m_X(f_2)}+Ce^{-c\rho((t_1,\bomega),(t_2,\bv)) \integer{\bv}}\cS_{\infty,l}(f_1)\cS_{\infty,l}(f_2),
\end{align*}
where
\[ \rho((t_1,\bomega),(t_2,\bv))=\min_{1\leq i\leq d}\{(v_i+1)\log t_2-(\omega_i+1)\log t_1\}, \]
and $\integer{\bomega}=\min_i \omega_i$, $\integer{\bv}=\min_i v_i$.
\end{proposition}

\medskip

The original idea of the proof of double equidistribution went back to the work of Kleinbock-Shi-Weiss \cite{KSW17} in the context of Lebesgue measure. For fractal measures, this was first done by Khalil-Luethi \cite{KL23}. The proof of Proposition \ref{prop: double equid} follows closely from that of \cite[Proposition 8.3]{BHZ25} and \cite[Theorem 1.7]{AG26} with minor changes. For completeness, we include a proof here.

\begin{proof}
To simplify notations, we write $\cS=\cS_{\infty,l}$ and $\cS(f_1,f_2)=\cS(f_1)\cS(f_2)$. We may assume that $m_X(f_2)=0$. To see this, note that by Theorem \ref{thm: weighted effective equid} we have
\begin{align}\label{eq: assume mX(f_2)=0}
&\abse{ \Delta^{\bomega,\bv}_{f_1,f_2}(t_1,t_2)-\Delta^{\bomega,\bv}_{f_1,f_2-m_X(f_2)}(t_1,t_2)}\nonumber\\
&\leq \abse{\int f_1(a_{\bomega}(t_1)u(\bs)x_0)\dd \sigma(\bs)-m_X(f_1)}\abse{m_X(f_2)}\nonumber\\
&\leq \cS(f_1)t_1^{-c\integer{\bomega}} \abse{m_X(f_2)},
\end{align}
where $c>0$ is provided by Theorem \ref{thm: weighted effective equid}. Hence, it remains to bound
\[ I=\abse{\int  f_1(a_{\bomega}(t_1)u(\bs)x_0)f_2(a_{\bv}(t_2)u(\bs)x_0)\dd \sigma(\bs)  }. \]
Denote 
\[ r=\min_{1\leq i\leq d}\{ (v_i+1)\log t_2-(\omega_i+1)\log t_1 \}. \]
We may assume that $r>0$, since otherwise the proposition follows trivially.  Let $t_3=t_2e^{-r/4}>1$. We have the following relations:
\begin{align}\label{eq: t_1 t_3 inverse}
    t_1^{\omega_k+1}t_3^{-(v_k+1)}\leq e^{-r/2}, \,\, \forall \, 1\leq k\leq d.
\end{align}
Choose parameters $(\eps_1,\cdots,\eps_d)\in (0,1)^d$ determined uniquely by
\begin{align}\label{eq: choice of eps i 2}
     \epsilon_1^{\tfrac{1}{d+1}}\cdots \epsilon_{k-1}^{\tfrac{1}{d+1}}\epsilon_k^{-\tfrac{d}{d+1}}\epsilon_{k+1}^{\tfrac{1}{d+1}}\cdots \epsilon_d^{\tfrac{1}{d+1}}=t_3^{\tfrac{dv_k}{d+1}}, \,\, \text{ for } 1\leq k\leq d,
\end{align}
or equivalently,
\[ \eps_k=t_3^{-\tfrac{d(1+v_k)}{d+1}}, \,\, \text{ for } 1\leq k\leq d. \]
Consider the measure 
\[ \mu=\mu_{1,d}^{\epsilon_1}*\cdots*\mu_{d,d}^{\eps_d} \, \text{ and }\, \lambda=\lambda_{1,d}^{\eps_1}*\cdots *\lambda_{d,d}^{\eps_d}.\]
For any $\phi\in \supp \lambda$, by Lemma \ref{lemma: basic prop of stopping time} (1) we have $\bttr_{\phi}\in \prod_{i=1}^d[r_{min}\eps_i,\eps_i]$, where $r_{min}=\min\{\ttr_{\phi}\in \supp \lambda_i: 1\leq i\leq d\}$. Hence, given any $\phi\in \supp \lambda$ and $\bs \in \supp \sigma$, by \eqref{eq: t_1 t_3 inverse} we have
\begin{align}\label{eq: bound on f1}
    &\abse{ f_1(a_{\bomega}(t_1)u(\phi(\bs))x_0)-f_1(a_{\bomega}(t_1)u(\ttb_{\phi})x_0)}\ll_{\sigma} e^{-r/2}\cS(f_1).
\end{align}
Using the stationarity  $\lambda*\sigma=\sigma$ and \eqref{eq: bound on f1}, we have
\begin{align}\label{eq: bound on I 1}
    I=&\abse{ \int f_1(a_{\bomega}(t_1)u(\ttb_{\phi})x_0) \int_{\R^d} f_2(a_{\bv}(t_2)u(\phi(\bs))x_0)\dd \sigma(\bs)\dd \lambda(\phi) }\nonumber\\
    &+e^{-r/2}\cS(f_1,f_2).
\end{align}
 As $g=a_g a_{\bv}(t_3)u(\ttb_g)$ for some $a_g\in A_{O(1)}$ for any $g\in \supp \mu$, the single effective equidistribution theorem (Theorem \ref{thm: weighted effective equid}) implies that for some $c>0$ and  any $\phi\in \supp \lambda$,
\begin{align}\label{eq: approximate int f_2}
    \int_{\R^d} f_2(a_{\bv}(t_2)u(\phi(\bs))x_0)\dd \sigma(\bs)&=\int_{\R^d} f_2(a^{-1}_{g_{\phi}}a_{\bv}(t_2t_3^{-1})u(\bs)g_{\phi}x_0)\dd \sigma(\bs)\nonumber\\
    &=O(e^{-cr\integer{\bv}/4}\inj(g_{\phi}x_0)^{-1})\cS(f_2).
\end{align}
Let 
\[ E=\{x: \inj(x)\leq e^{-cr\integer{\bv}/8}\}. \]
Condition \eqref{eq: choice of eps i 2} guarantees \eqref{eq: condition on epsilon i} holds for some permutation $S$ on $\dbracket{1,d}$. Then Theorem \ref{thm: weighted effective recurrence} implies that for some $c_1>0$,
\begin{align}\label{eq: bound on E}
    \mu*\delta_{x_0}(E)\ll e^{-c_1cr\integer{\bv}/8} .
\end{align}
Combining \eqref{eq: bound on I 1}, \eqref{eq: approximate int f_2} and \eqref{eq: bound on E}, we have
\begin{align}\label{eq: bound on I 2}
    I=O\left( e^{-r/2}+e^{-cr\integer{\bv}/8}+e^{-c_1cr\integer{\bv}/8} \right)\cS(f_1,f_2)=O(e^{-c_2r \integer{\bv}})\cS(f_1,f_2),
\end{align}
for some $c_2\lll_{c,c_1} 1$. The proof can be concluded by combining \eqref{eq: assume mX(f_2)=0} and \eqref{eq: bound on I 2}.
\end{proof}

\medskip

\bibliographystyle{abbrv} 
\bibliography{khintchine}

\end{document}